\documentclass[11pt]{article}
\usepackage{color}
\usepackage{graphicx}
\usepackage{amsmath}
\usepackage{amssymb}
\usepackage{mathrsfs}
\usepackage[numbers,sort&compress]{natbib}
\usepackage{amsthm}
\numberwithin{equation}{section}
\usepackage{pgf}
\usepackage{graphicx}
\usepackage{tikz}
\usepackage{stmaryrd}
\usepackage[latin1]{inputenc}
\usepackage[T1]{fontenc}
\usepackage[english]{babel}
\usepackage{listings}
\usepackage{xcolor,mathrsfs,url}
\usepackage{amssymb}
\usepackage{amsmath}
\usepackage{ifthen}
\newtheorem{theorem}{Theorem}
\newtheorem{lemma}{Lemma}
\newtheorem{remark}{Remark}
\newtheorem{corollary}{Corollary}
\newtheorem{proposition}{Proposition}

\usepackage {caption}
\usepackage{xcolor}
\usepackage{cases}

\usepackage{amsmath}
\usepackage{amsthm}

\theoremstyle{definition}
\newtheorem{rhp}{RH Problem}[section]

\usepackage{todonotes}
\usepackage{float}
\usepackage{subfigure}
\usepackage{hyperref}
\hypersetup{
    colorlinks=true, 
    linktoc=all,     
    linkcolor=blue,  
    citecolor=blue
}

\begin{document}

\title{A Riemann-Hilbert approach to  the existence of global solutions to the Fokas-Lenells equation on the  line }
\author{Qiaoyuan Cheng$^1$,  Engui Fan$^1$\thanks{\ Corresponding author and email address: faneg@fudan.edu.cn } }
\footnotetext[1]{ \  School of Mathematical Sciences  and Key Laboratory of Mathematics for Nonlinear Science, Fudan University, Shanghai 200433, P.R.
China.}

\date{}

\maketitle
\begin{abstract}
\baselineskip=16pt

We  obtain  the existence  of global solutions to the Cauchy problem of the   Fokas-Lenells (FL)  equation on the line
\begin{align}
&u_{xt}+\alpha\beta^2u-2i\alpha\beta u_x-\alpha u_{xx}-i\alpha\beta^2|u|^2u_x=0,\nonumber \\
&u(x,t=0)=u_0(x), \nonumber
\end{align}
without the small-norm assumption on initial data $u_0(x)\in H^3(\mathbb{R})\cap H^{2,1}(\mathbb{R})$.
  Our  main technical  tool is the inverse scattering  transform method
  based on the representation of a Riemann-Hilbert (RH) problem associated with the above Cauchy problem.
The existence and the uniqueness of the RH problem is shown via a
 general vanishing lemma.  The spectral problem associated with the FL  equation  is changed  into an equivalent
    Zakharov-Shabat-type spectral problem to establish  the RH problems on the real axis.
 By representing the solutions of the RH problem via the Cauchy integral protection and the reflection coefficients, the reconstruction formula is  used  to  obtain  a unique  local solution  of the FL equation.
 Further, the eigenfunctions and the reflection coefficients are shown
 Lipschitz continuous with respect to initial data,  which  provides a prior estimate of the solution to the FL equation.
   Based on the local solution and the uniformly  prior  estimate, we construct a  unique  global
solution in $H^3(\mathbb{R})\cap H^{2,1}(\mathbb{R})$  to the  FL equation.\\[3pt]
\end{abstract}
\baselineskip=18pt

\tableofcontents
\section{Introduction}
\hspace*{\parindent}
In this paper, we study the existence of global solutions to the Cauchy problem of the Fokas-Lenells (FL) equation
\begin{align}
&u_{tx}+\alpha\beta^2u-2i\alpha\beta u_x-\alpha u_{xx}+\sigma i\alpha\beta^2|u|^2u_x=0 \label{cs} \\
&u(x,t)|_{t=0} = u_0(x),\label{cs1}
\end{align}
The FL equation is an integrable generalization of the nonlinear Schr\"{o}dinger (NLS) equation, which is also tightly related with the derivative NLS model. Analogous to the derivation of the Camassa-Holm equation from the KdV equaiton  \cite{Fuchssteiner}, the FL equation was initially derived from the NLS equation by utilizing two Hamiltonian operators \cite{Fokas1}. In optics, the FL equation characterizes the higher-order linear and nonlinear optical effects, and has been derived as a suitable model for the femtosecond pulse propagation through single mode optical silica fiber, for which several interesting solutions have been constructed \cite{Hos1,Hos2,Hos3}.

Prior studies on the FL equation mainly concentrate on constructing the soliton solutions and deriving the long time asymptotics. Lenells and Fokas are the first to discover the soliton solutions to the FL equation via the inverse scattering transform and the dressing methods \cite{oan36, dfa37}.
 Matsuno obtained the bright and the dark soliton solutions to the  FL equation via the Hirota method \cite{adm38,adm39}. Later,
Liu et al. derived the double Wronskian solutions to the FL  equation via a bilinear approach \cite{zhang2021}, Vekslerchiket et al. constructed the lattice representation and the n-dark solitons of the FL equation in \cite{lra40}, while Wright et al. found the breather solutions of the FL equation via a dressing-B\"{a}cklund transformation related to the Riemann-Hilbert (RH) problem formulation \cite{shc41}.
A kind of rogue wave solution to the FL equation were obtained by using the Darboux transformation  \cite{rwo34}.  The Fokas method was used to investigate the initial-boundary value problem
for the FL equation  on the half-line and a finite interval, respectively \cite{Lenells2009,xiaofan}.
An algebro-geometric method was used to obtain algebro-geometric solutions for the  FL equation \cite{zhao2013}. The explicit one-soliton of  the  initial value problem for the FL equation
was given by the RH method \cite{xujian2}.
The inverse scattering transform for the FL
equation with nonzero boundary conditions was further investigated by using the RH method \cite{zf1}.
Recently the   perturbation theory was also used  to obtain the exact solution to the Fokas-Lenells equation \cite{Lashkin}.
On long time asymptotics, our recent works respectively characterized the long time behaviors of defocusing and focusing FL
equations \cite{Lta,cqy}.  However, little is known on the existence of local and global solutions to the FL equation.
In recent years, the  global well-posedness of the periodic initial value problem for the FL equation in a Sobolev space  was proved by  Fokas and  Himonas  \cite{wpo10}.
But to the best of our knowledge, the existence of local and global solutions to the FL equation on the line is still unknown.


In this paper,  we would like to prove the existence of   global solutions to the FL equation on the line
from the perspective of inverse scattering transform.
 This technique has been applied  to   prove  the existence of global solutions to the derivative NLS  equation
 by Pelinovsky and Shimabukuro  \cite{dep}.
 By establishing the Lipschitz continuity of the eigenfunctions and the scattering data,
 we obtain the Lipschitz continuous mapping between the initial data   and the solution to the FL equation.
In this way  we  are able  to establish the existence of local and global solutions to the FL equation.
Compared with the derivative NLS  equation, we find that   extension of  this  approach  to the FL equation will  confront some substantial difficulties,
 which are also different from the derivative NLS  equation  \cite{dep}.
\begin{itemize}
    \item[$\triangleright$]  {\bf Spectral singularity $k=0$}:  For the FL equation, its   Lax pair   involves  a singularity at
     $k=0$, which is however the key to control the non-blow-up of the reflection coefficients during the time evolution analysis.
     So we need  additional  requirements  to   handle the
   Jost functions, scattering data and their    norm  estimates.  The singularity $k=0$ does not meet  the premise in the  Zhou's technique   \cite[Theorem 1.8]{lss}.
   Therefore, existing techniques have to be strengthened for analyzing the more complicated FL equation.
We  overcome this  technical  difficulty of the FL equation by drawing the equivalence between the Jost function at $k=\infty$ and at $k=0$.
 We obtain a key asymptotic estimate on scattering data,
which sufficiently cover the singularity of  reflection coefficient $r(k)$ at $k=0$  in its time evolution. This allows us to extend the usage of the Beals-Coifman theory and  Zhou's technique   to the case of $k=0$.

  \item[$\triangleright$]    {\bf Two kinds of reconstruction formulas: }  For the FL equation,   we are only able to    recover  its   potential $u_x(x,t)$
   from  the limit of the  Jost function  as  $k \rightarrow \infty $.
 We should resort to the  time spectral problem (\ref{cslp2}) to obtain the original potential
$u(x,t)$   from the asymptotic   expansion of the Jost function    as  $k \rightarrow 0$.
 From these results, we are able to derive the reconstruction formula for    $u(x,t)$ as  $k\to 0$.
 While we obtain other reconstruction formula for    $u_x(x,t)$  as  $ z\rightarrow \infty $
 in new $z$-plane. These two kinds of  reconstruction formulas  are combined together to obtain the estimates on the solution to the FL equation.

        \item[$\triangleright$]   We transform  original spectral problem and RH problem  from the $k$ plane to the $z$ plane
         by exploiting  the  odd-even  property    of the Jost functions  in the column.
         This facilitates the further application of  the Fourier transform  and Cauchy integral projection
         to estimate the space of the solution to the RH problem and the potential.

\end{itemize}

The structure of the paper is as follows. In Section \ref{sec:section2}, we carry out inverse scattering transform on the $k$-plane.
 First, based on   the Lax pair of the FL equation (\ref{cs}),
   we construct the Jost functions $\psi(x;k)$  and  analyze their  asymptotics    of at two spectral singularities $k=0$ and $k=\infty$.
 Further  we construct a  RH problem $N(x;k)$ associated with the Cauchy problem (\ref{cs})-(\ref{cs1})  and then prove its  existence and the uniqueness  via a general vanishing lemma.
  In Section \ref{sec:section3}, we carry out inverse scattering transform on the $z$-plane.  We  impose a transformation  on $\psi(x;k)$ to obtain Jost function  $\Psi(x;z)$ and  a new RH
  problem  for  $M(x;z)$  on the $z$-plane. We further  establish  the Lipschitz continuous mapping from the initial data to the reflection coefficient $r_{1,2}(z)$.
In Section \ref{sec:section4},  to   estimate  the  solutions to the RH problem  $M(x;z)$,  we make transitions of the RH problem of  $M(x;z)$ to the RH problem of  $Q_{1,2}(x;k)$
    which allows us to use the existing properties of the scattering data to derive the estimates on the Beals-Coifman solutions to the RH problem of  $M(x;z)$.  In Section \ref{sec:section5},  based on the  reconstruction formulae and the Cauchy integral, subsequently
 we obtain estimates on  the potential $u(x)$  for  the FL equation  on the positive and negative half-lines respectively.
 In Section
\ref{sec:section6},   based on  the   local solution  and the uniformly  priori  estimates,
   we   show that   there exists  a   global  solution   $u(x,t)\in C([0,\infty), H^3(\mathbb{R})\cup H^{2,1}(\mathbb{R})$ to
    the Cauchy problem (\ref{cs})-(\ref{cs1}) of  the FL equation.

\section{Inverse scattering transform  on $k$-plane   }
\label{sec:section2}
\hspace*{\parindent}
In this section, we establish a  RH problem associated with the Cauchy problem  (\ref{cs})-(\ref{cs1}) of the  FL equation and show its existence and uniqueness.

  We first fix some notations used this paper.
If $I$ is an interval on the real line $\mathbb{R}$ and $X$ is a  Banach space, then $C (I,X)$ denotes the space of continuous functions on $I$ taking values in $X$. It is equipped with the norm
\begin{equation*}
	\|u\|_{C (I, X)}=\sup _{x \in I}\|u(x)\|_{X}.
\end{equation*}
 We define weighted Sobolev space by
\begin{align}
&L^{p,s}(\mathbb{R}):=\{ u(x)\in L^{p}(\mathbb{R}): \ \   \langle x \rangle^s u(x) \in L^{p}(\mathbb{R})  \}, \nonumber\\
&H^{k,s}(\mathbb{R})=\left\{u(x) \in L^{2,s}(\mathbb{R}): \ \  \partial^j_xu (x) \in L^{2,s}(\mathbb{R}),\quad j=1,\cdots,k \right\},\nonumber\\
&\mathcal{W}(\mathbb{R})=\{r (z)\in H^{1}(\mathbb{R}) \cap L^{2,1}(\mathbb{R}), z^{-2}r (z) \in L^2(\mathbb{R})\},\nonumber
\end{align}
where $ \langle x \rangle := \sqrt{1+|x|^2}$.

\subsection{Jost functions}
\hspace*{\parindent}
The FL equation (\ref{cs}) admits a  Lax pair
\begin{align}
&\phi_x+ik^2\sigma_3\phi=kP_x\phi,\label{cslp}\\
&\phi_t+i\eta^2\sigma_3\phi=H\phi,\label{cslp2}
\end{align}
where
\begin{align}
& \sigma_3=\begin{pmatrix}
1&0\\
0&-1
\end{pmatrix}, \ \ \ \  P=\begin{pmatrix}
0&u\\
-\bar{u}&0
\end{pmatrix},\label{qq}\\
&
\eta=\sqrt{\alpha}\left(k-\frac{\beta}{2k}\right),\quad H= \alpha kU_x+\frac{i\alpha\beta^2}{2}\sigma_3\left(\frac{1}{k}P-P^2\right). \nonumber
\end{align}
The FL spectral problem  (\ref{cslp}) is not  the Zakharov-Shabat type  spectral
  problem due to   the  multiplication  by $k$ in the matrix potential $kP_x$ and  the time part (\ref{cslp2})  admits spectral singularity at   $k=0$.
  To control the behavior of the eigenfunction $\psi(k)$   as $k\to \infty$ and $k\to 0$ in constructing the solution $u(x,t)$ to the FL equation $(\ref{cs})$,  we consider two different asymptotic expansions respectively at the singularities $k=0$ and $k=\infty$.
 Here, we recall the existing results on  constructing the  RH problem in \cite{xujian2,cqy}.

\noindent \textbf{Case I: $k=\infty$}

By making  a transformation
\begin{equation}
\psi(x,t;k)=\phi(x,t;k)e^{i(k^2x+\eta^2t)\sigma_3},\label{b3}
\end{equation}
the Lax pair (\ref{cslp})-(\ref{cslp2}) is changed to
\begin{align}
&\psi_x+ik^2[\sigma_3,\psi]=kP_x\psi,\label{laxn1}\\
&\psi_t+i\eta^2[\sigma_3,\psi]=H\psi.\label{laxn2}
\end{align}
This  Lax pair admits   the Jost functions with asymptotics
$$\psi^{\pm}(x,t;k)\sim I, \quad  x\rightarrow \pm \infty,$$
which  satisfy    Volterra  integral equations
\begin{equation}
\psi^{\pm}(x,t;k)=I+k\int_{ \pm \infty }^{ x }e^{-2ik^2(x-y)\widehat{\sigma}_3}
P_y(y)\psi^{\pm}(y,t;k) dy.\label{fi5}
\end{equation}
It is obvious that the    integration  in (\ref{fi5})    involves   the  term  $k u_{ x}(x)$  which  is not $L^2(\mathbb{R})$  bounded   since    $k$   may  go  to  infinity.
In our previous work, through the small $k$ and large $k$ estimates respectively,  we overcome this difficulty and prove  the existence and the differentiability of the Jost functions
$\psi^{\pm}(x,t;k)$ \cite{cqy}.

\begin{proposition} \label{proposition1}
Let $u_0(x)\in H^3(\mathbb{R})\cap H^{2,1}(\mathbb{R})$, and introduce the notation
$\psi^{\pm}(x,t;k)=\left(\psi^{\pm }_1 (x,t;k),\psi^{\pm }_2(x,t;k)\right) $
with the scripts  $1$ and $2$ denoting the first and second columns of $\psi^{\pm}(x,t;k)$. Then we have
\begin{itemize}

\item[$\blacktriangleright$] {\bf Analyticity}:   The integral equation (\ref{fi5})  admits  a    unique  solution
  $\psi^{\pm}(x,t;k)$.  Moreover, $ \psi^-_1(x,t;k)$, $ \psi^+_2(x,t;k)$ and $a(k)$  are analytical in the domain $ D_+;$  and   $ \psi^+_1(x,t;k)$ and $\psi^-_2(x,t;k)$ are analytical  in
  in the domain $D_-$,
  where  (see Figure \ref{figure1})
\begin{equation}
D_+=\{ k:  {\rm Im} k^2>0\}, \  \ D_-=\{ k:  {\rm Im} k^2<0\};
\end{equation}
  \begin{figure}[H]
\begin{center}
\begin{tikzpicture}
\draw [->](-2.5,0)--(2.5,0);
\draw [->](0,-1.8)--(0,1.8);
 \node at (1,1)  {$D_+$};
\node at (3, 0)  {$\text{Re }k$};
\node at (0, 2.1)  {$\text{Im }k$};
\node at (0.2, -0.2)  { $0$};
 \node at (-1, -1)  {$D_+$};
\node at (1, -1)  {$D_-$};
 \node at (-1, 1)  {$D_-$};

\end{tikzpicture}
\end{center}
\caption{The analytical domains $D_+$ and  $D_-$ for $\psi^{\pm}(x,t;k)$.}
\label{figure1}
 \end{figure}
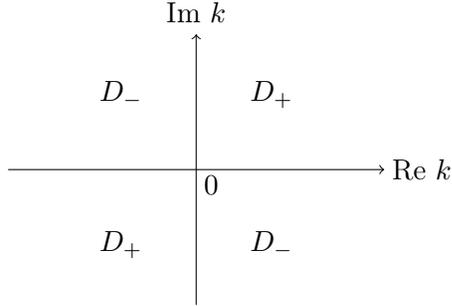
\item[$\blacktriangleright$] {\bf Symmetry}:     $ \psi^{\pm}(x,t;k)$ and $S(k)$  satisfy   the symmetry relations
\begin{align}
&\psi^{\pm}(x,t;k) =\sigma_2 \overline{\psi^{\pm}(x,t;\bar {k} )} \sigma_2,\ \ \ \psi^{\pm}(x,t;k) =\sigma_3  \psi^{\pm}(x,t;-k) \sigma_3,\label{dcx}
\end{align}
\item[$\blacktriangleright$] {\bf Asymptotics}:   $ \psi^{\pm}(x,t;k)$ and $S(k)$ have asymptotic properties
\begin{align}
&\psi^{\pm}(x,t;k)=\mathrm{e}^{ -ic_\pm(x)\sigma_3}+\mathcal{O}(k^{-1}),\quad k\rightarrow\infty,\label{jjw}
\end{align}
where
\begin{equation}
c_\pm(x)=\frac{1}{2}\int_{\pm\infty}^x|u_y(y,t)|^2dy. \label{c1c2}
\end{equation}
\end{itemize}
\end{proposition}

It is noteworthy that $\psi^\pm(x,t;k)$ do not approach  to the identity matrix as $k\rightarrow\infty$. To normalize the first term in the
asymptotics, we define the second new function $\mu^\pm(x,t;k)$ which satisfies the following relation
\begin{equation}
\psi^\pm(x,t;k)=e^{-ic_-(x)\widehat{\sigma}_3}\mu^\pm(x,t;k)e^{ic_+(x)\sigma_3},\label{psi}
\end{equation}
then  $\mu^{\pm}(x,t;k)$ satisfy the following Volterra-type integral
\begin{equation}
\mu^{\pm}(x,t;k)=I+\int_{\pm\infty}^{x} \mathrm{e}^{-\mathrm{i} k^2\left(x-y \right) \widehat{ \sigma}_3}\left(P_1 \mu^{\pm}\right)\left(y,t;k\right) dy, \label{gtit}
\end{equation}
where $P_1 $ is  given   by
\begin{align}
&P_1=e^{\frac{i}{2} \int_{-\infty}^x|u_y(y, t)|^2 dy \widehat{\sigma}_{3}}(k P_x+\frac{i}{2} |u_x|^2   \sigma_{3}).\nonumber
\end{align}

The potential can be recovered  from the limit of the  Jost function $\psi(x,t;k)$
\begin{equation}
u_x(x,t)= 2i m(x,t) e^{4i \int_{-\infty}^x |m(s,t)|^2 ds},  \label{rjjw}
\end{equation}
where
\begin{equation}
m(x,t)=\lim_{k \rightarrow \infty} (k\psi (x,t;k))_{12}.\nonumber
\end{equation}
The formula (\ref{rjjw}) shows that we only  recover the  potential $u_x(x,t)$
  from  the  Jost function as $k \rightarrow \infty $.
 We should resort to the  time spectral problem (\ref{cslp2}) to obtain the original potential
$u(x,t)$   from the asymptotic   expansion of the Jost function  $\psi(x,t,k)$    as  $k \rightarrow 0$.

\noindent \textbf{Case II: $k=0$}

We consider the transformation
\begin{equation}
\varphi(x,t,k)=\phi(x,t;k)e^{i(k^2x+\eta^2t)\sigma_3}, \label{wer}
\end{equation}
then the Lax pair (\ref{cslp})-(\ref{cslp2})  changes to
\begin{align}
&\varphi_x+ik^2[\sigma_3,\varphi]=kP_x\varphi,\label{lax5}\\
&\varphi_t+i\eta^2[\sigma_3,\varphi]= H\varphi,\label{lax6}
\end{align}
 which is
 integrated along $(\pm \infty,t) \rightarrow (x,t)$ and leads to two Volterra-type integrals
\begin{equation}
\varphi^{\pm}(x,t;k)=I+k\int_{ \pm \infty }^{ x }e^{-2ik^2(x-y)\widehat{\sigma}_3}
P_y(y)\varphi^{\pm}(y,t;k) dy. \label{fii}
\end{equation}
We seek the asymptotic expansion of  $\varphi$ in the Lax pair (\ref{lax5})-(\ref{lax6}) as $k \rightarrow 0$
and find that
\begin{equation}
\varphi(x,t;k)=I+k\begin{pmatrix}
0 & u \\
v & 0
\end{pmatrix}+\mathcal{O}\left(k^2\right), \quad k \rightarrow 0,\label{vxk}
\end{equation}
which implies that the solution $u(x,t)$ to the FL equation can be reconstructed by
\begin{equation}
u(x,t)=\lim _{k\rightarrow 0}  ( k^{-1} \varphi (x,t;k))_{12}.
\end{equation}
By establishing the RH problem, we strive to get a reconstruction formula for $u(x,t)$.

From  (\ref{psi}) and (\ref{wer}), we observe that  the   $\mu^\pm(x,t;k)$, $\psi^\pm(x,t;k)$ and $\varphi^\pm(x,t;k)$ are
related to the same Lax pair (\ref{cslp})-(\ref{cslp2})  and  therefore satisfy the relation
\begin{align}
&\psi^\pm(x,t;k)=e^{-ic_- \sigma_{3}} \mu^\pm(x,t;k) e^{ic\sigma_3},\nonumber\\
&\mu^\pm(x,t;k)=e^{-ic_- \sigma_{3}} \varphi^\pm(x,t;k) e^{-i\left(k^2x+\eta^2 t\right) \sigma_3} C^\pm(k) e^{i\left(k^2 x+\eta^2t\right) \sigma_3} e^{ic\sigma_3}.\label{pvgx}
\end{align}
By letting $x \rightarrow \pm \infty$  respectively,  we find
\begin{equation}
C^-(k)=e^{- i c \sigma_3}, \quad C^+(k)={I}.\nonumber
\end{equation}

From (\ref{vxk}) and (\ref{pvgx}), the solution $u(x,t)$ to the FL equation can be expressed as
\begin{align}
&u(x,t)e^{2i(c_-(x)+c)}=\lim _{k\rightarrow 0}(k^{-1}\psi^+(x,t;k))_{12},\label{u1}\\
&u(x,t)e^{-i(2c_-(x)+c)}=\lim _{k\rightarrow 0}(k^{-1}\psi^-(x,t;k))_{12}.\label{u2}
\end{align}

\begin{remark}
In particular, we use the relation between $u(x,t)$ and $\psi(x,t;k)$ instead of $u(x,t)$ and $\mu(x,t;k)$. The benefits of this approach will be seen later in the transformation $A(x;k)$, where the linear  spectral problem  (\ref{cslp}) is in turn transformed to a spectral problem of the Zakharov-Shabat type.
\end{remark}

In the following sections, we first consider   the partial spectral problem  (\ref{cslp})  with $t$ being  a parameter,  so we  omit the variable $t$ as usual, for example $\psi(x,t;k)$ is just written as $\psi(x;k)$.  We will discuss the  time evolution of scattering data and  reconstruct the  potential $u(x,t)$ with $t$  in the Section \ref{sec:section6}.

\subsection{A  basic RH problem }
\hspace*{\parindent}
Since  $\psi^{\pm}(x,t;k)$ are two solutions to the spectral problem (\ref{cslp}),
 they are   linearly dependent and satisfy the scattering relation
\begin{equation}
\psi^-(x,t;k)=\psi^+(x,t;k)e^{-ik^2x\widehat{\sigma}_3}S(k),\label{phstgx}
\end{equation}
where $S(k)$ is a scattering matrix given by
\begin{equation}
S(k)=\begin{pmatrix}
a(k) & b(k) \\[3pt]
-\overline{b(\bar{k})} & \overline{a(\bar{k})}
\end{pmatrix}, \ \ \det S(k)=1.\label{xxgx}
\end{equation}
By  Proposition \ref{proposition1} and (\ref{phstgx}),
it is easy to show that $S(k)$ admits the following  symmetries
\begin{align}
&S(k) =\sigma_2 \overline{S( \bar {k} )} \sigma_2,\ \ \ S(k) =\sigma_3  S(-k) \sigma_3  \label{dcx3}
\end{align}
 and asymptotics
\begin{align}
&S(k)=I+\mathcal{O}(k^{-1}),\quad k\rightarrow\infty.
\end{align}

By using the relation  $\det\left(\psi_1^\pm,\psi_2^\pm\right)=1$ and  the scattering relation (\ref{phstgx}), the scattering data $a(k)$ and $b(k)$  can be expressed in term of    determinant
\begin{align}
&a(k)  =\det \left(\psi^-_1(0;k), \psi^+_2(0;k)\right),\label{ak} \\
&b(k) =\det\left(\psi_1^+(0;k), \psi_1^-(0; k)\right),\label{bk}
\end{align}
and we can show that

\begin{proposition}  The scattering data $a(k)$ and $b(k)$ are even and odd functions
  respectively and admit  the following  properties
\begin{itemize}

\item[$\blacktriangleright$] { Symmetries}:
\begin{equation}
a(-k)=a(k), \quad k \in \overline{D}_+, \quad b(-k)=-b(k), \quad \operatorname{Im} k^{2}=0.\label{abdc}
\end{equation}

\item[$\blacktriangleright$] {  Asymptotics}:
\begin{align}
&a(k)=  e^{-ic} +\mathcal{O}\left(k^{-1}\right),\quad    \ \
 b(k)=\mathcal{O}\left(k^{-1}\right), \quad k \rightarrow \infty, \nonumber\\
 &a(k)=e^{-ic}\left(1+\mathcal{O}\left(k^2\right)\right),\quad    \ \
 b(k)=\mathcal{O}\left(k^3\right),\quad k\rightarrow 0,\label{aj0}
\end{align}
where
\begin{equation}
c=c_-(x) - c_+(x) = \frac{1}{2}\int_{-\infty}^\infty |u_x |^2dx,\nonumber
\end{equation}
and $c_\pm (x)$ are  defined by  (\ref{c1c2}).

\end{itemize}
\end{proposition}

In  the  analysis of   a  RH problem by using inverse scattering transform, to avoid  techquetical difficulty that  the  singularities give rise to, in  general the  condition that scattering data   $a(k)$   admits  no   eigenvalues or resonances is requested.

We denote $\mathcal{G}$ as a set of initial value $u_0(x)$ such that
\begin{align}
	\mathcal{G}=\left\lbrace u_0(x): \ u_0(x)  \in H^{3}(\mathbb{R} )\cap H^{2,1}(\mathbb{R} ),\ a(k) \neq 0\text{ in } \mathbb{C}^-\right\rbrace,\label{asum}
\end{align}
whose  reasonableness     will be  shown  by subsequent  Proposition \ref{prow}.

For $ u_0(x)  \in \mathcal{G}$,  we  define the reflection coefficient
\begin{equation}
r(k):=\frac{b(k)}{a(k)}, \quad k \in \mathbb{R} \cup i\mathbb{R},
\end{equation}
 and a matrix  function
	   \begin{equation}
	   	N(x;k):=\begin{cases}
	   		e^{ic_+(x)\sigma_3}\left(\frac{\psi_1^-(x;k)}{a(k)}, \psi_2^+(x;k)\right),  \; k\in D^+,\\[5pt]
	   		e^{ic_+(x)\sigma_3}\left(\psi_1^+(x;k),
\frac{\psi_2^-(x;k)}{\overline{a(\bar{k})}}\right), \; k\in D^-,
	   	\end{cases}
	   \end{equation}
which then  satisfy the following  basic  RH problem

\begin{rhp} 
\label{rhp:21}
Find a  matrix  function $N(x;k)$  with the following properties:
\begin{itemize}

\item[$\blacktriangleright$]  \emph{Analyticity}: $ N(x;k)$ is analytical   in $ \mathbb{C}\setminus \mathbb{R}\cup i\mathbb{R} $.

\item[$\blacktriangleright$] \emph{Jump condition}: $N(x;k)$   satisfies the jump condition
\begin{equation}
N_+(x;k) =N_-(x;k) V(x;k),\label{pr}
\end{equation}
where $V(x;k)=I+J(x;k)$  is  the jump matrix with
\begin{equation}
J(x;k)=\begin{cases}
\begin{pmatrix}
|r(k)|^2 & \overline{r(k)} \mathrm{e}^{-2ik^2 x}\\
r(k) e^{2ik^2 x} & 0
\end{pmatrix},\quad k \in \mathbb{R},\\ \\
\begin{pmatrix}
-|r(k)|^2 & -\overline{r(k)} e^{-2ik^2x} \\
r(k) e^{2ik^2x} & 0
\end{pmatrix}, \quad k \in i\mathbb{R}.  \label{V11}
\end{cases}
\end{equation}

\item[$\blacktriangleright$] \emph{Asymptotic conditions}:
\begin{equation}
N(x ; k) \rightarrow I
\quad \text { as } \quad|k| \rightarrow \infty.
\end{equation}
\end{itemize}
\end{rhp}
From (\ref{V11}), it is easy to find  that   the matrix $J(x;k)$ is Hermitian for $k\in \mathbb{R}$ and the matrix $J(x;k)$ is not Hermitian for $k\in i \mathbb{R}$.

\subsection{Solvability  of the   RH problem }
\hspace*{\parindent}
 For a given  function $h(z) \in L^p(\mathbb{R})$ with $1 \leq p<\infty$, the Cauchy operator is  defined  by
\begin{equation}
\mathcal{C}(h)(z):=\frac{1}{2\pi i} \int_{\mathbb{R}} \frac{h(s)}{s-z} d s, \quad z \in \mathbb{C} \backslash \mathbb{R}.
\end{equation}
 When $z \pm i\varepsilon$
approaches to a point on the real axis $z \in \mathbb{R}$ transversely from the upper and the lower half planes,   the Cauchy operator
$\mathcal{C}$ becomes the Plemelj projection operators defined  respectively by
\begin{equation}
\mathcal{P}^{\pm}(h)(z):=\lim _{\varepsilon \downarrow 0} \frac{1}{2 \pi i} \int_{\mathbb{R}} \frac{h(s)}{s-(z \pm   i\varepsilon)} d s \quad z \in
\mathbb{R}.
\end{equation}
We list the basic properties of the Cauchy and the Plemelj projection operators in the following proposition  \cite{dep,ams,dp}.
\begin{proposition}
\label{p3}
For every $h \in L^{p}(\mathbb{R}), 1 \leq p<\infty$, the Cauchy operator $\mathcal{C}(h)$ and the projection operator $\mathcal{P}^{\pm}(h)$ has the following properties:
\begin{itemize}

\item[$\blacktriangleright$]  $\mathcal{C}(h)$
 is analytic in $\mathbb{C}^\pm$ and goes to zero as $|z|\rightarrow \infty$.

\item[$\blacktriangleright$]
If $h \in L^1(\mathbb{R})$,  then in   $\mathbb{C}^+$or $\mathbb{C}^-$,   the Cauchy operator admits the following asymptotic
\begin{equation}
\lim _{|z| \rightarrow \infty} z \mathcal{C}(h)(z)=-\frac{1}{2 \pi i} \int_{\mathbb{R}} h(s)ds.\label{lch}
\end{equation}
\end{itemize}
And the projection operator $\mathcal{P}^{\pm}(h)$ has the following properties:
\begin{itemize}
\item[$\blacktriangleright$]  $\mathcal{C}(h)$ approaches to $\mathcal{P}^{\pm}(h)$ almost everywhere, when a point $z \in \mathbb{C}^{\pm}$approaches to a point $z_0\in \mathbb{R}$ by any non-tangential contour from $\mathbb{C}^{\pm}$.

\item[$\blacktriangleright$]
For  $1<p<\infty$, there exists a positive constant $c$ such that
\begin{equation}
\left\|\mathcal{P}^{\pm}(h)\right\|_{L^p(\mathbb{R})} \leq c\|h\|_{L^p(\mathbb{R})}.\label{pjx}
\end{equation}
\end{itemize}
\end{proposition}

Now, we begin to analyse  the jump matrix $V(x;k)$ defined by the reflection coefficient $r(k)$. By using (\ref{xxgx}) and  (\ref{dcx3}),   we can derive the following   relation
\begin{align}
 &|a(k)|^2+|b(k)|^2=1, \quad  k \in \mathbb{R}, \\
 &|a(k)|^2-|b(k)|^2=1, \quad  k \in i\mathbb{R}.
\end{align}
 If $a(k)$ has no singularity on $\mathbb{R}\cap  i\mathbb{R} $,
  we obtain the following proposition which ensures if $r(k)$ is bounded and satisfies
\begin{equation}
1-|r(k)|^2=\frac{1}{|a(k)|^2} \geq c_{0}^2>0 \quad k \in i \mathbb{R},\label{1_rk}
\end{equation}
where $c_{0}^{-1}:=\sup _{k\in\mathbb{R}}|a(k)|.$
By using (\ref{V11}), direct calculation shows that
 \begin{align}
 &\frac{1}{2}(V +  V^H)  = \left(\begin{array}{cc}1+ |r(k)|^2   & \overline{r(k)}e^{-2ik^2x} \\
r(k)e^{2ik^2x}  & 1 \end{array} \right), \ \ \ k\in \mathbb{R}, \label{sc9at9}\\
&\frac{1}{2}\left(V +  V^H\right) = \left(\begin{array}{cc} 1-|r(k)|^2   & 0 \\
0 & 1 \end{array} \right), \ \ \ k\in i\mathbb{R}, \label{sc9at10}
\end{align}
 which implies that  the jump  matrix $V$ has  strictly positive real part on both $\mathbb{R}$ and $i\mathbb{R}$.
This property allows us to   show  the following result as in      \cite{dep}.

\begin{proposition}
\label{p4}
For every $r(k) \in L_z^{\infty}(\mathbb{R})$ with $z=k^2$ satisfying (\ref{1_rk}), it can be proved that for every
$x\in\mathbb{R}$ and every column-vector $s \in \mathbb{C}^2$, we have
\begin{equation}
\operatorname{Re} (s^HVs) \geq \alpha_-   s^Hs, \quad k \in \mathbb{R} \cup i\mathbb{R},\label{regt}
\end{equation}
and
\begin{equation}
\|Vs\| \leq \alpha_+ \|s\|, \quad k \in \mathbb{R} \cup i\mathbb{R}.\label{ig}
\end{equation}
where $\alpha_- $ and $\alpha_+ $ are positive constants.
\end{proposition}

Next  we show the solvability of the RH problem \ref{rhp:21}
via the  solvability of a  Fredholm equation. For this purpose, we make a  trivial factorization to  the jump matrix
\begin{equation}
V=b_-^{-1} b_+, \  \ \  b_-=I, \  \ \  b_+=V,
\end{equation}
which leads to
\begin{equation}
w_-=0, \  w_+=J, \ w=w_-+w_+=J.
\end{equation}
Therefore  the Cauchy operator   can be  given by
\begin{equation}
\mathcal{C}_{w} f=\mathcal{P}^+\left(f w_-\right)+\mathcal{P}^{-}\left(f w_+\right)=\mathcal{P}^-(fJ).
\end{equation}
According to Beals-Coifman theory, the solution of the RH problem \ref{rhp:21} can be given by
\begin{align}
&N(x;k)=I+\frac{1}{2\pi i} \int_{\mathbb{R}\cup i \mathbb{R}} \frac{\varrho(x;s)J}{s-k} d s=I+\mathcal{C}(\varrho J)(z),\label{mb}
\end{align}
where $z=k^2$ and  $\varrho$ satisfies the  Fredholm equation
\begin{equation}
 \varrho-\mathcal{P}^{-}(\varrho J)=I,\nonumber
\end{equation}
which is equivalent to
\begin{equation}
N_-(x;k) = I+ \mathcal{P}^{-}\left(N_{-}  J \right)(z), \quad z \in \mathbb{R},\label{bcfc}
\end{equation}
by using the relation $N_- =  \varrho b_-=\varrho$. Once $N_{-}(x ; k) $ is found from the Fredholm integral equation (\ref{bcfc}),
 $N_+(x;k) $ can be obtained from the projection formula
\begin{equation}
N_+ (x;k)=I+ \mathcal{P}^+\left(N_- J \right)(z), \quad z \in \mathbb{R}.\label{g22}
\end{equation}

To consider solvability of the  equation (\ref{bcfc}),
we   let $G_\pm=N_\pm -I$  and transform  it   into a equivalent form
\begin{equation}
G_-(x;k) =  \mathcal{P}^{-}\left(G_{-}  J \right)(z) + \mathcal{P}^{-}\left(  J \right),   \label{bc6fc}
\end{equation}
Therefore the  solvability of   the RH problem  \ref{rhp:21}
if and only if there is a solution $G_{-}(x;k) \in L_{z}^{2}(\mathbb{R})$ of the Fredholm integral equation  (\ref{bcfc}).

\begin{proposition}
For $r(k)\in L_{z}^2(\mathbb{R})\cap L_{z}^{\infty}(\mathbb{R})$, then linear inhomogeneous equation (\ref{bc6fc})
has  a unique solution $G_-(x;k)\in L_{z}^2(\mathbb{R})$.
\label{lm3}
\end{proposition}
\begin{proof}
Since the operator $I-\mathcal{P}^{-}$ is  a Fredholm operator with  the index zero,
  the uniqueness of a solution to the linear integral equation (\ref{bc6fc})   is   equivalent
 to  show that   the homogeneous equation
 \begin{equation}
\left(I-\mathcal{P}^-\right) s=0 \label{gh22}
\end{equation}
  has a unique zero solution in $L_{z}^2(\mathbb{R})$.
For a given solution   $s \in L_{z}^2(\mathbb{R})$  of the equation  (\ref{gh22}),  we define two analytical  functions in $\mathbb{C} \backslash \mathbb{R}$ by
\begin{equation}
s_1(z):=\mathcal{C}(sJ)(z) \quad \text { and } \quad s_2(z):=\mathcal{C}(sJ)^{H}(z),
\end{equation}
where the superscript $H$   denotes the Hermite conjugate. In the upper half plane $\mathbb{C}^+$, taking 0 as the center of the circle and sufficiently large $R>0$ as the radius, we obtain a closed enclosing line $\{|z|=R, \operatorname{Im}z>0\}\cup (-R,R)$. As $s_1(z)$ and $s_2(z)$ are analytic functions when $z\in \mathbb{C}^+$, the Cauchy  theorem implies that
\begin{equation}
\oint s_{1}(z) s_2(z)dz=0.
\end{equation}
Because $s(z),J \in L_z^2(\mathbb{R})$, we know
$s_{1,2}(z)=\mathcal{O}\left(z^{-1}\right)$ as $|z| \rightarrow \infty$. Thus, the integral on the semi-circle goes to zero as $R \rightarrow
\infty$, from which we obtain
\begin{equation}
\begin{aligned}
\int_{\mathbb{R}} s_1(z) s_2(z) dz  =\int_{\mathbb{R}}\left[\mathcal{P}^{-}(s J)+s J\right]\left[\mathcal{P}^{-}(sJ)\right]^{H} \mathrm{d}z=0.\nonumber
\end{aligned}
\end{equation}
where the identity $\mathcal{P}^+-\mathcal{P}^-=I$   is used. As $\mathcal{P}^{-}(sJ)=s$, we finally obtain
\begin{equation}
 \int_{\mathbb{R}}sVs^{H} dz=0 \Longrightarrow  \int_{\mathbb{R}}{\rm Re} (sVs^{H} ) \mathrm{d}z=0,\label{gisg}
\end{equation}
which yields $ s=0$. Otherwise,  if $s\not=0$,  by Proposition \ref{p4},   we have
\begin{equation}
 \int_{\mathbb{R}}{\rm Re} ( sVs^{H} ) dz\geq \alpha_-\int_{\mathbb{R}}  |s|^2 \mathrm{d}z>0, \nonumber
\end{equation}
which is contradict  with  (\ref{gisg}).
  Therefore, the equation (\ref{gisg}) implies that $s=0$ and  the homogeneous equation
  (\ref{gh22})  has a unique solution in $L_{z}^2(\mathbb{R})$.
\end{proof}

\section{Inverse scattering transform  on  $z$-plane }
\label{sec:section3}
\hspace*{\parindent}
We decide to transform  original spectral problem and RH problem  from the $k$- plane to the $z$-plane by exploiting  the  odd-even  property    of the Jost functions in the column. This facilitates the further application of the Fourier transform  and Cauchy integral projection to estimate the space of the solution to the RH problem and the potential $u$.
\subsection{A ZS type spectral problem}
\hspace*{\parindent}
To put the original RH problem \ref{rhp:21} into a  RH problem  with a jump contour  on the  real axis, we introduce a transformation   defined by
\begin{equation}
\Phi(x;z)=A(x;k)\psi(x;k)B(x,k),\label{b1}
\end{equation}
with $z=k^2$ and
$$A(x;k)= \begin{pmatrix}
1 & 0 \\
-\bar{u}_x & 2ik
\end{pmatrix}, \ \ B(x;k)= \begin{pmatrix}
1 & 0 \\
0 & (2ik)^{-1}
\end{pmatrix},$$
then spectral problem (\ref{laxn1}) is changed into a  Zakharov-Shabat type spectral problem
\begin{equation}
\Psi_x+iz\left[\sigma_3,\Psi\right]=\widetilde{Q} \Psi,\label{xl1}
\end{equation}
where
\begin{equation}
\widetilde{Q} =\frac{1}{2 i}\begin{pmatrix}
|u_x|^2 & u_x \\
-2i\bar{u}_{xx}-\bar{u}_x|u_x|^2 & -|u_x|^2
\end{pmatrix}. \nonumber
\end{equation}

It can be shown that $\Psi^\pm(x;z)$ satisfies the following Volterra integral equations
\begin{equation}
\Psi^{\pm}(x;z)=I+\int_{\pm \infty}^xe^{-iz(x-y)\widehat{\sigma}_3}\widetilde{Q}\Psi^{\pm}(y;z)dy.\label{psi12}
\end{equation}

Denote $\Psi^{\pm}(x;z)=\left(\Psi^{\pm }_1 (x;z),\Psi^{\pm }_2(x;z)\right) $
with the subscripts  $1$ and $2$ denoting the first and second columns of $\Psi^{\pm}(x;z)$,
then we can show that

\begin{proposition}
\label{la1}
Suppose that $u (x) \in H^3(\mathbb{R}) \cap H^{2,1}(\mathbb{R})$. For every $z \in  \mathbb{R}$, there exists unique
solutions $\Psi^{\pm}(\cdot;z) \in L^{\infty}_x(\mathbb{R})$ satisfying the integral equations (\ref{psi12}) which have the following properties
\begin{itemize}

\item[$\blacktriangleright$] {\bf Analyticity}:
for every $x \in \mathbb{R}, \Psi_1^-(x; \cdot)$ and $\Psi_2^+(x ; \cdot)$ are analytically continued  in $\mathbb{C}^+$, whereas $\Psi_1^+(x;\cdot)$ and $\Psi_2^-(x;
\cdot)$ are analytically continued in $\mathbb{C}^-$.

\item[$\blacktriangleright$] {\bf Boundedness}:
there exists a positive $z$-independent constant $c$ such that
\begin{equation}
\left\|\Psi^{\pm}(\cdot;z)\right\|_{L^{\infty}_x (\mathbb{R})} \leq c.
\quad z \in \mathbb{C}^\pm.\label{pc}
\end{equation}
\end{itemize}
\end{proposition}
\begin{proof}
As illustrative example, we consider the boundedness and the analyticity of the Jost functions $\Psi^-(x;z)$.
From (\ref{psi12}), we get the following two integral equations
\begin{align}
&\Psi^-_1(x;z):=e_1+\int_{-\infty}^x
\text{diag}\left(1,e^{2iz(x-y)}\right)\widetilde{Q} \Psi^-_1(y;z) dy, \label{wew}\\
&\Psi^-_2(x;z):=e_2+\int_{-\infty}^x
\text{diag}\left(e^{-2iz(x-y)},1\right)\widetilde{Q} \Psi^-_2(y;z) dy.
\end{align}
It suffices to give the proof for $\Psi^-_1(x;z)$.  We
define the integral operator $F$ by
\begin{equation}
(Ff)(x;z):=\int_{-\infty}^x
\text{diag}\left(1,e^{2iz(x-y)}\right)\widetilde{Q} f(y)\mathrm{d}y,\label{kf}
\end{equation}
with vector  $f=(f_1,f_2)^T,$   then the  integral equation  (\ref{wew}) can be written in a operator equation
\begin{equation}
\Psi^-_1=e_1+F \Psi^-_1.\label{psii}
\end{equation}
According to the Fredholm alternative, we  show that   the homogeneous equation
\begin{equation}
(I-F)\Psi^-_1=0 \Longleftrightarrow F \Psi^-_1 =\Psi^-_1.\label{ps23}
\end{equation}
has a unique zero solution. The equation (\ref{ps23}) implies that we can show the operator $F$ has a unique fixed point by  the Banach fixed point theorem.

For vector function $f(x)=(f_1, f_2)^T \in L^\infty$, define the norm by
$$\| f\|_{L^\infty }= \| f_1\|_{L^\infty }+\|f_2\|_{L^\infty }. $$
 We expand (\ref{kf}) in the following form
\begin{equation}
Ff(x;z)=(h_{1},
h_{2}
)^T,
\end{equation}
where
\begin{align}
&h_{1 }=\frac{1}{2i} \int_{-\infty}^{x} (|u_y|^2 f_1+u_yf_2)dy,\nonumber\\
&h_{2 }=- \frac{1}{2i} \int_{-\infty}^{x} ((2i \bar{u}_{yy}+\bar{u}_y|u_y|^2)f_1+|u_y|^2 f_2)e^{2iz(x-y)}dy.\nonumber
\end{align}
We divide the above components of $Ff(x;z)$ into two parts to deal with. For the first and second term, we have for every $z\in \mathbb{C}^+$ and every $x_0 \in \mathbb{R}$,
\begin{align}
& \|h_{1 }\|_{L^\infty} = \sup_{x\in (-\infty, x_0)}
 \bigg| \frac{1}{2i} \int_{-\infty}^x(|u_y|^2 f_1+u_yf_2)dy\bigg|\nonumber \\
 &\leq \frac{1}{2  } ( \|f_1\|_{L^\infty} \|u_x\|_{L^2}^2+\|f_2\|_{L^\infty} \|u_x\|_{L^1})dy\nonumber \\
 &  \leq\frac{1}{2}(\|u_x\|^2_{L^2}+\|u_x\|_{L^1})\|f\|_{L^{\infty}}.
\end{align}
Noting that $z\in \mathbb{C}^+$, $|e^{2iz(x-y)}|\leq 1 $, in the similar way above,  we can show that
\begin{equation}
\begin{aligned}
\|h_{2 }\|_{L^\infty} \leq\frac{1}{2}(2\|u_{xx}\|_{L^1}+\|u_x\|^3_{L^3}+\|u_x\|^2_{L^2})\|f(\cdot; z)\|_{L^{\infty}}.
\end{aligned}
\end{equation}
Hence,
\begin{equation}
\begin{aligned}
\|Ff\|_{L^\infty} \leq  (2\|u_{x}\|^2_{L^2}+\|u_x\|^3_{L^3}+2\|u_{xx}\|_{L^1}+\|u_{x}\|_{L^1})\|_{L^{\infty}}.\nonumber
\end{aligned}
\end{equation}
which implies that $F$ is a contraction  operator  if $x_0 \in \mathbb{R}$ is chosen such that
\begin{equation}
2\|u_{y}\|^2_{L^2}+\|u_y\|^3_{L^3}+2\|u_{yy}\|_{L^1}+\|u_{y}\|_{L^1}<1.\label{xyq}
\end{equation}
According to the Banach fixed point theorem, for $x_0$ and every $z\in \mathbb{C}^+$, there exists a unique solution $\Psi^-_1(x;z)
\in L^{\infty}\left(-\infty, x_0\right)$ to the   equation (\ref{ps23}). $\mathbb{R}$ can be covered by a finite number of intervals, in which the estimate (\ref{xyq}) is satisfied. By putting the unique solutions in each
subinterval together, we finally obtain the unique solution $\Psi^-_1(\cdot;z) \in L^{\infty}(\mathbb{R})$ for every $z\in\mathbb{C}^+$.

The analyticity of $\Psi^-_1(x;\cdot)$ in $\mathbb{C}^+$ for every $x \in \mathbb{R}$ follows from the absolute and the uniform convergence of the
 Neumann series of the analytic functions in $k$. Define  a Neumann sequence by
\begin{align}
&w_0=e_1, \ \ w_{n+1}(x;z)= Fw_n = \int_{- \infty}^x F(x,y;z)w_n(y)dy.\label{3numen}
\end{align}
which constitute a series
\begin{align}
&w(x;z)= \sum_{n=0}^\infty w_n(x;z)=\sum_{n=0}^\infty F^ne_1.\label{3opio}
\end{align}

For a  matrix function $\widetilde{Q}=(\widetilde{Q}_{i j})_{ij=1}^2$,  define its $L^1$ matrix norm by
\begin{equation}
\|\widetilde{Q}\|_{L^1}:=\sum_{i=1}^2\sum_{j=1}^2\left\|\widetilde{Q}_{i j}\right\|_{L^1}.
\end{equation}
If $u \in H^3(\mathbb{R})\cup H^{2,1}(\mathbb{R})$, we have $\widetilde{Q}(u)
\in L^1(\mathbb{R})$, where the matrix $\widetilde{Q}(u)$ appears in the integral kernel $F$ given by (\ref{kf}). For (\ref{3numen}), we have
\begin{equation}
\|w_n\|_{L^{\infty}}= \left\| F^n e_1 \right\|_{L^{\infty}} \leq \frac{1}{n!}\left\|\widetilde{Q}(u)\right\|_{L^1}^n.\label{knjs}
\end{equation}
Hence, the Neumann series (\ref{3opio}) absolutely and uniformly  converges the solution $\Psi^-_1(x;z)$  of  the Volterra integral equation (\ref{wew})  for every $x \in
\mathbb{R}$ and $z \in \mathbb{C}^+$. Moreover $\Psi^-_1(x;z)$ is analytic in $\mathbb{C}^+$ for every $x \in \mathbb{R}$ and satisfies the bound (\ref{pc}).
\end{proof}

\begin{proposition} \label{prow}
The small-norm constraint
\begin{equation}
2\|u_{x}\|^2_{L^2}+\|u_x\|^3_{L^3}+2\|u_{xx}\|_{L^1}+\|u_{x}\|_{L^1}<1.\label{xyq3}
\end{equation}
is a sufficient condition which guarantees that $a(z)$ has no eigenvalues and no spectral singularity.
\end{proposition}
\begin{proof}
As we have already obtain $a(k)$ is even of $k$ in (\ref{abdc}), we will use the notation $a(z)$, which is essentially the same as $a(k)$. We only need to prove that under the small-norm constraint, $a(z)>0$ as $\mathrm{Im}z>0$. It is easy to see that if (\ref{xyq3}) holds, the Volterra integral equation (\ref{wew})  has the unique solution $\Psi^-_1(x;z)\in L^\infty_x(\mathbb{R})$, which also satisfies
\begin{equation}
\left\|\Psi^-_1(x;z)-e_{1}\right\|_{L^{\infty}_x(\mathbb{R})}<1, \quad z \in \mathbb{C}^{+}.
\end{equation}
Recall the integral expression of $a(z)$
\begin{equation}
a(z)=1+k \int_{\mathbb{R}} u_y \psi^-_{21}(x;k) \mathrm{d}x. \label{wewe}
\end{equation}
While by  (\ref{b1}), we get
\begin{equation}
\psi^-_{21}(x;k)=\frac{1}{2ik}\left(\bar{u}_x \Psi^-_{11}(x;z)+\Psi^-_{21}(x;z)\right).\nonumber
\end{equation}
The equation  (\ref{wewe}) can  be written  as
\begin{equation}
a(z)=1+\frac{1}{2i} \int_{\mathbb{R}}\left(|u_y|^2 \Psi^-_{11}(y;z)+u_y \Psi^-_{21}(y;z)\right) \mathrm{d}y,\nonumber
\end{equation}
which yields
\begin{equation}
|a(z)|\geq 1-\left|\frac{1}{2i} \int_{\mathbb{R}}\left(|u_y|^2\Psi^-_{11}(y;z)+u_y \Psi^-_{21}(y;z)\right) \mathrm{d} y\right|. \label{weed}
\end{equation}
  We estimate the above integral equation by
\begin{equation}
\begin{aligned}\nonumber
&\left|\frac{1}{2 i} \int_{\mathbb{R}}\left(|u_y|^2 \Psi^-_{11}(y;z)+u_y \Psi^-_{21}(y;z)\right) \mathrm{d}
y\right|\\
=&\frac{1}{2}\left|\int_{\mathbb{R}}\left(|u_y|^2\left(\Psi^-_{11}(y;z)-1\right)+|u_y|^2+u_y \Psi^-_{21}(y;z)\right) \mathrm{d}y\right| \\
\leq & \frac{1}{2}\|u_y\|_{L^2}^2+\frac{1}{2}\left\|\Psi^-_{11}-1\right\|_{L^{\infty}}\|u_x\|_{L^2}^2+\frac{1}{2}\left\|\Psi^-_{21}\right\|_{L^{\infty}}\|u_x\|_{L^{1}} \\
\leq & \frac{1}{2}\left(2\|u_x\|_{L^{2}}^{2}+\|u_x\|_{L^{1}}+\|u\|_{L^{3}}^{3}+2\left\|u_{xx}\right\|_{L^{1}}\right).
\end{aligned}
\end{equation}
Hence if we  add the small-norm constraint to the above equation to deduce that the right-hand side is small than $1/2$,  then the equation (\ref{weed}) indicates
\begin{equation}
|a(z)|>1/2, \quad \operatorname{Im} z>0.
\end{equation}
As $a(z)$ continues to the boundary $\mathbb{R}$, we can still get the same result on the boundary by the order preservation of the limit
\begin{equation}
|a(z)|\geq 1/2, \quad \operatorname{Im} z=0.
\end{equation}
The above procedure shows that $a(k)$ has no eigenvalues and no spectral singularity.
\end{proof}
This proposition show that  the small-norm constraint (\ref{xyq3}) is a sufficient
condition that the assumption  (\ref{asum})  is  satisfied.

The following proposition will   play  an  important role  in the  subsequent various  estimates on  the   Jost function, coefficients, Cauchy integral and potentials.

\begin{proposition} \cite{dep} \label{y1}
If $w\in H^1(\mathbb{R})$, then
\begin{equation}
\sup _{x \in \mathbb{R}}\left\|\int_{-\infty}^{x} \mathrm{e}^{2 i z(x-y)} w(y) dy\right\|_{L_{z}^2(\mathbb{R})} \leq
\sqrt{\pi}\|w\|_{L^2} ,\label{wgj1}
\end{equation}
and
\begin{equation}
\sup _{x \in \mathbb{R}}\left\|2 i z \int_{-\infty}^{x} \mathrm{e}^{2 i z(x-y)} w(y) d y+w(x)\right\|_{L_z^2(\mathbb{R})} \leq
\sqrt{\pi}\left\|\partial_{x} w\right\|_{L^2}.\label{wjd}
\end{equation}
Moreover, if $w \in L^{2,1}(\mathbb{R})$, then for every $x_0 \in \mathbb{R}^-$, we have
\begin{equation}
\sup _{x \in\left(-\infty, x_0\right)}\left\|\langle x\rangle \int_{-\infty}^x \mathrm{e}^{2i z(x-y)} w(y) d
y\right\|_{L_z^2(\mathbb{R})} \leq \sqrt{\pi}\|w\|_{L^{2,1}\left(-\infty, x_0\right)},\label{pil}
\end{equation}
where $\langle x\rangle:=\left(1+x^{2}\right)^{1/2}$.
\end{proposition}

\begin{proposition}
\label{lma2}
Under the conditions of Proposition \ref{la1}, for every $x \in \mathbb{R}$, the Jost functions $\Psi^{\pm}(x ; z)$
 admits  the following limits
\begin{align}
&\lim _{|z| \rightarrow 0} \Psi^{\pm}(x;z) =\begin{pmatrix}
1 &-u_x(x)\\
-\bar{u}_x(x)&1
\end{pmatrix},\\[4pt]
&\lim_{|z| \rightarrow \infty}\Psi^{\pm}(x;z)=e^{-ic_\pm\sigma_3}.\label{pinf}
\end{align}
\end{proposition}
\begin{proof}  We write the integral equation (\ref{psi12}) in the following scalar form
\begin{align}
&\Psi^-_{11}(x;z)=1+\frac{1}{2i}\int_{-\infty}^x u_y(y)\left[\bar{u}_y(y) \Psi^{-}_{11}(y;z)+\Psi^-_{21}(y;z)\right] \mathrm{d}y,\label{p11}\\
&\Psi^-_{21}(x;z)=-\frac{1}{2i} \int_{-\infty}^x e^{2iz(x-y)}\left[\left(2i\bar{u}_{yy}+|u_y|^2 \bar{u}_y\right) \Psi^-_{11} +|u_y|^2\Psi^-_{21} )\right] \mathrm{d}y,\label{p115}\\
&\Psi^-_{12}(x;z)=\frac{1}{2i}\int_{-\infty}^x e^{-2iz(x-y)}u_y(y)\left[\bar{u}_y(y) \Psi^-_{12}(y;z)+\Psi^-_{22}(y;z)\right] \mathrm{d}y,\label{p12}
\end{align}

By using (\ref{pc}), we know that  $ \Psi^-_{11}(x; z), \Psi^-_{21}(x; z) \in L^{\infty}_x$, hence for $u(x) \in H^3(\mathbb{R}) \cap H^{2,1}(\mathbb{R})$,  we have
\begin{align}
&\| ( 2i\bar{u}_{yy}+|u_y|^2 \bar{u} ) \Psi^-_{11}(y ; z)+ |u_y|^2\Psi^-_{21}(y;z)\|_{L^1} \leq   c (\|u_{xx}\|_{L^1}+ \|u_x\|_{L^2}^2).\nonumber
\end{align}
 While for ${\rm Im  }z>0$,  we have   $| e^{2iz(x-y)}| = e^{-{\rm Im}z(x-y)} \rightarrow 0, \ k\rightarrow \infty$,  then by using
  the Lebesgue dominated convergence theorem,  the equation (\ref{p115}) yields
\begin{align}
\Psi^-_{21}(x;z) \rightarrow 0, \ z\rightarrow \infty. \label{op1}
\end{align}

Taking the limit $z\rightarrow \infty  $ in (\ref{p11}) leads to
\begin{equation}
\lim_{|z|\rightarrow\infty}\Psi^-_{11}(x;z) =1+\frac{1}{2i} \int_{-\infty}^{x}|u_y(y)|^2 \lim_{|z|\rightarrow\infty}\Psi^-_{11}(y;z) \mathbf{d} y,\nonumber
\end{equation}
which admits  a  unique solution
\begin{align}
\lim_{|z|\rightarrow\infty}\Psi^-_{11}(x;z)= e^{-ic_-(x)}.  \label{op2}
\end{align}
 Finally combing (\ref{op1}) and (\ref{op2}) gives
\begin{equation}
\lim_{|z|\rightarrow\infty}\Psi^-_1(x;z)=e^{-ic_-(x)}e_1. \label{wfew}
\end{equation}
In a similar way,  we can show that
\begin{equation}
\lim_{|z|\rightarrow\infty}\Psi^-_2(x;z)=e^{ic_-(x)}e_2.\nonumber
\end{equation}
\end{proof}

In the following proposition, we give   smooth    properties of  the Jost functions
\begin{proposition}
\label{lma3}
If $u(x) \in   H^{2,1}(\mathbb{R})$, then for every $x \in \mathbb{R}^{\pm}$, we have
\begin{equation}
\Psi^{\pm}(x;\cdot)-e^{-ic_\pm(x)\sigma_3}\in H^1(\mathbb{R}).\label{mm1}
\end{equation}
\end{proposition}
\begin{proof}
Without loss of generality, we prove the statement for the Jost function $\Psi^- (x;z)$.
 We write the integral equation (\ref{psi12}) for $\Psi^-_1(x;z)$ into an operator equation  form
\begin{equation}
(I-F)\Psi^-=I,\label{mk}
\end{equation}
where the operator $F$ is given by (\ref{kf}).

Subtracting   the term  $(I-F) e^{-ic_-(x)\sigma_3}$ from   both sides of  equation  (\ref{mk}),   we  obtain  an  equivalent form
\begin{equation}
(I-F)\left(\Psi_--e^{-ic_-(x)\sigma_3}\right) =\begin{pmatrix}
0&n\\
m&0
\end{pmatrix}, \label{i-k}
\end{equation}
where
\begin{align}
&m=\int_{-\infty}^xe^{2iz(x-y)}w(y)dy,\quad   w(x)=-\partial_x(u_xe^{ic_-(x)}),\label{mxe}\\
&n=\int_{-\infty}^xe^{-2iz(x-y)}u_y(y)e^{ic_-(x)}dy.\nonumber
\end{align}

If $u \in   H^{2,1}(\mathbb{R})$, then $ w(x)  \in L^{2,1}(\mathbb{R})$. By the  estimates (\ref{wgj1}) and (\ref{pil}), we have $m(x;z), n(x;z) \in
L_x^\infty\left(\mathbb{R} ; L_k^2(\mathbb{R})\right)$. With the Sobolev inequality $\|u_x\|_{L^{\infty}} \leq$
$\frac{1}{\sqrt{2}}\|u\|_{H^{2,1}}$, the following bound is valid for every $x_0\in \mathbb{R}^-$:
\begin{equation}
\begin{aligned}
\sup _{x\in\left(-\infty, x_0\right)}\|\langle x\rangle m(x;z)\|_{L_z^2(\mathbb{R})} & \leq \sqrt{\pi}\left(\left\|
u_{xx}\right\|_{L^{2,1}}+\frac{1}{2}\left\|u_x^3\right\|_{L^{2,1}}\right) \\
& \leq c\left(\|u\|_{H^{2,1}}+\|u\|_{H^{2,1}}^3\right),\label{gjh}
\end{aligned}
\end{equation}
where $c$ is a positive  constant.

By  a  similar way  in deriving  (\ref{knjs}),  for every $f(x;z) \in
L_x^{\infty}\left(\mathbb{R} ; L_z^2(\mathbb{R})\right)$,   we can obtain
\begin{equation}
\left\|\left(F^n f\right)(x;z)\right\|_{L_x^{\infty} L_z^2} \leq \frac{1}{n
!}\left\|\widetilde{Q}(u)\right\|_{L^1}^n\|f(x;z)\|_{L_x^{\infty} L_z^2}.\nonumber
\end{equation}
Therefore, the operator $I-F$ is invertible on the space $L_x^{\infty}\left(\mathbb{R} ; L_z^2(\mathbb{R})\right)$. The bound of the inverse operator can be obtained accordingly
\begin{equation}
\left\|(I-F)^{-1}\right\|_{L_x^{\infty} L_z^2\rightarrow L_x^{\infty} L_z^2} \leq \sum_{n=0}^{\infty}
\frac{1}{n!}\left\|\widetilde{Q}(u)\right\|_{L^1}^n=\mathrm{e}^{\left\|\widetilde{Q}(u)\right\|_{L^1}}.\label{i-kn}
\end{equation}
Clearly, (\ref{i-kn}) is the norm in the space $L_x^{\infty}\left(\left(-\infty, x_0\right); L_z^2(\mathbb{R})\right)$
for every $x_0 \in \mathbb{R}$. By using (\ref{i-k}), (\ref{gjh}), and (\ref{i-kn}), we obtain $\Psi^{\pm}(x;\cdot)-e^{ic_\pm\sigma_3}\in L^2(\mathbb{R})$ from the following estimate for every $x_0\in \mathbb{R}^-$
\begin{equation}
\sup _{x\in\left(-\infty, x_0\right)}\left\|\langle x\rangle\left(\Psi^-(x;z)-e^{-ic_-(x)\sigma_3}\right)\right\|_{L_z^2(\mathbb{R})} \leq c
e^{\left\|\widetilde{Q}(u)\right\|_{L^1}}\left(\|u\|_{H^{2,1}}+\|u\|_{H^{2,1}}^3\right).\label{qjxgj}
\end{equation}

To complete the proof of (\ref{mm1}), our next task is to show that $\partial_z \Psi^-(x;z) \in L_x^{\infty}\left(\left(-\infty, x_0\right) ;
L_z^2(\mathbb{R})\right)$ for every $x_0 \in \mathbb{R}^-$. Column analysis is used here. According to the characteristics of the derivative with respect to $k$, we introduce the vector
\begin{equation}
q(x;z):= [ \partial_z \Psi^-_{11} (x;z), \ \  \partial_z \Psi^-_{21}(x;z)-2 i x \Psi^-_{21}(x;z)]^T,\nonumber
\end{equation}
which satisfies the following expression
\begin{equation}
(I-F)q(x;z)=m_1 e_1+m_2 e_2+m_3 e_2,\label{i-kv}
\end{equation}
where
\begin{align}
&m_1(x;z)=\int_{-\infty}^{x} y u_y(y)\Psi^{-}_{21}(y;z) \mathrm{d}y,\nonumber \\
&m_2(x;z)=\int_{-\infty}^{x} y e^{2 i z(x-y)}\left(2 i \bar{u}_{yy}(y)+|u_y(y)|^{2} \bar{u}_y(y)\right)\left(\Psi^-_{11}(y ;z)-e^{-ic_-(x)}\right) \mathrm{d}y,\nonumber \\
&m_3(x;z)=\int_{-\infty}^{x} y e^{2 i z(x-y)}\left(2 i \bar{u}_{yy}(y)+|u_y(y)|^{2} \bar{u}_y(y)\right) e^{-ic_-(x)} \mathrm{d}y.\nonumber
\end{align}
For every $x_0 \in \mathbb{R}^-$, using the H\"{o}lder's inequality to each term in the integral equation (\ref{i-kv}), we obtain the following bounds
by (\ref{wgj1}):
\begin{align}
&\sup _{x \in(-\infty, x_{0})}\left\|m_1(x;z)\right\|_{L_z^2(\mathbb{R})} \leq\|u\|_{L^1}\sup _{x \in\left(-\infty, x_0\right)}\left\|\langle x\rangle \Psi^-_{21}(x ; z)\right\|_{L_z^2(\mathbb{R})}, \nonumber\\
&\sup _{x \in(-\infty, x_{0})}\left\|m_2(x;z)\right\|_{L_z^2(\mathbb{R})}
\leq(2\| u_{xx}\|_{L^1}+\left\|u_x\right\|_{L^3}^3 )\sup _{x\in\left(-\infty, x_0\right)}\left\|\langle x\rangle\left(\Psi^-_{11}(x ; z)-e^{-ic_-(x)}\right)\right\|_{L_z^2(\mathbb{R})},\nonumber \\
&\sup _{x \in(-\infty, x_0)}\left\|m_3(x;z)\right\|_{L_z^{2}(\mathbb{R})}\leq \sqrt{\pi}\left(2\left\| u_{xx}\right\|_{L^{2,1}}+\left\|u_x^3\right\|_{L^{2,1}}\right).\nonumber
\end{align}
Because of the estimate (\ref{qjxgj}), the first two inequalities have finite bounds.

Using  (\ref{gjh}), (\ref{qjxgj}), and the integral equation (\ref{i-kv}), we summarize that $q(x;z) \in L_x^{\infty}\left(\left(-\infty, x_0\right)
; L_z^2(\mathbb{R})\right)$ for every $x_0 \in \mathbb{R}^-$. Also under the property showed in (\ref{qjxgj}), we finally obtain $\partial_z
\Psi^+_{11}(x;z) \in L_x^{\infty}\left(\left(-\infty, x_0\right) ; L_z^2(\mathbb{R})\right)$ for every $x_0\in \mathbb{R}^-$. This
completes the proof of (\ref{mm1}).
\end{proof}

\begin{proposition}
If $u\in H^{2,1}(\mathbb{R})$ and $u \in C^2(\mathbb{R})$, then Jost functions $\Psi^\pm_1(x;z)$ admits the following limits
\begin{equation}
\lim _{|z| \rightarrow \infty}z(\Psi^\pm_1(x;z)-e^{-ic_\pm(x)}e_1)=\widehat{\Psi}^\pm_{11}(x)e_1+\widehat{\Psi}^\pm_{21}(x)e_2.\label{uxgj}
\end{equation}
with
\begin{align}
&\widehat{\Psi}^\pm_{11}(x):=-\frac{1}{4}e^{-ic_\pm} \int_{\pm \infty}^{x}\left[u_y(y) \bar{u}_{yy}(y)+\frac{1}{2i}|u_y(y)|^4\right] dy,\label{j1}\\
&\widehat{\Psi}^\pm_{21}(x)=\frac{1}{2i} \partial_{x}(\bar{u}_{x}(x)e^{ic_\pm(x)}).\label{j2}
\end{align}
If $u \in H^3(\mathbb{R}) \cap H^{2,1}(\mathbb{R})$, then for every $x \in \mathbb{R}$, we have
\begin{equation}
z(\Psi^{\pm}_1(x;z)-e^{ic_-}e_1)-\left(\widehat{\Psi}^\pm_{11}(x) e_{1}+\widehat{\Psi}^\pm_{21}(x) e_2\right) \in L_z^2(\mathbb{R}).\label{k2p}
\end{equation}
\end{proposition}
\begin{proof}
 For every $x\in\mathbb{R}$ and every small $\delta>0$, we split the integral equation of $\Psi^-_{21}(x;z)$ for $(-\infty, x-\delta)$ and $(x-\delta, x)$, which is rewritten in the equivalent form:
\begin{equation}
\begin{aligned}
\Psi^-_{21}(x;z)&= \int_{-\infty}^{x-\delta} e^{2 iz(x-y)} \nu(y ;z) d y+\nu(x ;z) \int_{x-\delta}^{x} e^{2 i z(x-y)} \mathrm{d} y \\
&+\int_{x-\delta}^{x} e^{2 i z(x-y)}(\nu(y ;z)-\nu(x ; z)) \mathrm{d} y \equiv I_1+I_2+I_3,\nonumber
\end{aligned}
\end{equation}
where
\begin{equation}
\nu(x ;z):=-\frac{1}{2i}\left[\left(2 i \bar{u}_{xx}(x)+|u_x(x)|^2 \bar{u}_x(x)\right)\Psi^-_{11}(x;z)+|u_x(x)|^2 \Psi^-_{21}(x ; z)\right].\nonumber
\end{equation}
As $\nu(\cdot ;z) \in L^1(\mathbb{R})$, we have $I_1\rightarrow 0$ as $k\rightarrow\infty$. Meanwhile, $\nu(\cdot ; z) \in L^1(\mathbb{R})$ makes $I_3\rightarrow 0$ as $k \rightarrow\infty$.
As for $I_2$, we get the specific value
\begin{equation}
I_2=-\frac{1}{2iz}(1-e^{2iz\delta}) \nu(x ; z) .
\end{equation}
By choosing $\delta:=[\operatorname{Im}(z)]^{-1/2}$ such that $\delta \rightarrow 0$ as $\operatorname{Im}(z) \rightarrow \infty$. Then
\begin{equation}
\begin{aligned}
&\lim _{|k| \rightarrow \infty} (z\Psi^-)_{21}(x ;k)=-\frac{1}{2i} \lim _{|z| \rightarrow \infty} \nu(x;z)\\
&=-\frac{1}{4}\left(2 i \bar{u}_{xx}(x)+|u_x(x)|^{2} \bar{u}_x(x)\right) e^{ic_\pm}(x)=\widehat{\Psi}^-_{21}(x),
\end{aligned}
\end{equation}
with $\widehat{\Psi}^-_{21}(x)=\frac{1}{2i} \partial_{x}(\bar{u}_x(x) e^{ic_\pm(x)})$.

Obviously, the limit of $z\Psi^-_{21}(x ; z)$ reveals the relation between $z\Psi^-_{21}(x ; z)$ and $u_{xx}$,
which yields the limit (\ref{uxgj}). To deal with $\Psi^{-}_{11}(x ; z)$, (\ref{p11}) can be rewritten as the differential equation
\begin{equation}
\Psi^-_{11,x}(x;z)=\frac{1}{2 i}|u_x|^{2} \Psi^-_{11}(x;z)+\frac{1}{2 i} u_x(x) \Psi^-_{21}(x;z).\nonumber
\end{equation}
Using $e^{ic_-(x)}$ as the integrating factor,
\begin{equation}
\partial_x\left(e^{ic_-} \Psi^-_{11}(x ;z)\right)=\frac{1}{2i} u_x(x)e^{ic_-(x)} \Psi^-_{21}(x ;z).\nonumber
\end{equation}
We obtain another integral equation for $\Psi^-_{11}(x;z)$ :
\begin{equation}
\Psi^-_{11}(x;z)=e^{-ic_-(x)}+\frac{1}{2i} e^{-ic_-(x)} \int_{-\infty}^x u_y(y) e^{ic_-(y)}\Psi^-_{21}(x;z) dy,\label{pj11}
\end{equation}
Also, taking the limit $|z| \rightarrow \infty$ of $z(\Psi^-_{11}(x;z)-e^{ic_-(x)})$, we obtain
\begin{equation}
\lim _{|z| \rightarrow \infty}z(\Psi^-_{11}(x;z)-e^{-ic_-(x)})=\widehat{\Psi}^-_{11}(x).\nonumber
\end{equation}
with $\widehat{\Psi}^-_{11}(x)=-\frac{1}{4} e^{-ic_-(x)} \int_{-\infty}^{x}\left[u_y(y) \bar{u}_{yy}(y)+\frac{1}{2 i}|u_y(y)|^{4}\right] dy$. In the end, we obtain the limit of $\Psi^-_1$ in (\ref{uxgj}).

It's a natural idea to look at the space $z\left(\Psi^-_1-e^{-ic_-(x)}e_1\right)-\left(\widehat{\Psi}^-_{11} e_1+\widehat{\Psi}^-_{21} e_2\right)$. To prove (\ref{k2p}) for $\Psi^-_1$, under the result shown in (\ref{uxgj}), we subtract the right side term from the left side and obtain
\begin{equation}
\begin{aligned}
&(I-F)\left[z\left(\Psi^{-}_1-e^{-ic_-}e_1\right)-\left(\widehat{\Psi}^-_{11} e_1+\widehat{\Psi}^-_{21} e_{2}\right)\right]\\
=&zde_{2}-(I-F)\left(\widehat{\Psi}^-_{11} e_1+\widehat{\Psi}^-_{21} e_2\right),\label{if}
\end{aligned}
\end{equation}
Combining the integral equation (\ref{pj11}), we obtain
\begin{equation}
zm(x;z)e_2-(I-F)\left(\widehat{\Psi}^-_{11} e_1+\widehat{\Psi}^-_{21} e_2\right)=\widetilde{m}(x;z) e_2,
\end{equation}
with
\begin{equation}
\begin{aligned}
\widetilde{m}(x;z)&=z\int_{-\infty}^{x} e^{2 iz(x-y)} w(y) d y+\frac{1}{2i} w(x) \\
&-\frac{1}{2 i}\int_{-\infty}^xe^{2iz(x-y)}\left[\left(2i\bar{u}_{yy}(y)+\bar{u}_y(y)|u_y(y)|^{2}\right)\widehat{\Psi}^-_{11}(y)+|u_y(y)|^2 \widehat{\Psi}^-_{21}(y)\right] dy,\nonumber
\end{aligned}
\end{equation}
where $w$ is defined in (\ref{mxe}). By using bounds (\ref{wgj1}) and (\ref{wjd}), we have $\widetilde{m}(x;z) \in L_x^{\infty}\left(\mathbb{R} ; L_z^2(\mathbb{Z})\right)$ if $u \in H^3(\mathbb{R}) \cap H^{2,1}(\mathbb{R})$. Multiply both sides of this equation by $(I-F)^{-1}$ on $L_{x}^{\infty}\left(\mathbb{R} ; L_z^2(\mathbb{Z})\right)$, we improves the proof process of (\ref{k2p}) for $\Psi^-_1(x;z)$.
\end{proof}

\begin{proposition}
If $u\in H^3(\mathbb{R})\cap H^{2,1}(\mathbb{R})$, then for every $x \in \mathbb{R}^{\pm}$,  the   Jost functions $\psi^\pm(x;k)$  have the following properties
\begin{align}
&\psi^\pm_{11} -e^{-ic_\pm(x)}, \   2ik\psi^\pm_{21} -\bar{u}_x e^{-ic_\pm(x)}\in H_z^1(\mathbb{R}), \  \partial_z\psi^-_{11 } \in L^\infty_xL^\infty_z,\label{pk1}\\
&2ik\psi^{\pm}_{22}(x;k)-e^{ic_\pm(x)}, \quad \psi^{\pm}_{12}(x;k)\in H_z^1(\mathbb{R}),\ \  k^{-1}\psi^\pm_{21}(x;k) \in H_z^1(\mathbb{R}).\label{pk3}
\end{align}

\end{proposition}
\begin{proof}
From the first column of the  transformation (\ref{b1}), we have
\begin{align}
&\Psi^\pm_{11}(x;k)=\psi^\pm_{11}(x;k),\\
&\Psi^\pm_{21}(x;k)=-\bar{u}_x \psi^\pm_{11}(x;k)+ 2ik \psi^\pm_{21}(x;k).
\end{align}
Recalling the result (\ref{mm1}), we immediately obtain
$$\psi^\pm_{11}(x;k)-e^{-ic_\pm(x)} = \Psi^\pm_{11}(x;k)-e^{-ic_\pm(x)}\in H_z^1(\mathbb{R}). $$
Noting that $u_x(x)\in L^\infty(\mathbb{R})$,  $\Psi^\pm_{21}(x;k) \in H^1_z(\mathbb{R}) $, we find
$$ 2ik\psi^\pm_{21}(x;k)-\bar{u}_x e^{-ic_\pm(x)} =\Psi^\pm_{21}(x;k)+\bar{u}_x(\psi^\pm_{11}(x;k)-  e^{-ic_\pm(x)})\in  H_z^1(\mathbb{R}). $$

To prove  $k^{-1}\psi^\pm_{21}(x;k)$,  we start with its integral equation
\begin{equation}
\begin{aligned}
k^{-1}\psi^{\pm}_{21}(x ; k)&=-\int_{\pm \infty}^{x} e^{2iz(x-y)}  \bar{u}_y e^{-ic_\pm}  d y\\
&-\int_{\pm \infty}^{x} e^{2 i z(x-y)} \bar{u}_y\left(\psi^{\pm}_{11}(y ;k)-e^{-ic_\pm} \right) d y.\label{kpk}
\end{aligned}
\end{equation}
Due to $u\in H^3(\mathbb{R})\cap H^{2,1}(\mathbb{R})$ and $\psi^\pm_{11}(x;k)-e^{-ic_\pm(x)}\in H_z^1(\mathbb{R})$,
by using a similar way to Proposition \ref{y1},
we know that the two integrals on the right-hand side belong to $H_z^1(\mathbb{R})$, which yields the solution $k^{-1}\psi^{\pm}_{12}(0; k)\in H^1_z(\mathbb{R})$.

In a similar way to  the second  column of the  transformation (\ref{b1}),  we   can  show the first two formulas  in   (\ref{pk3}).
\end{proof}

\subsection{Lipschitz continuity  of the  scattering data}
\hspace*{\parindent}
Lipschitz continuity  of the  Jost functions is the basis of Lipschitz continuity  of the  scattering data. Our first priority is to find the Lipshitz continuous map from $u$ to $\Psi^\pm(x;z)$.  As a direct corollary of Proposition \ref{lma3},  we show  the map
\begin{equation}
H^{2,1}(\mathbb{R}) \ni u \rightarrow\left(\Psi^{\pm}(x;z)-
e^{-ic_\pm(x)\sigma_3}\right) \in
L_x^{\infty}\left(\mathbb{R}^{\pm}; H_z^1(\mathbb{R})\right)
\end{equation}
is Lipschitz continuous.

\begin{corollary}
Suppose that $u, \tilde{u} \in H^{2,1}(\mathbb{R})$ satisfy $\|u\|_{H^{2,1}},\|\tilde{u}\|_{H^{2,1}} \leq \delta$ for some $\delta>0$,
and their    corresponding Jost
functions  are   $\Psi^{\pm}(x;z)$ and $\widetilde{\Psi}^{\pm}(x;z)$ respectively. Then, there is a positive   constant $c$ such that for every $x \in
\mathbb{R}^{\pm}$, we have
\begin{equation}
\left\|\Psi^{\pm}(x; \cdot)-e^{-ic_\pm(x)\sigma_3}-\widetilde{\Psi}^\pm(x;
\cdot)+e^{-i\tilde{c}_\pm(x)\sigma_3} \right\|_{H^1} \leq
c\|u-\tilde{u}\|_{H^{2,1}}.\label{mpm1}
\end{equation}
Furthermore, if $u, \tilde{u} \in H^3(\mathbb{R}) \cap H^{2,1}(\mathbb{R})$ satisfy $\|u\|_{H^3 \cap H^{2,1}},\|\tilde{u}\|_{H^3 \cap H^{2,1}} \leq \delta$, then for every $x \in \mathbb{R}$, there is a positive $c$ such that
\begin{equation}
\left\|\breve{\Psi}^{\pm}(x ; \cdot)-\breve{\tilde{\Psi}}^{\pm}(x ; \cdot)\right\|_{L^{2}} \leq c\|u-\tilde{u}\|_{H^3 \cap H^{2,1}},\label{tpc}
\end{equation}
where
$$
\breve{\Psi}^{\pm}(x ;z):=z\left(\Psi^\pm_1-e^{-ic_-}e_1\right)-\left(\widehat{\Psi}^\pm_{11} e_1+\widehat{\Psi}^\pm_{21} e_{2}\right).
$$
\end{corollary}
\begin{proof}
As an illustrative example,  we only  prove  (\ref{mpm1}) for the   Jost function $\Psi^-_1(x;z)$. For every $x \in \mathbb{R}$, it is obvious that
\begin{equation}
\begin{aligned}
\left|e^{-ic_-(x)}-e^{-i\tilde{c}_-(x)} \right|&=\left|e^{\frac{1}{2i}
\int_{-\infty}^x\left(|u_y(y)|^2-|\tilde{u}_y(y)|^2\right)dy}-1\right|\leq 2\delta c_1 \|u_y-\tilde{u}_y\|_{L^2}.\label{mqmj}
\end{aligned}
\end{equation}
Using the integral equation (\ref{i-k}), we obtain
\begin{equation}
\begin{aligned}
&\left(\Psi^--e^{-ic_-(x)\sigma_3}
\right)-\left(\widetilde{\Psi}^--e^{-ic_-(x)\sigma_3}(\tilde{u})
\right)\\
&=(I-F)^{-1}(T-\widetilde{T})+(I-F)^{-1}(F-\tilde{F})(I-\tilde{F})^{-1} \widetilde{T},\label{cgj}
\end{aligned}
\end{equation}
where
$$T=\begin{pmatrix}
0&n\\
m&0
\end{pmatrix}, \ \ \ \widetilde{T}=\begin{pmatrix}
0&\widetilde{n}\\
\widetilde{m}&0
\end{pmatrix},$$
and  $\widetilde{F}$, $\widetilde{T}$ denote the same as $F$, $T$ but with $u$ being replaced by $\tilde{u}$. To estimate the first term in the right-hand side of (\ref{cgj}), we write
\begin{equation}
m(x;z)-\widetilde{m}(x;z)=\int_{-\infty}^x e^{2iz(x-y)}[w(y)-\widetilde{w}(y)]dy,\label{hqht}
\end{equation}
where
\begin{equation}
w-\widetilde{w}=\left(\bar{u}_{xx}+\frac{1}{2 i}|\tilde{u}_x|^2 \bar{u}_x\right)
e^{ic_-(x)}-\left(\bar{\tilde{u}}_{xx}+\frac{1}{2i}|\tilde{u}_x|^2 \bar{\tilde{u}}_x\right) e^{ic_-(x)}(\tilde{u}).\nonumber
\end{equation}
By using (\ref{mqmj}), we obtain $\|w-\widetilde{w}\|_{L^{2,1}} \leq c_{2} \|u-\tilde{u}\|_{H^{2,1}}$, where $c_2 $ is another   positive
constant.

Under (\ref{hqht}) and the result in Proposition \ref{y1}, we obtain for every $x_0 \in \mathbb{R}^-$:
\begin{equation}
\sup _{x \in\left(-\infty, x_{0}\right)}\|\langle x\rangle(m(x;z)-\widetilde{m}(x;z))\|_{L_z^2(\mathbb{R})} \leq \sqrt{\pi}
c_2 \|u-\tilde{u}\|_{H^{2,1}}.\label{xh}
\end{equation}
The estimate of $n$ is analogous. This gives the estimate for the first term in $(\ref{cgj})$.

For the second term in the right-hand side of (\ref{cgj}), we use (\ref{kf}) and find $F$ is a Lipschitz continuous operator from $L_x^{\infty}\left(\mathbb{R} ;
L_z^2(\mathbb{R})\right)$ to $L_x^{\infty}\left(\mathbb{R}; L_z^2(\mathbb{R})\right)$ which means for every $f\in
L_x^{\infty}\left(\mathbb{R};L_z^2(\mathbb{R})\right)$, we have
\begin{equation}
\|(F-\widetilde{F})f\|_{L_x^{\infty} L_z^2} \leq c_3 \|u-\tilde{u}\|_{H^{2,1}}\|f\|_{L_{x}^{\infty} L_z^2},\label{k-kt}
\end{equation}
where $c_3 $ is another positive constant independent of $f$. Combining (\ref{gjh}), (\ref{i-kn}), (\ref{cgj}), (\ref{xh}) and
(\ref{k-kt}), we derive for every $x_0 \in \mathbb{R}^-$:
\begin{equation}
\sup _{x \in\left(-\infty, x_0\right)}\left\|\langle x\rangle\left(\Psi^-(x; \cdot)-
e^{-ic_-(x)\sigma_3}-\widetilde{\Psi}^-(x;\cdot)+e^{-i\tilde{c}_-(x)\sigma_3}(\tilde{u})\right)\right\|_{L_z^2(\mathbb{R})}\leq c\|u-\tilde{u}\|_{H^{2,1}},\nonumber
\end{equation}
which gives the proof of (\ref{mpm1}) for $\Psi^-$ and $\widetilde{\Psi}^-$. The proof of the bound
(\ref{tpc}) is completed by recalling the same analysis to the integral equations (\ref{i-kv}) and (\ref{if}).
\end{proof}

As the spectral problem (\ref{cslp}) admits no resonances  shown by Proposition \ref{prow}, there hence exists a positive number $a_0$ such that
\begin{equation}
|a(k)| \geq a_0 >0, \quad k \in \mathbb{R} \cup i\mathbb{R},\label{aka}
\end{equation}
then we prove some properties of the scattering data in the $z$-plane.
\begin{proposition}
\label{l5}
If $u \in H^3(\mathbb{R}) \cap H^{2,1}(\mathbb{R})$, we have the following properties about $a(z)$ and $b(k)$
\begin{equation}
a(z)-\hat{a},\quad k b(k),\quad k^{-1} b(k)\in H_z^1(\mathbb{R}), \label{362}
\end{equation}
with
$$
\hat{a}=1+\frac{1}{2i}\int_{\mathbb{R}}|u_y(y)|^2e^{-ic_-(y)}\mathrm{d}y,
$$
and
\begin{equation}
k b(k), \quad k^{-1} b(k) \in L_z^{2,1}(\mathbb{R}).
\end{equation}
\end{proposition}
\begin{proof}
Due to the relation (\ref{phstgx}), the integration of $a(k)$ is as follows
\begin{equation}
a(z)=1+k \int_{\mathbb{R}} u_y(y) \psi^-_{21}(x;k) \mathrm{d}y.
\end{equation}
It can be deformed into
\begin{equation}
a(z)-\hat{a}=\frac{1}{2i}\int_{\mathbb{R}}(|u_y(y)|^2(\Psi^-_{11}(y;z)-e^{-ic_-})+u_y(y) \Psi^-_{21}(y;z))\mathrm{d}y.\nonumber
\end{equation}

According to $u\in H^3(\mathbb{R})\cap H^{2,1}(\mathbb{R})$ and (\ref{mm1}),
\begin{equation}
\begin{aligned}
\|a(z)-\hat{a}\|_{H_z^1(\mathbb{R})}&\leq  \sup_{x\in \mathbb{R}}  \| u_x \|_{L^2_x}^2
\sup_{x\in \mathbb{R}} \| \Psi^-_{11}-e^{-ic_-}\|_{H_z^1} +\|\bar{u}_x \|_{L^1_x} \sup_{x\in \mathbb{R}} \| \Psi^-_{21}\|_{H_z^1},\nonumber
\end{aligned}
\end{equation}
which yields $a(z)-\hat{a}\in H_z^1(\mathbb{R})$.

We then analyze the property of $b(k)$. By using the determinant (\ref{bk}), we write
\begin{equation}
b(k)=\psi^+_{11}(0 ;k) \psi^-_{21}(0;k)-\psi^+_{21}(0;k) \psi^-_{11}(0;k) .
\end{equation}
Under the relation $\Psi^+_{11}(x ;z)=\psi^+_{11}(x ;k)$ and $\Psi^+_{21}(x ;z)=-u_x\psi^+_{11}(x ;k)+2ik\psi^+_{21}(x ;k)$, we obtain
\begin{equation}
2ikb(k)=\Psi^+_{11}(0 ;z) \Psi^-_{21}(0;z)-\Psi^+_{21}(0;z) \Psi^-_{11}(0;z).\label{ikbk}
\end{equation}
With the same way by subtracting the limiting values of $\Psi^\pm_{11}(0;z)$ and $\Psi^\pm_{21}(0;z)$, we have
\begin{equation}
\begin{aligned}
2ikb(k)&=(\Psi^+_{11}(0 ;z)-e^{ic_+(0)}) \Psi^-_{21}(0;z)+e^{ic_+(0)}\Psi^-_{21}(0;z)\\
&-\Psi^+_{21}(0;z) (\Psi^-_{11}(0;z)-e^{ic_-(0)})-e^{ic_-(0)}\Psi^+_{21}(0;z).
\end{aligned}\label{kbk}
\end{equation}
Each term of (\ref{kbk}) belongs to $H_{z}^1(\mathbb{R})$. We conclude that $k b(k) \in H_{z}^{1}(\mathbb{R})$. In addition, based on the determinant (\ref{bk}), we have another equation about $k^{-1} b(k)$
\begin{equation}
k^{-1}b(k)=\Psi^+_{11}(0 ; z) k^{-1} \psi^-_{21}(0 ; k)-\Psi^-_{11}(0;z)k^{-1}\psi^+_{21}(0 ; k)
\end{equation}
Recalling that $k^{-1} \psi^{\pm}_{21}(0;k)$ belongs to $H_z^1(\mathbb{R})$ by (\ref{kpk}), we obtain $k^{-1}b(k) \in H_z^1(\mathbb{R})$.

Since $z k^{-1} b(k)=kb(k) \in H_z^1(\mathbb{R})$, we note that $k^{-1} b(k) \in L_z^{2,1}(\mathbb{R})$ . In addition, to prove that $kb(k)\in L_z^{2,1}(\mathbb{R})$, we multiply both sides of equation (\ref{ikbk}) by $z$ and put it in the form as follows
\begin{equation}
\begin{aligned}
2 i k z b(k)=& \Psi^+_{11}(0 ; z)\left(z \Psi^-_{21}(0;z)-\widehat{\Psi}^-_{21}(0)\right)-\Psi^-_{11}(0;z)\left(z \Psi^+_{21}(0; z)-\widehat{\Psi}^+_{21}(0)\right) \\
&+\widehat{\Psi}^-_{21}(0)\left(\Psi^+_{11}(0;z)-e^{ic_+(0)}\right)-\widehat{\Psi}^+_{21}(0)\left(\Psi^-_{11}(0;z)-e^{ic_-(0)}\right)
\end{aligned}\label{kzbk}
\end{equation}
where we have used the identity $\widehat{\Psi}^-_{21}(0) e^{ic_+(0)}-\widehat{\Psi}^+_{21}(0) e^{ic_-(0)}=0$ which is obtained from limits (\ref{pinf}) and (\ref{uxgj}). By (\ref{mm1}) and (\ref{k2p}), each terms in the representation (\ref{kzbk}) are in $L_z^2(\mathbb{R})$ so far. Therefore, we derive the final result that $kb(k) \in L_{z}^{2,1}(\mathbb{R})$.
\end{proof}

\begin{proposition}
\label{l12}
Suppose that $u(x) \in H^{3 }(\mathbb{R})\cup H^{2,1}(\mathbb{R})$, then we have
\begin{equation}
z^{-1}k^{-1}b(k)\in H^1_z(\mathbb{R}). \label{374}
\end{equation}
\end{proposition}
\begin{proof}  From (\ref{fi5}) and (\ref{phstgx}),    the  scattering data $b(k)$ and  Jost functions
$\psi^-_{11}, \psi^-_{21}$
  have the following integral representation
\begin{align}
&k^{-1} b(k)=- \int_{-\infty}^{\infty} \bar{u}_y(y) \psi^-_{11}(y;k) e^{-2iz y}\mathrm{d}y,\label{rr4}\\
&\psi^-_{11}(x;k)=1-k \int_{-\infty}^{x} u_y(y) \psi^-_{21}(y;k) \mathrm{d}y, \label{rr5}\\
&\psi^-_{21}(x;k)=k \int_{-\infty}^{x} \bar{u}_y(y) \psi^-_{11}(y ; k) e^{2iz(x-y)} \mathrm{d}y.\label{rr6}
\end{align}
In order  to estimate the decaying property of $b(k)$,  we write (\ref{rr4}) in  the form
\begin{align}
&k^{-1} b(k)= \int_{-\infty}^{\infty} \bar{u}_y(y)\left(\psi^-_{11}-1\right) e^{-2iz y} dy+ \int_{-\infty}^{\infty} \bar u_y (y)e^{-2izy}dy.\label{bzk}
\end{align}

Through integration by parts and the Fourier transformation,  the second integral  in (\ref{bzk}) becomes
\begin{equation}
\begin{aligned}
& \int_{-\infty}^{\infty} \bar{u}_y e^{-2izy} \mathrm{d}y =  \bar{u}(y) e^{-2izy}|_{-\infty}^{+\infty}+2iz \int_{-\infty}^{\infty} \bar{u}e^{-2izy}\mathrm{d}y
= i z \widehat{\bar{u} } (  z).\label{kqe}
\end{aligned}
\end{equation}

Next we make estimate on  the first integrand in (\ref{bzk}). Substituting  (\ref{rr6}) into  (\ref{rr5}) yields
\begin{equation}
\psi^-_{11}-1=-z \int_{-\infty}^x u_y(y) e^{2izy} dy \int_{-\infty}^{y} \bar{u}_s(s) e^{-2izs} \psi^-_{11}d s:=-z K \psi^-_{11},\label{ddy}
\end{equation}
where integral operator is defined  by
\begin{equation}
Kf =\int_{-\infty}^x u_y(y) e^{2i z y}dy \int_{-\infty}^{y} \bar{u}_s(s) e^{-2i z s} f ds,\nonumber
\end{equation}
which implies that 
$$ \|K \|_{L^\infty\to L^\infty}  \leq \|u_x\|_{L^2(\mathbb{R}) }^2.$$
With (\ref{ddy}),  the first term of (\ref{bzk})  becomes
\begin{align}
 &\int_{-\infty}^{\infty} \bar{u}_y(y)\left(\psi^-_{11}-1\right) e^{-2izy}\mathrm{d}y= -
 z \int_{-\infty}^{\infty} \bar{u}_y(y) e^{-2izy} K\psi^-_{11}\mathrm{d}y \nonumber\\
 & = - \frac{1}{2} z \widehat{ \bar{u}_x K \psi^-_{11}}( z).\label{kup3}
\end{align}
Combining  (\ref{bzk}), (\ref{kqe}) and (\ref{kup3}), we obtain
\begin{equation}
z^{-1} k^{-1} b(k)= - \frac{1}{2} \widehat{ \bar{u}_x   K  \psi^-_{11}   } (  z)+ i \widehat{\bar{u}}  (  z ).
\end{equation}
By Plancherel formula, we have
\begin{align}
& \|z^{-1} k^{-1} b(k)\|_{H^1_z}\leq \| \widehat{ \bar{u}_x   K  \psi^-_{11}   } (  z)\|_{H^1_z} + \|  i \widehat{\bar{u}}  (  z )\|_{H^1_z}\nonumber \\
&=\|   \bar{u}_x   K  \psi^-_{11}  ( x)\|_{L^{2,1}} +  \|  {u}  ( x )\|_{L^{2,1}}\nonumber \\
 & \leq  c   \|K \|_{L^\infty\to L^\infty}   \sup_{x\in \mathbb{R} } \| \psi^-_{11}  \|_{L^\infty_z }  \|   {u}_x  (x) \|_{L^{2,1}_x}  + c  \| u (x)\|_{L^{2,1}_x}.\nonumber
\end{align}

\end{proof}

\begin{lemma} \label{ertf} Let the set $ \Omega \subseteq \mathbb{R}$,
if  $f \in L^\infty(\Omega), \ f \in L^2 (\Omega)$,  then for  arbitrary   $2\leq p<\infty$,  we have
\begin{align}
& \|f\|_{L^p(\Omega)}\leq c \|f\|_{L^2(\Omega)}^{2/p}. \label{wefe1}
\end{align}
\end{lemma}

\begin{proof}  If  $f\in L^\infty(\Omega)$,   then for  $2\leq p<\infty$,  we have  $  |f|^{ p-2}\in  L^\infty(\Omega) $.
Further by  Holder inequality
\begin{align}
&   \|f\|_{L^p(\Omega) }^p = \int_{\Omega} |f|^{ p} dx=\int_{\Omega} |f|^2  |f|^{ p-2}  dx
\leq c  \int_{\Omega} |f|^2     dx = c  \| u \|_{L^2}^2. \nonumber
\end{align}
\end{proof}

\begin{proposition}
\label{l13}
Suppose that $u(x) \in H^{3 }(\mathbb{R})\cup H^{2,1}(\mathbb{R})$, then for  fixed small  $\delta>0$,  we have
\begin{align}
&z^{-2}k^{-1}b(k)\in L^2_z(-\delta, \delta), \quad   z^{-2} r_{1,2}\in L^2(\mathbb{R}).
\end{align}
\end{proposition}

\begin{proof}
The formula  (\ref{374}) in Proposition \ref{l12} implies that
\begin{align}
&  z^{-1}  k^{-1}b(k)  \in L^2_z(\mathbb{R}), \ \ \ z^{-1}  k^{-1}b(k)\in L^\infty_z(\mathbb{R}).  \label{po1}
\end{align}
which, with Lemma  \ref{ertf},    yields
  \begin{align}
& z^{-1}  k^{-1}b(k)   \in L^p_z(\mathbb{R}), \quad   2\leq p<\infty.  \label{po3}
\end{align}
Noting that for  fixed small  $\delta>0$,
$$ \|z^{-1}   \|_{L^{1/2}(-\delta,\delta) } =\left(\int_{-\delta}^{\delta} |z|^{-1/2} \mathrm{d}z\right)^2\leq c, $$
which together with (\ref{po3}),   and  by  Yang  inequality, we then derive that
\begin{align}
&  \|z^{-2}   k^{-1}b(k)  \|_{L^2(-\delta,\delta) }=  \|  z^{-1} ( z^{-1}  k^{-1}b(k)) \|_{L^2(-\delta,\delta) }\nonumber\\
 &\leq   \|z^{-1}   \|_{L^{1/2}(-\delta,\delta) }^{1/5}   \|  z^{-1}  k^{-1} b(k) \|_{L^8}^{4/5}(-\delta,\delta)\leq c.\label{po4}
\end{align}

Fix small  $\delta>0$,  let $\Omega=\mathbb{R}\setminus  (-\delta,\delta) $ and $\chi$ denote the indicator function of  the interval $\Omega$. Since  $ r_{1 }(0;z)\in L^\infty (\mathbb{R}), \  |a(k)|^{-1}\leq a_0^{-1} $ shown in (\ref{aka}),
 then  by   (\ref{po4}), we have
\begin{align}
&   \| z^{-2} r_{1 }( z) \|_{L^2(\mathbb{R})} \leq   \| \chi  z^{-2} r_{1}( z) \|_{L^2(\mathbb{R})} + \big\| (1-\chi)  \frac{1}{2i a(k)} z^{-2} k^{-1} b(k) \big\|_{L^2(\mathbb{R})} \nonumber\\
&\leq  2 \|r_{1 }(  z) \|_{L^{\infty}}  \|   z^{-2}  \|_{L^2(\delta,+\infty)} +    a_0^{-1}    \| z^{-2} k^{-1} b(k)   \|_{L^2 ( ( 0, \delta))}. \label{wege2}
\end{align}
In a similar way, by using Proposition \ref{l12},  we can show that
\begin{align}
&   \| z^{-2} r_{2 }( z) \|_{L^2(\mathbb{R})} \leq
 2 \|r_{ 2}(  z) \|_{L^{\infty}}  \|   z^{-2}  \|_{L^2(\delta,+\infty)} +   a_0^{-1}    \| z^{-1} k^{-1} b(k)   \|_{L^2( ( 0,\delta))}.\label{wege3}
\end{align}

\end{proof}

In conclusion, we show that the mapping from $u$ to scattering data $a(z)$ and $b(k)$
\begin{equation}
\begin{aligned}
&H^{2,1}(\mathbb{R}) \ni u \rightarrow a(z)-1\in H_{z}^1(\mathbb{R}),\\
&H^{3}(\mathbb{R}) \cap H^{2,1}(\mathbb{R}) \ni u \rightarrow k b(k), k^{-1} b(k)\in H_{z}^1(\mathbb{R}),\\
&H^3(\mathbb{R}) \cap H^{2,1}(\mathbb{R}) \ni u \rightarrow k b(k), k^{-1} b(k) \in L_{z}^{2,1}(\mathbb{R}),\\
&H^3(\mathbb{R}) \cap H^{2,1}(\mathbb{R}) \ni u \rightarrow z^{-2}kb(k), z^{-2}k^{-1}b(k) \in L_{z}^2(\mathbb{R}).
\end{aligned}\label{abul}
\end{equation}
is Lipschitz continuous.

If we express this property of Lipschitz continuity in detail the following corollary can be drawn:
\begin{corollary}
\label{cr3}
Let $u, \tilde{u} \in H^3(\mathbb{R}) \cap H^{2,1}(\mathbb{R})$ satisfy $\|u\|_{H^{3} \cap H^{2,1}},\|\tilde{u}\|_{H^3 \cap H^{2,1}} \leq \delta$ for some $\delta>0$. $u$ and $\tilde{u}$ correspond to two sets of scattering data $a, b$ and $\tilde{a}, \tilde{b}$, respectively. Then, there is a positive   constant $c$ such that
\begin{align}
&\left\|a(z)-\tilde{a}(z)\right\|_{H_z^1}\leq c \|u-\tilde{u}\|_{H^3 \cap H^{2,1}},\nonumber\\
&\|k b(k)-k \tilde{b}(k)\|_{H_z^1}+\left\|k^{-1} b(k)-k^{-1} \tilde{b}(k)\right\|_{H_z^1} \leq c \|u-\tilde{u}\|_{H^{2,1}},\nonumber\\
&\|k b(k)-k\tilde{b}(k)\|_{L_z^{2,1}}+\left\|k^{-1} b(k)-k^{-1} \tilde{b}(k)\right\|_{L_z^{2,1}} \leq c\|u-\tilde{u}\|_{H^3 \cap H^{2,1}}.\nonumber
\end{align}
\end{corollary}
\subsection{A new RH problem on the $z$-plane}
\hspace*{\parindent}
In order to use  the  theorems   on   classical Cauchy integral and  its projection   on the real axis,
 we  change the original  RH problem \ref{rhp:21} with jump contour on $\mathbb{R}\cup i\mathbb{R}$  in the $k$-plane
 into a  RH problem with jump contour on $\mathbb{R} $  in the $z$-plane.

We define  a matrix function
	   \begin{equation}
	   M(x ;z):=\begin{cases}
	   		e^{ic_+(x)\sigma_3}\left(\frac{\Psi_1^-(x ; z)}{a(z)},\Psi^+_2(x ;z)\right),  \; z\in \mathbb{C}^+,\\[5pt]
	   		e^{ic_+(x)\sigma_3}\left(\Psi_1^+(x ; z), \frac{\Psi^-_2(x;z)}{\bar{a}(z)}\right), \; z\in \mathbb{C}^-,
	   	\end{cases}\label{pdy}
	   \end{equation}
which then  satisfies the following RH problem on $z$-plane (see Figure \ref{klabdaz55}).

\begin{rhp} 
\label{rhp:31}
Find a  matrix  function $ M(x;z)$  with the following properties:
\begin{itemize}

\item[$\blacktriangleright$]  \emph{Analyticity}:  $M(x;z)$ is analytic in $ \mathbb{C}\setminus \mathbb{R}$.

\item[$\blacktriangleright$] \emph{Jump condition}: $M(x;z)$ satisfies the jump condition
\begin{equation}
M_+(x;z)= M_-(x;z)(I+R(x;z) ),\quad z \in \mathbb{R},\label{mrhp1}
\end{equation}
where the jump matrix  is defined  by
\begin{equation}
R(x;z)=\begin{pmatrix}
\bar{r}_1(z) r_2(z) & \bar{r}_1(z) \mathrm{e}^{-2iz x} \\
r_2(z) \mathrm{e}^{2iz x} & 0
\end{pmatrix},\label{rrf}
\end{equation}
with
\begin{equation}
r_1(z):=-\frac{b(k)}{2ika(k)}, \quad r_2(z):=\frac{2ikb(k)}{a(k)}, \quad z \in \mathbb{R}, \label{rxdy}
\end{equation}
which   satisfy the relations
\begin{align}
&r_2(z)=4 z r_1(z), \quad z \in \mathbb{R},\label{rfrz}\\
&\bar{r}_1(z) r_2(z)=|r(k)|^2, \quad z \in \mathbb{R}^{+}, \quad k \in \mathbb{R},\label{r11} \\
& \bar{r}_1(z)r_2(z)=-|r(k)|^{2}, \quad z \in \mathbb{R}^{-}, \quad k \in i \mathbb{R}.\label{r22}
\end{align}

\item[$\blacktriangleright$] \emph{Asymptotic conditions}:
\begin{equation}
M(x ; z) \rightarrow I
\quad \text { as } \quad|z| \rightarrow \infty. \label{we5ed}
\end{equation}
\end{itemize}
\end{rhp}

\begin{proof}  The first two  assertions  are  easy to be checked.
By using   (\ref{b1}) and (\ref{b2}),  we  obtain that
\begin{equation}
\Psi^\pm_2(x;z)=\frac{1}{2ik}\begin{pmatrix}
\psi^\pm_{12}\\
-\bar{u}_x\psi^\pm_{12}+2ik\psi^\pm_{22}
\end{pmatrix},\nonumber
\end{equation}
which together with  (\ref{jjw}) gives
\begin{equation}
\lim_{|z|\rightarrow\infty}\Psi^\pm_2(x;z)=e^{-ic_\pm}e_2,
\end{equation}
which together with (\ref{wfew}) leads to  (\ref{we5ed}).
\end{proof}

\begin{figure}[H]
\begin{center}
\begin{tikzpicture}
\draw [- ](-4,-4)--(-0.3,-4);
\draw [- ](-2.2,-5.2)--(-2.2,-2.5);
\draw [-latex](-2,-4)--(-1.1,-4);
\draw [-latex](-2,-4)--(-3.3,-4);
\draw [-latex](-2.2,-2.5)--(-2.2,-3.4);
\draw [-latex](-2.2,-4)--(-2.2,-4.7);
\node [thick] [above]  at (0.1,-4.2){\footnotesize ${\rm Re}k$};
\node [thick] [above]  at (-2.2,-2.6){\footnotesize ${\rm Im}k$};
\node [thick] [above]  at (-1.5,-3.8){\footnotesize $N_+$};
\node [thick] [above]  at (-1.5,-4.8){\footnotesize $N_-$};
\node [thick] [above]  at (-3,-3.8){\footnotesize $N_-$};
\node [thick] [above]  at (-3,-4.8){\footnotesize $N_+$};
\node [thick] [above]  at (6,-4.2){\footnotesize ${\rm Re}z$};
 \draw [-latex ] (-0.5,-3)--(1.5,-3 );
 \node [thick] [above]  at ( 0.5,-2.9) {\footnotesize $\Psi =A\psi B$};
\node [thick] [above]  at (3.5,-3.5){\footnotesize $M_+$};
\node [thick] [above]  at (3.5,-5){\footnotesize $M_-$};
\draw [-latex ](1.5,-4)--(3.5,-4);
\draw [  ](1.5,-4)--(5.5,-4);
\end{tikzpicture}
\end{center}
\caption{\footnotesize The left figure is the RH problem for $N(x;k)$ on $k$-plane whose jump
contour is $ \mathbb{R} \cap i\mathbb{R}$;  The right figure is the RH problem for $M(x;z)$ on $z$-plane whose jump
contour is $ \mathbb{R} $. }
\label{klabdaz55}
 \end{figure}
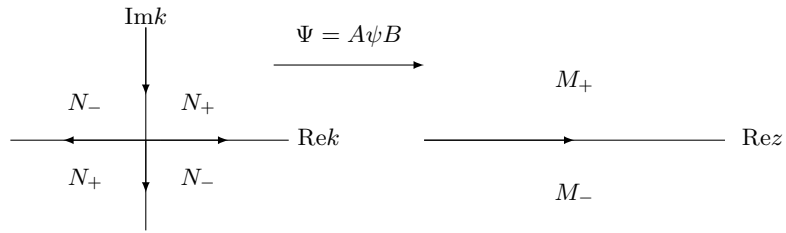

\subsection{Lipschitz continuity  of  reflection coefficient}
\hspace*{\parindent}
Further, we hope to obtain the Lipschitz continuity between the new scattering coefficient $r_{1,2}(z)$ and the initial value $u_0(x)$,
and also analyze the properties of the original scattering coefficient $r(k)$.

\begin{proposition}
\label{l6}
The two relations between $u(x)$ and $r_{1,2}(z)$ are listed here:
\begin{itemize}

\item[$\blacktriangleright$] If $u \in H^{2,1}(\mathbb{R})$, then $r_{1,2} \in H^{1}(\mathbb{R})$.

\item[$\blacktriangleright$] if $u \in H^3(\mathbb{R}) \cap H^{2,1}(\mathbb{R})$, then $r_{1,2} \in L^{2,1}(\mathbb{R})$ and $z^{-2}r_{1,2} \in L^2(\mathbb{R})$.

\end{itemize}
Furthermore, the mapping
\begin{equation}
H^3(\mathbb{R}) \cap H^{2,1}(\mathbb{R}) \ni u \rightarrow\left(r_1, r_2\right) \in \mathcal{W},\label{3r1r2}
\end{equation}
is Lipschitz continuous.
\end{proposition}
\begin{proof}
The assertion $r_{1,2} \in H^{1}(\mathbb{R})$  is directly  derived from Proposition \ref{l5} and \ref{l13}.
  Noting that
\begin{align}
&r_2(z)-\widetilde{r}_2(z)=\frac{2 i kb}{a}-\frac{2 i k\tilde{b}}{\tilde{a}},\label{rrt}\\
&z^{-2}(r_2(z)-\widetilde{r}_2(z))=\frac{2 i kz^{-2}b}{a}-\frac{2 i kz^{-2}\tilde{b}}{\tilde{a}}.\label{rrt2}
\end{align}
By  the Lipschitz continuity from $u(x)$ to $a(z)$ and $b(k)$ showed in (\ref{abul}),  then
 Lipschitz continuity of the mapping (\ref{3r1r2}) for $r_2(z)$ follows from the representation (\ref{rrt}) and (\ref{rrt2}).
 The Lipschitz continuity for  $r_1(z)$ also can be  proved in a similar way.
\end{proof}

\begin{proposition}
\label{pp1}
If $r_{1,2}(z)\in H_{z}^1(\mathbb{R})\cap L_{z}^{2,1}(\mathbb{R})$, then $r(k)\in L_{z}^{2,1}(\mathbb{R}) \cap
L_{z}^{\infty}(\mathbb{R})$.
\end{proposition}
\begin{proof}
As $r_{1,2} (z)\in L^{2,1}(\mathbb{R})$ and $|r(k)|^{2}=\operatorname{sign}(z) \bar{r}_1(z) r_2(z)$ for every $z \in
\mathbb{R}$, we have
\begin{equation}
\|r(k)\|_{L^{2,1}(\mathbb{R})}\leq \|r_1(z)\|_{L^2(\mathbb{R})}\|r_2(z)\|_{L^{2,1}(\mathbb{R})}
\end{equation}
Thus, we show that $r(k) \in L_{z}^{2,1}(\mathbb{R})$.

To give the proof of $r(k) \in L_{z}^{\infty}(\mathbb{R})$, we can define $r(k)$ in two equivalently forms according to the definition of $r_{1,2}(z)$ in (\ref{rxdy}):
\begin{equation}
r(k)=\begin{cases}
-2ik r_1(z) & |k| \leq 1 \\
(2ik)^{-1} r_2(z) & |k| \geq 1.
\end{cases}
\end{equation}
As $r_{1,2} \in H^{1}(\mathbb{R})$, we have $r_{1,2} \in L^{\infty}(\mathbb{R})$ and thereafter $r(k) \in L_{z}^{\infty}(\mathbb{R})$.
\end{proof}

\begin{proposition}
\label{pp2}
If $r_2(z) \in H_{z}^{1}(\mathbb{R}) \cap L_{z}^{2,1}(\mathbb{R})$, then $\left\|k r_2(z)\right\|_{L_{z}^{\infty}}
\leq\left\|r_2\right\|_{H^{1} \cap L^{2,1}}$
\end{proposition}
\begin{proof}
  For $r_2(z) \in H_{z}^1(\mathbb{R}) \cap L_{z}^{2,1}(\mathbb{R})$, we use Cauchy-Schwartz inequality to derive
\begin{align}
&|k r_2(z)|^2 = |z r_2^2(z)|=  \bigg | \int_0^z\left(r_2(s)^2+2 sr_2(s) r_2'(s)\right) d s\bigg|\nonumber\\
&\leq \|r_2(z)\|_{L^2} +2 \|r_2'(z) \|_{L^2}  \|z r_2 (z) \|_{L^2}\leq \|r_2(z)\|_{H^1\cap L^{2,1}},\nonumber
\end{align}
which implies   the desired bound.
\end{proof}

\section{Estimates on solutions to the RH problem  }
\label{sec:section4}
\subsection{Transitions of the RH problem}
\hspace*{\parindent}
The RH problem \ref{rhp:31} admits  the Beals-Coifman solution
\begin{equation}
M_{\pm}(x;z)=I+\mathcal{P}^{\pm}\left(M_{-}(x; \cdot) R(x ;\cdot)\right)(z), \quad z \in \mathbb{R},\label{mip3}
\end{equation}
which  can be used  to  estimate  the   columns of $M_{\pm}(x;z)-I$.

Next, we introduce the 2-by-2 matrix
\begin{equation}
T(x;z)=\left(M_{-,1}(x;z)-e_1, \quad M_{+,2}(x;z)-e_2\right),\label{qmu}
\end{equation}
and rewrite  (\ref{mip3})   into an equivalent form
\begin{equation}
T-\mathcal{P}^+\left(T R_+\right)-\mathcal{P}^-\left(T R_-\right)=F,\label{mjfzh}
\end{equation}
where
\begin{align}
&R_+(x;z)=\begin{pmatrix}
0 & \bar{r}_1(z) \mathrm{e}^{-2iz x} \\
0 & 0
\end{pmatrix}, \quad R_-(x;z)=\begin{pmatrix}
0 & 0 \\
r_2(z) \mathrm{e}^{2iz x} & 0
\end{pmatrix},\\
&F =\left(\mathcal{P}^{-}(r_2( z) \mathrm{e}^{2iz x})e_2, \quad \mathcal{P}^+(\bar{r}_1(z)
e^{-2iz x})e_1\right),\label{fep}
\end{align}
The estimate  on $F$ can be obtained  from  Proposition \ref{p6}. Our main task is to make a further estimate on the Cauchy integral  projection  in  (\ref{mjfzh}).

To analyze the derivatives of $M_{-,1}(x;z)$ and $M_{+,2}(x;z)$, we take the derivative of the inhomogeneous equation (\ref{mjfzh}) in $x$, which gives
\begin{equation}
\partial_x T-\mathcal{P}^+\left(\partial_x T\right) R_+-\mathcal{P}^-\left(\partial_x T\right) R_-=\widetilde{F},\label{pxm}
\end{equation}
where
\begin{equation}
\begin{aligned}
\widetilde{F}
=& 2 i\left(e_2 \mathcal{P}^-(z r_2(z) e^{2izx}), \quad e_1 \mathcal{P}^{+}(-z\bar{r}_1(z) e^{-2iz
x})\right) \\
+& 2 i\begin{pmatrix}
\mathcal{P}^-(zr_2(z)M^+_{12}(x;z) e^{2izx})&-\mathcal{P}^+(z\bar{r}_1(z)\left(M^-_{11}(x ;
z)-1\right) e^{-2izx}) \\[4pt]
\mathcal{P}^{-}(zr_2(z)\left(M^+_{22}(x;z)-1\right) e^{2izx}) & -\mathcal{P}^+(z\bar{r}_1(z)
M^{-}_{21}(x;z) \mathrm{e}^{-2iz x})
\end{pmatrix}.
\end{aligned}
\end{equation}

 Proposition \ref{pp2} inspires us to start with $kr_{1,2}$ to solve the problem. Therefore, for every $k\in \mathbb{C}\backslash\{0\}$, we introduce two new matrices
\begin{equation}
V_1(k):=\begin{pmatrix}
1 & 0 \\
0 & 2ik
\end{pmatrix},\quad
V_2(k):=\begin{pmatrix}
(2ik)^{-1} & 0 \\
0 & 1
\end{pmatrix},
\end{equation}
to give the original problem the form we want and $V_{1,2}$ admits
\begin{equation}
V_1^{-1}(k) R(x;z) V_1(k)=V_2^{-1}(k) R(x;z) V_2(k)=J(x;k),
\quad z \in \mathbb{R}, \quad k \in \mathbb{R} \cup i \mathbb{R}.\label{djbh}
\end{equation}

It can be seen from the following analysis that the transformation of $M(x;z)$ using these two matrices can not only complete the estimation of $M(x;z)$ itself and its derivative, but also make full use of the good properties of jump matrix $J(x;k)$.

In the following RH problem, the properties of the matrix elements are characterized in the $z$ plane: If $r_{1,2} \in H_{z}^{1}(\mathbb{R}) \cap L_{z}^{2,1}(\mathbb{R})$, then Proposition
\ref{pp2} implies that $J \in L_z^1(\mathbb{R}) \cap L_{z}^{\infty}(\mathbb{R})$ and $F(x ;z) \in$
$L_{z}^{2}(\mathbb{R})$ for every $x \in \mathbb{R}$. We consider the class of solutions to the RH problem \ref{rhp:31} such that for every
$x\in \mathbb{R}$. Therefore, in the following subsection, we equivalently reduce the RH problem \ref{rhp:31} in the $z$ plane to the RH problem related with the matrix $J(x;k)$ instead of the matrix $R(x;z)$.

\begin{rhp}
\label{rhp:3_2}
Find a  matrix  function $ Q_j(x;k)$  with the following properties:

\begin{itemize}

\item[$\blacktriangleright$]   \emph{Analyticity}: $Q_{j }(x;k)$ are analytic functions of $z $ in $\mathbb{C}^{\pm}$.

\item[$\blacktriangleright$]  \emph{Jump condition}: $Q_j(x;k)$ satisfies the jump condition
\begin{equation}
Q_{j,+}(x ; k)=(I+ J)Q_{j,-}(x;k)+D_j, \quad k \in \mathbb{R} \cup i\mathbb{R},\label{gzgf}
\end{equation}
with
\begin{equation}
Q_{j,\pm}(x;k):=M_{\pm}(x;z) V_j(k)-V_j(k), \quad D_j=V_j(k)J,\quad j=1,2.
\end{equation}

\item[$\blacktriangleright$]  \emph{Parity}:   the  columns of $Q_{j,\pm}(x;k), Q_{j,-}(x;k) J$, and $D_j $ have the same parity in $k$.

\item[$\blacktriangleright$]  \emph{Asymptotic conditions}:
\begin{equation}
Q_{j,\pm}(x;k) \rightarrow 0, \quad \text {as} \quad |k| \rightarrow \infty.
\end{equation}
\end{itemize}
\end{rhp}

\begin{proposition}
\label{pp9}
The RH problem \ref{rhp:3_2}  exists a  unique  solution.
\end{proposition}

\begin{proof}
The property is equivalent to the existence and the uniqueness of the solution to the RH problem \ref{rhp:21}, which is proved in Proposition \ref{lm3}.
\end{proof}

We can express the solution of the RH problem \ref{rhp:3_2} with the Cauchy integrals
\begin{equation}
Q_{j,\pm}(x;k)=\mathcal{C}\left(Q_{j,-}(x ; k) J+D_j\right)(z), \quad z \in \mathbb{C}^{\pm}.\label{g1221}
\end{equation}
There is a solution $Q_{j,-} (x;k)\in L_{z}^2(\mathbb{R})$ to the Fredholm integral equation:
\begin{equation}
Q_{j,-}(x;k)=\mathcal{P}^-\left(Q_{j,-}(x;k) J +D_j\right)(z), \quad z \in \mathbb{R},\label{g1}
\end{equation}
Once $Q_{j,-}(x;k) \in L_z^2(\mathbb{R})$ is found, $Q_{j,+}(x;k) \in$ $L_{z}^2(\mathbb{R})$ is obtained from the projection formula
\begin{equation}
Q_{j,+}(x;k)=\mathcal{P}^{+}\left(Q_{j,-}(x;k)J +D_j \right)(z), \quad z \in \mathbb{R}.\label{g22}
\end{equation}
As a result of Proposition \ref{pp9}, we know that $Q_{j,-}(x;k) \in L_z^2(\mathbb{R})$ exists.

To  estimate  the solutions to the RH problem \ref{rhp:31},  we   prove that the operator $\left(I-\mathcal{P}^-\right)^{-1}$ in the
 Fredholm equation (\ref{g1}) exists in  $L_{z}^2(\mathbb{R})$.
\begin{proposition}
\label{mp18}
If  $r(k) \in L_z^{2}(\mathbb{R}) \cap L_z^{\infty}(\mathbb{R})$ satisfying (\ref{1_rk}),  then
$\left(I-\mathcal{P}^{-}\right)^{-1}$ is a bounded operator in  $L_z^2(\mathbb{R})$. In particular, there is a
positive constant $c$ which only depends on $\|r(k)\|_{L_z^{\infty}}$ so that for every row vector $f \in L_z^{2}(\mathbb{R})$, we have
\begin{equation}
\left\|\left(I-\mathcal{P}^{-}\right)^{-1} f\right\|_{L_z^{2}} \leq c\|f\|_{L_z^{2}},\label{pn1}
\end{equation}
\end{proposition}
\begin{proof}
We consider the linear inhomogeneous equation (\ref{g1}) with $D_j \in L_z^2(\mathbb{R})$. Based on $\mathcal{P}^+-\mathcal{P}^-=I$,  then
(\ref{g1}) and (\ref{g22}) can be rewritten into  inhomogeneous equations
\begin{equation}
Q_{j,-}-\mathcal{P}^-\left(Q_{j,-} J\right)=\mathcal{P}^-(D_j), \quad Q_{j,+} -\mathcal{P}^-\left(Q_{j,-} J\right)=\mathcal{P}^+(D_j).\label{ihgsf}
\end{equation}
By Proposition \ref{lm3}, since $\mathcal{P}^{\pm}(D_j) \in L_z^2(\mathbb{R})$, there exists a unique solution to the inhomogeneous equation (\ref{ihgsf}), so that the decomposition $Q_j=Q_{j,+}-Q_{j,-}$ is unique. Therefore, we only need to find the estimates of $Q_{j,+}$ and $Q_{j,-}$ in $L_z^2(\mathbb{R})$.
We begin to deal with $Q_{j,-}$ and define two analytic functions in $\mathbb{C} \backslash \mathbb{R}$ by
\begin{equation}
q_1(z):=\mathcal{C}\left(Q_{j,-}J\right)(z) \quad \text {and} \quad q_2(z):=\mathcal{C}\left(Q_{j,-}J+D_j\right)^H(z)
\end{equation}
in a similar way as in the proof of Proposition \ref{lma3}. By Proposition \ref{p3}, we have $q_1(z)=\mathcal{O}\left(z^{-1}\right)$, $q_2(z) \rightarrow 0$ as $|z|\rightarrow\infty$, and
$q_{2}(z)\rightarrow 0$ as $|z| \rightarrow \infty$. Since $D_j\in L_z^2(\mathbb{R}),Q_{j,-}\in L_z^2(\mathbb{R})$, and
$J(k) \in L_z^2(\mathbb{R}) \cap L_z^{\infty}(\mathbb{R})$. Therefore, the integral on closed enclosing line $\{|z|=R, \operatorname{Im}z>0\}\cup (-R,R)$ goes to zero as $R \rightarrow \infty$ by Lebesgue's dominated convergence theorem.

Repeating the
procedures in the proof of Proposition \ref{lm3}, we obtain
\begin{equation}
0=\oint q_1(z) q_2(z) dz=\int_{\mathbb{R}}\left[Q_{j,-}\mathcal{P}^-(D_j)+Q_{j,-}J\right] Q_{j,-}^H \mathrm{d}z,
\end{equation}
in which we use the first inhomogeneous equation in the system (\ref{ihgsf}). By the bound (\ref{regt}) and the bound (\ref{pjx}), under the Cauchy-Schwartz inequality, there exists a positive
constant $\rho_- $ such that
\begin{equation}
\rho_- \left\|Q_{j,-}\right\|_{L^2}^2 \leq \operatorname{Re} \int_{\mathbb{R}} Q_{j,-}(I+J) Q_{j,-}^H dz=\operatorname{Re}
\int_{\mathbb{R}} \mathcal{P}^-(D_j) Q_{j,-}^H dz \leq\|D_j\|_{L^2}\left\|Q_{j,+}\right\|_{L^2}.\nonumber
\end{equation}
It is worth to note that the above estimate holds independently for the
corresponding row vectors in the matrices $Q_{j,-}$ and $D_j$. As $Q_{j,-}=\left(I-\mathcal{P}^-\right)^{-1} \mathcal{P}^-(D_j)$, for every row-vector $f
\in L_z^2(\mathbb{R})$ of the matrix $D_j \in L_z^2(\mathbb{R})$, the above inequality yields
\begin{equation}
\left\|\left(I-\mathcal{P}^-\right)^{-1} \mathcal{P}^- f\right\|_{L_z^2} \leq C_-^{-1}\|f\|_{L_z^2}.\label{pn2}
\end{equation}
Similarly, for $Q_{j,+}$, with the bounds (\ref{pjx}) (\ref{regt}) and (\ref{ig}) in Proposition \ref{p3}, there are positive constants $\rho_+ $ and $\rho_-$ such that
\begin{equation}
\begin{aligned}
&\rho_- \left\|Q_{j,+}\right\|_{L^2}^2 \leq \operatorname{Re} \int_{\mathbb{R}} Q_{j,+}(I+J)^H Q_{j,+}^{H} dz\\
&=\operatorname{Re}
\int_{\mathbb{R}} \mathcal{P}^+(D_j)(I+J)^{H} Q_{j,+}^{H} d z \leq \rho_+\|D_j\|_{L^2}\left\|Q_{j,+}\right\|_{L^2}. \nonumber
\end{aligned}
\end{equation}
As $Q_{j,+}=\left(I-\mathcal{P}^-\right)^{-1} \mathcal{P}^+(D_j)$, for every row vector $f
\in L_z^2(\mathbb{R})$  of the matrix $D_j \in L_z^2(\mathbb{R})$, the above inequality ensures
\begin{equation}
\left\|\left(I-\mathcal{P}^{-}\right)^{-1} \mathcal{P}^+ f\right\|_{L_{z}^{2}} \leq C_{-}^{-1} C_{+}\|f\|_{L_z^2}.\label{psgj}
\end{equation}
Finally, this proposition is proved via the bounds (\ref{pn2}), (\ref{psgj}) and the triangle inequality.
\end{proof}

\subsection{Estimates on the Beals-Coifman solutions }
\hspace*{\parindent}
In this subsection, we represent each column of the RH problem \ref{rhp:3_2} via the projection operators and give them fairly accurate estimates. Searching for analytic matrix functions $M_{\pm}(x;\cdot)$ in
$\mathbb{C}^{\pm}$ for every $x \in \mathbb{R}$, we introduce the following notations for the column vectors of the matrices $M_{\pm}(x;z)$
\begin{equation}
M_{\pm}(x ; z)=\left( M_{\pm,1}(x ; z), M_{\pm,2}(x ; z)\right).\label{mr}
\end{equation}
The expression of $Q_{j,\pm}(x;k)$ is given by
\begin{align}
&Q_{1, \pm} (x;k)=M_{\pm}(x;z) V_1(k)-V_1(k)=\left(\left(M_{\pm,1}(x;z)-e_1\right), (2ik)M_{\pm,2}(x;z)-e_2\right),\nonumber\\
&Q_{2, \pm}(x;k)=M_{\pm}(x;z) V_2(k)-V_2(k)=\left((2ik)^{-1}\left(M_{\pm,1}(x;z)-e_1\right), M_{\pm,2}(x;z)-e_2\right),\nonumber
\end{align}
where
\begin{equation}
D_j(x;k):=V_j(k)J =R(x;z)V_j(k), \ j=1,2.\label{f2}
\end{equation}
  We use
\begin{equation}
\left(Q_-J+D_j\right)_{i1}=\left(M_- V_j J\right)_{i1}=\left(M_-R V_j\right)_{i1}=\left(M_- R\right)_{i1},\quad i=1,2,\label{qjd}
\end{equation}
to obtain the expression of the first column of $Q_{j,\pm}$
\begin{align}
&M_{\pm,1}(x;z)-e_1=\mathcal{P}^{\pm}\left(M_-(x ;\cdot) R(x;\cdot)\right)_{11} , \quad z \in \mathbb{R},\label{mup1}\\
&(2ik)^{-1}(M_{\pm,1}(x;z)-e_1)=\mathcal{P}^{\pm}\left((2ik^{-1}) M_-(x ;\cdot) R(x;\cdot)\right)_{11} , \quad z \in \mathbb{R},\label{mup}
\end{align}
and the second column of $Q_{j,\pm}$
\begin{align}
&M^{\pm}_2(x ; z)-e_{2}=\mathcal{P}^{\pm}\left(M_{-}(x ; \cdot) R(x ; \cdot)\right)_{21} , \quad z \in \mathbb{R},\label{etap}\\
&2ik(M^{\pm}_2(x ; z)-e_{2})=\mathcal{P}^{\pm}\left(2ikM_-(x ; \cdot) R(x ; \cdot)\right)_{21} , \quad z \in \mathbb{R},\label{etap2}
\end{align}
\begin{remark}
From the results (\ref{mup1})-(\ref{etap2}) obtained above,   we    find that two versions   (\ref{mup1}) and (\ref{mup}) may seem to be inconsistent, as well as (\ref{etap}) and (\ref{etap2})
are inconsistent unless  we show that  (\ref{mup}) and (\ref{etap2}) are redundancy.
For the purpose,  considering Cauchy integral  projection   $\mathcal{P}^{\pm}\left(k f(k) \right)$ on  $k\in \mathbb{R}\cap i \mathbb{R}$,
   if $f(k)$ is  even function,  we define its orientated  integral contour is
the left graph  in   Figure \ref{figure3},  and if $f(k)$ is  even function, we take  its orientated  integral contour is
the right graph  in   Figure \ref{figure3}.
In this way  we can eliminate the coefficients $(2ik)^{-1}$ at both sides of the equation (\ref{mup}) and $2ik$ at both sides of the equation (\ref{etap2}).
\end{remark}

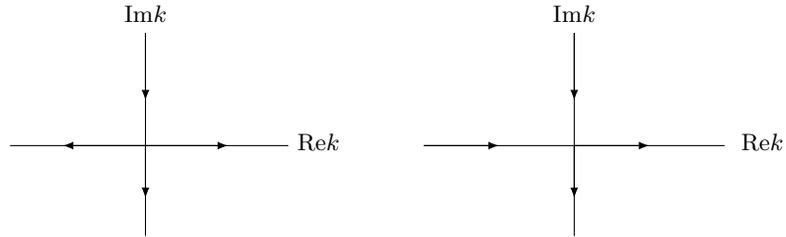
\begin{figure}[H]
\begin{center}
\begin{tikzpicture}
\draw [- ](-4,-4)--(-0.3,-4);
\draw [- ](-2.2,-5.2)--(-2.2,-2.5);
\draw [-latex](-2,-4)--(-1.1,-4);
\draw [-latex](-2,-4)--(-3.3,-4);
\draw [-latex](-2.2,-2.5)--(-2.2,-3.4);
\draw [-latex](-2.2,-4)--(-2.2,-4.7);
 \node [thick] [above]  at (0.1,-4.2){\footnotesize ${\rm Re}k$};
 \node [thick] [above]  at (-2.2,-2.5){\footnotesize ${\rm Im}k$};
\node [thick] [above]  at (6,-4.2){\footnotesize ${\rm Re}k$};
 \node [thick] [above]  at (3.5,-2.5){\footnotesize ${\rm Im}k$};
\draw [-latex ](3.5,-4)--(4.5,-4);
\draw [-latex ](1.5,-4)--(2.5,-4);
\draw [  ](1.5,-4)--(5.5,-4);
\draw [ ](3.5,-5.2)--(3.5,-2.5);
\draw [-latex ](3.5,-2.5)--(3.5,-3.4);
\draw [-latex ](3.5,-4)--(3.5,-4.7);
\end{tikzpicture}
\end{center}
\caption{\footnotesize The  orientated   integral contour for $\mathcal{P}^{\pm}\left(k f(k) \right)$ on  $k\in \mathbb{R}\cap i \mathbb{R}$: The left graph  is the    integral contour  for   even function $f(k)$;  The right
 graph  is the    integral contour for odd function $f(k)$. }
\label{figure3}
 \end{figure}


Combining (\ref{mup1}) with (\ref{etap}), we rewrite (\ref{mr}) into
\begin{equation}
M_{\pm}(x;z)=I+\mathcal{P}^{\pm}\left(M_{-}(x; \cdot) R(x ;\cdot)\right) , \quad z \in \mathbb{R},\label{mip}
\end{equation}
for representing the solution to the RH problem \ref{rhp:31} on the real line. The analytical
 continuation of functions $M_{\pm}(x ; \cdot)$ in
$\mathbb{C}^{\pm}$ is given by the Cauchy operators
\begin{equation}
M (x ; z)=I+\mathcal{C}\left(M_{-}(x ; \cdot) R(x ; \cdot)\right) , \quad z \in \mathbb{C}^{\pm}. \label{see}
\end{equation}
In the following proposition,  we show that   $M_\pm(x;z)$ can be estimated with  $r_{1,2}(z)$.
\begin{proposition}
\label{ll9}
Suppose $r_{1,2}(z)  \in H^{1}(\mathbb{R}) \cap L^{2,1}(\mathbb{R})$ such that the inequality (\ref{1_rk}) is satisfied. $M_\pm(x;z)$ has the estimate for
every $x \in \mathbb{R}$,
\begin{equation}
\left\|M_{\pm}(x;\cdot)-I\right\|_{L^2(\mathbb{R})} \leq c\left(\left\|r_1\right\|_{L^2(\mathbb{R})}+\left\|r_2\right\|_{L^2(\mathbb{R})}\right).\label{mjir}
\end{equation}
with a positive constant $c$ which only depends on $\left\|r_{1,2}\right\|_{L^{\infty}(\mathbb{R})}$.
\end{proposition}
\begin{proof}
Under the condition in proposition \ref{pp1}, $R(x;z)V_{j}(k)$ belongs to $L_z^2(\mathbb{R})$ for every $x \in \mathbb{R}$ by he explicit expressions of $D_j$ by (\ref{f2}). Combining the expression (\ref{mup1}) and (\ref{etap}), we conclude that the estimate of $R(x;z)V_{j}(k)$ is essential for using the Proposition \ref{mp18}.

Meanwhile, there exists a positive constant $c$ which only
depends on $\left\|r_{1,2}\right\|_{L^{\infty}(\mathbb{R})}$ such that for every $x \in \mathbb{R}$,
\begin{equation}
\left\|R(x;z)V_{j}(k)\right\|_{L_z^{2}} \leq c\left(\left\|r_1\right\|_{L^{2}}+\left\|r_2\right\|_{L^{2}}\right),\quad j=1,2.\label{rm}
\end{equation}
According to
\begin{equation}
\mathcal{P}^-( D_j) =\mathcal{P}^-\left(R(x;z)V_{j}(k)\right)(k),\quad j=1,2,
\end{equation}
and the above discussion, the integral equation (\ref{mip}) for the projection operator $\mathcal{P}^-$ is obtained from the integral
equations (\ref{ihgsf}).

In the end, each element of
$M_-(x;z)$ satisfies the bound (\ref{psgj}) for the corresponding row vectors of $\mathcal{P}^-( D_j)$. Combining the bounds
(\ref{pn1}) and (\ref{rm}), we finally derive the bound (\ref{mjir}).
\end{proof}

Next, we begin to derive the accurate estimates on the solution to the integral equations (\ref{mip}), where the key step is to derive the estimation on the
scattering coefficients $r_1$ and $r_2$ via the Fourier theory.

For a given function $f(z)\in L^2(\mathbb{R})$,  we define  Fourier transform and inverse transform
  $$\widehat{f}(\xi) = \frac{1}{2 \pi} \int_{\mathbb{R}} f (z) \mathrm{e}^{-i z \xi  } \mathrm{d} z,\  \ \ f(z) =   \int_{\mathbb{R}} \widehat{f}(\xi) \mathrm{e}^{ i z \xi  } \mathrm{d}\xi,$$
 then we can show    the  following proposition
\begin{proposition}
\label{prop7}
 If $f(z)\in L^2(\mathbb{R})$,  then we have
\begin{align}
&\mathcal{P}^+\left(f(z) e^{-2iz x}\right)  =\int_{2x}^{+\infty} \widehat{f}(\xi)e^{iz(\xi-2x) }\mathrm{d}\xi,\label{wdw}\\
&\mathcal{P}^-\left(f(z) e^{2iz x}\right) =-\int^{\infty}_{2x} \widehat{f}(\xi)e^{-iz(\xi-2x ) }\mathrm{d}\xi,\label{wdw2}
\end{align}
  which change  the Cauchy projections   into  the integral  on the positive   half-lines.
\end{proposition}

\begin{proof} By using the definition of projection operator and Fourier inverse transform, we have
 \begin{align}
 &     \mathcal{P}^+  ( f(z)  e^{-2izx})=\frac{1}{2\pi i}\lim _{\epsilon \downarrow 0}\int_{\mathbb{R}} \frac{f(s)e^{-2isx}}{s-(z+i\varepsilon)}\mathrm{d}s\nonumber\\
 &  = \frac{1}{2\pi i}\lim _{\epsilon \downarrow 0}\int_{\mathbb{R}}\left( \int_{\mathbb{R}}  \widehat{f} (\xi) e^{ i s\xi} \mathrm{d}\xi\right)    \frac{e^{-2isx}}{s-(z+i\varepsilon)}ds\nonumber\\
 & = \frac{1}{2\pi i}\int_{\mathbb{R}}\widehat{f}  (\xi ) \left(  \lim _{\epsilon \downarrow 0} \int_{\mathbb{R}}         \frac{e^{is(\xi-2x)   }}{s-(z+i\varepsilon)}ds\right)\mathrm{d}\xi.\label{pzrf}
\end{align}
Further through    residue computation, we obtain
\begin{align}
&\lim _{\epsilon \downarrow 0} \frac{1}{2 \pi i} \int_{\mathbb{R}} \frac{\mathrm{e}^{i s(\xi-2x)}}{s- (z+i \epsilon)} d s=\lim _{\epsilon \downarrow
0}\begin{cases}
\mathrm{e}^{ i(z+i \epsilon)(\xi-2x)}, & \text {if} \quad \xi-2x>0 \\
0, & \text {if} \quad \xi-2x<0
\end{cases}\nonumber\\[6pt]
&=\chi(\xi-2x)e^{iz(\xi-2x) },\label{orer}
\end{align}
with $\chi(s)$ being a  characteristic function.  Substituting (\ref{orer}) into  (\ref{pzrf}) yields
\begin{align}
 & \mathcal{P}^+\left(f(z) e^{-2iz x}\right)  =\int_{2x}^{+\infty} \widehat{f}(\xi)e^{iz(\xi-2x) }\mathrm{d}\xi.\label{gopr4r2}
\end{align}

By definition, we have
\begin{align}
 &    \mathcal{P}^-  ( f(z) e^{ 2izx}) =\frac{1}{2\pi i}\lim_{\varepsilon \rightarrow 0}\int_{R} \frac{f(s)e^{ 2isx}}{s-(z-i\varepsilon)}\mathrm{d}s.\nonumber
\end{align}
Taking conjugation  on both sides and using (\ref{gopr4r2}), we have
 \begin{align}
 &    \overline{\mathcal{P}^-  (  f(z) e^{ 2izx})} = -\lim_{\varepsilon \rightarrow 0}\frac{1}{2\pi i}\int_{R} \frac{ \bar{f}(s)e^{-2isx}}{s-(z+i\varepsilon)}ds =
  - \mathcal{P}^+ (  \bar f(z) e^{-2izx})\nonumber\\
  &=-\int_{2x}^{+\infty}\widehat{\bar{f}\  }  (\xi)e^{iz(\xi-2x) }  \mathrm{d}\xi. \label{ww}
\end{align}
Again  taking conjugation  on both sides of (\ref{ww})  gives
 \begin{align}
 &  \mathcal{P}^-  (   f(z) e^{ 2izx})   =- \int_{2x}^\infty   \widehat{f}  (\xi)e^{-iz(\xi-2x) }  \mathrm{d}\xi,\nonumber
\end{align}
which exactly is  the formula  (\ref{wdw2}).
\end{proof}
By using the above Proposition \ref{prop7}, we can derive a series of estimators on reflection coefficients $r_{1,2} (z)$.
\begin{proposition}
\label{p6}
For every $x_0 \in \mathbb{R}^+$ and every $r_{1,2}(z) \in H^1(\mathbb{R})$, we have
\begin{align}
&\sup _{x \in\left(x_{0}, \infty\right)}\left\| \langle x\rangle \mathcal{P}^{+}\left(z^{-i}\bar{r}_1(z) e^{-2i z x}\right)\right\|_{L_z^2}
\leq\left\|z^{-i}\bar r_1(z)\right\|_{H^1_z},\quad i=0,1,\label{xp}\\
&\sup _{x \in\left(x_0, \infty\right)}\left\|\langle x\rangle \mathcal{P}^-\left(r_2(z) e^{2iz x}\right)\right\|_{L_{z}^2}
\leq\left\|r_2(z)\right\|_{H^1_z}.\label{xpfr}
\end{align}
In addition, we have
\begin{align}
&\sup _{x \in \mathbb{R}}\left\|\mathcal{P}^+\left(z^{-i}\bar{r}_1(z)e^{-2iz x}\right)\right\|_{L_z^\infty} \leq
\frac{1}{\sqrt{2}}\left\|z^{-i}r_1(z)\right\|_{H^1},\quad i=0,1,\label{pzr}\\
&\sup_{x \in \mathbb{R}}\left\|\mathcal{P}^-\left(r_2(z) \mathrm{e}^{2iz x}\right)\right\|_{L_{z}^{\infty}} \leq
\frac{1}{\sqrt{2}}\left\|r_2(z)\right\|_{H^{1}}.\label{pfr}
\end{align}
Furthermore, if $r_{1,2}(z)  \in L^{2,1}(\mathbb{R})$, then
\begin{align}
&\sup _{x \in \mathbb{R}}\left\|\mathcal{P}^+\left(z  \bar{r}_1(z) e^{-2 i z x}\right)\right\|_{L_{z}^{2}} \leq\left\|
r_1(z)\right\|_{L_{z}^{2,1}},  \label{lq1}\\
&\sup _{x \in \mathbb{R}}\left\|\mathcal{P}^{-}\left(z  r_2(z) e^{2 i z x}\right)\right\|_{L_{z}^{2}} \leq\left\|
r_2(z)\right\|_{L_{z}^{2,1}}.\label{lq2}
\end{align}

\end{proposition}
\begin{proof}
For a given function $r(z) \in L^{2}(\mathbb{R})$,  then its  Fourier transform $\widehat{r} (\xi) \in
L^{2}(\mathbb{R})$ and  by  Plancherel formula, we have
\begin{equation}
\|r(z)\|_{L^2}^2=2\pi\|\widehat{r}(\xi)\|_{L^2}^2.
\end{equation}
Further, a general conclusion shows $r (z)\in H^{1}(\mathbb{R})$ if and only if $\widehat{r} (\xi)\in L^{2,1}(\mathbb{R})$. Also, $r(z) \in L^{2,1}(\mathbb{R})$ if and only if $\widehat{r}
(\xi) \in H^{1}(\mathbb{R})$.
To prove (\ref{xp}), we use Proposition \ref{prop7}  and  get
\begin{equation}
\begin{aligned}
\mathcal{P}^{+}\left(z^{-1}\bar{r}_1(z) e^{-2 i z x}\right)
&=\int_{2x}^{\infty}\widehat{ (z^{-1}\bar{r}_1(z))}(\xi)  e^{iz(y-2x) } d\xi.
\end{aligned}
\end{equation}
By the bound (\ref{pil}) in Proposition \ref{y1},  the  bound (\ref{xp}) is obtained
\begin{equation}
\sup _{x\in\left(x_0, \infty\right)}\left\|\langle x\rangle \int_{2x}^{\infty} \widehat{ (z^{-1}\bar{r}_1)}(\xi)  e^{i(\xi-2x)z} d\xi\right\|_{L_{z}^2} \leq
\sqrt{2\pi}\left\|\widehat{ (z^{-1}\bar{r}_1)}(\xi) \right\|_{L^{2,1}}\leq
\sqrt{2\pi}\left\|  z^{-1}\bar{r}_1(z)   \right\|_{H^1_z}.\nonumber
\end{equation}
Similarly, we  derive the bound (\ref{pzr})  as follows
\begin{equation}
\left\|\mathcal{P}^+\left(z^{-i}\bar{r}_1(z) e^{-2iz x}\right)(z)\right\|_{L_{2}^{\infty}}
\leq   \sqrt{\pi}\left\|\widehat{(z^{-i}\bar{r}_1)}(\xi)\right\|_{L_z^{2,1}} \leq
\frac{1}{\sqrt{2}}\left\|z^{-i}r_1(z)\right\|_{H^1}.
\end{equation}
The bounds (\ref{xpfr}), (\ref{pfr}), (\ref{lq1}) and (\ref{lq2}) are obtained in the same way.
\end{proof}
Our ultimate goal is to estimate the solutions $M(x;z)$   to the  RH problem \ref{rhp:31}. By Proposition \ref{ll9}, these solutions on the real line can be
written in the integral Fredholm form (\ref{mip}). One remaining task is to estimate the column vectors $M_{-,1}-e_1$ and $M_{+,2}-e_2$. From the equation
(\ref{mup1}), we obtain
\begin{equation}
M_{-,1}(x;z)-e_1=\mathcal{P}^-\left(r_2(z) e^{2iz x}M_{+,2}(x;z)\right)(z), \quad z \in \mathbb{R}.\label{muf}
\end{equation}
Also, starting from (\ref{etap}), we know
\begin{equation}
M_{+,2}(x ;z)-e_{2}=\mathcal{P}^+\left(\bar{r}_1(z) \mathrm{e}^{-2iz x} M_{-,1}(x; z)\right)(z), \quad z \in
\mathbb{R}.\label{etaz}
\end{equation}
Both the calculations take advantage of the following facts
\begin{align}
&\left(M_{-} R\right)_{11}=e^{ic_+(x)\sigma_3}\left(M_{-} R\right)_{11}=r_2(z) e^{2 i
z x} M_{+,21},\nonumber\\
&\left(M_- R\right)_{21}=e^{ic_+(x)\sigma_3}\left(M_{-} R\right)_{21}=\bar{r}_1(z) e^{-2iz x} M_{-,11}.\nonumber
\end{align}

From the explicit expression (\ref{fep}) for $F(x ; z)$, we note that the first row vector of $F(x ; z)$ is
equal to $F(x ; z)V_2(z)$ and the first row vector of $F(x ; z)$ is
equal to $F(x ; z)V_1(z)$ which have the following representation
\begin{align}
&F_1(x ; z)=(F(x ; z)V_2)_1=\left(0, \mathcal{P}^{+}(\bar{r}_1(z) \mathrm{e}^{-2 i z x})\right),\\
&F_2(x ; z)=(F(x ; z)V_1)_2=\left(\mathcal{P}^{-}\left(r_2(z) \mathrm{e}^{2 i z x}\right), 0\right).
\end{align}

\begin{proposition}
\label{l7}
For every $x_0 \in \mathbb{R}^+$ and every $r_{1,2} \in H^1(\mathbb{R})$, the unique solution to the system of the integral equations (\ref{muf}) and
(\ref{etaz}) satisfies the estimates
\begin{equation}
\sup _{x\in\left(x_0, \infty\right)}\left\|\langle x\rangle M_{-,21}(x ;z)\right\|_{L_z^2(\mathbb{R})} \leq c\left\|r_2\right\|_{H^1(\mathbb{R})},\label{xmurf}
\end{equation}
and
\begin{equation}
\sup _{x \in\left(x_0, \infty\right)}\left\|\langle x\rangle M_{+,12}(x ; z)\right\|_{L_z^2(\mathbb{R})} \leq c\left\|r_1\right\|_{H^1(\mathbb{R})},\label{xetar}
\end{equation}
where $c$ is a positive constant depending on $\left\|r_{1,2}\right\|_{L^{\infty}}$. Moreover, if $r_{1,2} \in H^{1}(\mathbb{R}) \cap
L^{2,1}(\mathbb{R})$, we also have
\begin{equation}
\sup _{x \in \mathbb{R}}\left\|\partial_x M_{-,21}(x ;z)\right\|_{L_{z}^2(\mathbb{R})} \leq c\left(\left\|r_1\right\|_{H^1(\mathbb{R}) \cap
L^{2,1}(\mathbb{R})}+\left\|r_2\right\|_{H^1(\mathbb{R}) \cap L^{2,1}(\mathbb{R})}\right),\label{pxmu}
\end{equation}
and
\begin{equation}
\sup _{x \in \mathbb{R}}\left\|\partial_x M_{+,12}(x ; z)\right\|_{L_z^2} \leq c\left(\left\|r_1\right\|_{H^1(\mathbb{R}) \cap
L^{2,1}(\mathbb{R})}+\left\|r_2\right\|_{H^1(\mathbb{R}) \cap L^{2,1}(\mathbb{R})}\right),\label{pxeta}
\end{equation}
where $c$ is another positive constant which depends on $\left\|r_{1,2}\right\|_{L^{\infty}(\mathbb{R})}$.
\end{proposition}
\begin{proof}
From the explicit expression (\ref{fep}), the first row vector of $T(x ;z)V_2(k)$ is given by
\begin{equation}
\left((2ik)^{-1}(M_{-,11}(x ;z)-1), \quad M_{+,12}(x ; z)-\bar{r}_1(z) \mathrm{e}^{-2iz
x}(M_{-,11}(x;z)-1)\right),
\end{equation}
and the second row vector of $T(x ;z)V_1(k)$ is given by
\begin{equation}
\left(M_{-,21}(x ; z), \quad 2 i k\left(M_{+,22}(x ; z)-1-\bar{r}_{+}(z) \mathrm{e}^{-2 i z x} M_{-,21}(x ; z)\right)\right).
\end{equation}
Using the bound (\ref{pn1}), we have for every $x \in \mathbb{R}$,
\begin{align}
&\left\|M_{-,21}(x ; z)\right\|_{L_{z}^{2}} \leq c\left\|\mathcal{P}^{-}\left(r_2(z) \mathrm{e}^{2 i z x}\right)\right\|_{L_{z}^{2}},\\
&\left\|(2ik)^{-1}\left(M_{-,11}(x ; z)-1\right)\right\|_{L_z^{2}} \leq c\left\|\mathcal{P}^{+}\left(\bar{r}_1(z)e^{-2
i z x}\right)\right\|_{L_z^2},\label{kmurp}\\
&\left\|2 i k\left(M_{-,22}(x ; z)-1-\bar{r}_1(z) \mathrm{e}^{-2 i z x} M_{-,21}(x ; z)\right)\right\|_{L_{z}^{2}} \leq c\left\|\mathcal{P}^{-}\left(r_2(z) \mathrm{e}^{2 i z x}\right)\right\|_{L_{z}^{2}},\nonumber\\
&\left\|M_{+,12}(x ;z)-\bar{r}_1(z) e^{-2 i z x}\left(M_{-,11}(x ; z)-1\right)\right\|_{L_z^2} \leq
c\left\|\mathcal{P}^{+}\left(\bar{r}_1(z) e^{-2 i z x}\right)\right\|_{L_z^2}. \nonumber
\end{align}
In the end, we find
\begin{equation}
\begin{aligned}
\left\|M_{+,12}(x ; z)\right\|_{L_{ z}^{2}} \leq c \left\|\mathcal{P}^+\left(\bar{r}_1(z) \mathrm{e}^{-2iz x}\right)\right\|_{L_{z}^2},\label{etakp}
\end{aligned}
\end{equation}
where the positive constant $c$ still has the only dependence on $\left\|r_{1,2}\right\|_{L^{\infty}}$. According to the bounds (\ref{xp}) and (\ref{etakp}),
we obtain the bound (\ref{xetar}).

For estimating the derivative of $M(x;z)$, under (\ref{pxm}), the first row vector of $\widetilde{F}(x;z)V_2(z)$ and the second row vector of $\widetilde{F}(x;z)V_1(z)$ belongs to $L_z^2(\mathbb{R})$ by $kr_{1,2}(z) \in L_z^{\infty}(\mathbb{R})$, owing to the bounds (\ref{lq1}) and (\ref{lq2}) in Proposition \ref{p6}, as well as the bounds (\ref{mjir}) and (\ref{kmurp}). Combining with the previous analysis, we obtain the bounds (\ref{pxmu})
and (\ref{pxeta}). So far we've proved the existence of the derivative of $M(x;z)$.
\end{proof}


\section{Reconstruction and estimates of the potential}
\label{sec:section5}
\hspace*{\parindent}
Recalling  reconstruction formulas obtained in  (\ref{u1}), (\ref{u2}) and (\ref{uxgj}), we have
\begin{align}
&u(x)e^{2i(c_-(x)+c)}=\lim _{k\rightarrow 0}(k^{-1}\psi^+(x;k))_{12},\label{ucc2}\\
&u(x)e^{-i(2c_-(x)+c)}=\lim _{k\rightarrow 0}(k^{-1}\psi^-(x;k))_{12},\label{ucc3}\\
&\partial_x\left(\bar{u}_x(x) e^{ic_\pm}\right)=2 i \lim _{|z| \rightarrow \infty} (z \Psi^{\pm}(x;z))_{21}, \label{ucc1}
\end{align}
 which gives the relation between the potential $u(x)$ and the  Jost functions.

Next, we draw the parallel between the properties of the potential $u$ recovered by the equations (\ref{ucc2})-(\ref{ucc3}) with the properties of the matrices $M_{\pm}(x;z)$, i.e., the solution to the RH problem \ref{rhp:31} satisfying the integral equations (\ref{mip}). Then, using the relation
\begin{equation}
\frac{\psi^\pm_2(x;k)}{2ik}=\Psi^\pm_2(x;z),
\end{equation}
 we  get the reconstruction of $u(x)$ and $u_x(x)$ as follows
\begin{align}
&u(x)e^{-i(2c_-(x)+c)}=2 i\lim _{z\rightarrow 0}M_{+,12}(x;z),\label{uxx3}\\
&u(x)e^{2i(c_-(x)+c)}=2 i\lim _{z\rightarrow 0}M_{-,12}(x;z).\label{uxx2}\\
&\partial_x\left(\bar{u}_x(x) e^{ic_+(x)}\right)=2 i e^{-ic_+(x)} \lim _{|z|\rightarrow \infty}z M_{\pm,21}(x;z),\label{uxx}
\end{align}


\subsection{Estimates on the positive half-line}
\hspace*{\parindent}
We  shall prove two important conclusions:
\begin{itemize}

\item[$\blacktriangleright$]  if $r_{1,2} \in H^1(\mathbb{R})$ and $z^{-2}r_{1,2}\in L^2(\mathbb{R})$,   the reconstruction formulas (\ref{uxx1}) and  (\ref{uc2}) recover $u$ in the class $H^{2,1}\left(\mathbb{R}^{+}\right)$.
\item[$\blacktriangleright$]  if $r_{1,2} \in \mathcal{H}$, then $u$ is in the class $H^3\left(\mathbb{R}^{+}\right)\cap H^{2,1}\left(\mathbb{R}^{+}\right)$.
    \end{itemize}

Since $r_{1,2} \in \mathcal{H}$, we have $R(x;\cdot)\in L^1(\mathbb{R}) \cap L^2(\mathbb{R})$ for every $x\in \mathbb{R}$.
Therefore,  using  the solution representation (\ref{see}),  we  rewrite the reconstruction formulas (\ref{uxx}) and (\ref{uxx2}) in the explicit form
\begin{align}
e^{ic_+(x)} \partial_{x}\left(\bar{u}_x(x) e^{ic_+(x)}\right)&=-\frac{1}{\pi}
 \int_{\mathbb{R}} r_2(z) e^{2 iz x}\left[M_{-,22}(x;z)+\bar{r}_1(z)e^{-2izx} M_{-,21}(x ;z)\right] \mathrm{d}z\nonumber \\
&=-\frac{1}{\pi} \int_{\mathbb{R}} r_2(z) e^{2isx} M_{+,22}(x;z)\mathrm{d}z,\label{uxx1}
\end{align}
and
\begin{equation}
e^{2i(c_-(x)+c)} u(x)=\frac{1}{\pi} \int_{\mathbb{R}} z^{-1}\bar{r}_1(z) e^{-2 i z x} M_{+,11}(x ;z) \mathrm{d}z.\label{uc2}
\end{equation}
\begin{proposition}
\label{l8}
Suppose $r_{1,2}\in \mathcal{H}$ such that the inequality (\ref{1_rk}) is satisfied, $u(x) \in
H^3\left(\mathbb{R}^{+}\right) \cap H^{2,1}\left(\mathbb{R}^+\right)$  admits the estimate
\begin{equation}
\|u\|_{H^3(\mathbb{R}^{+}) \cap H^{2,1}(\mathbb{R}^+)} \leq c(\|r_1\|_{\mathcal{W}(\mathbb{R})}+\|r_2\|_{\mathcal{W}(\mathbb{R})}),\label{ugj1}
\end{equation}
where $c$ is a positive constant which depends on $\left\|r_{1,2}\right\|_{H^1 \cap L^{2,1}}$.
\end{proposition}
\begin{proof}
We use reconstruction formulas (\ref{uxx1}) and  (\ref{uc2}) to obtain the estimate (\ref{ugj1}).  The reconstruction formula (\ref{uc2}) is rewritten as:
\begin{align}
&e^{2i(c_-(x)+c)} u(x) =\frac{1}{\pi} \int_{\mathbb{R}} z^{-1} \bar{r}_1(z)  e^{-2izx} \mathrm{d}z\nonumber \\
&+\frac{1}{\pi} \int_{\mathbb{R}}z^{-1} \bar{r}_1(z) e^{-2izx}\left(M_{-,11}(x;z)-1\right) \mathrm{d}z :=I_1(x)+I_2(x).\label{ugxj}
\end{align}


First,  $I_1(x)$  in $L^{2,1}(\mathbb{R})$ can be  controlled by $z^{-1} \bar{r}_1(z)$ in $H^{1}(\mathbb{R})$ since
\begin{align}
&\|I_1(x)\|_{L^{2,1} }  = \frac{1} {\pi} \| \widehat{(z^{-1} \bar{r}_1)}(2x)\|_{L^{2,1}_x} = \frac{1}{\pi} \|  z^{-1} \bar{r}_1(z)  \|_{ H^1_z}.  \label{ugxj3}
\end{align}

To estimate  $I_2(x)$ in $L^{2,1}\left(\mathbb{R}^+\right)$,  by using the inhomogeneous equation (\ref{mup1})  and integrating   by parts, we   obtain
\begin{equation}
I_2(x)=-\int_{\mathbb{R}} r_2(z) M_{+,12}(x;z) \mathrm{e}^{2isx} \mathcal{P}^+\left(z^{-1}\bar{r}_1(z) \mathrm{e}^{-2izx}\right)  \mathrm{d}z.
\end{equation}
With the estimates  (\ref{xp})   and   (\ref{xetar}),  by  Cauchy-Schwartz inequality, we have for every $x_0 \in
\mathbb{R}^+$,
\begin{align}
&\sup _{x\in\left(x_0, \infty\right)}\left|\langle x\rangle^2 I_2(x)\right| \leq\left\|r_2\right\|_{L^{\infty}} \sup _{x \in\left(x_0,
\infty\right)}\left\|\langle x\rangle M_{+,12}(x ; z)\right\|_{L_{z}^2}\nonumber\\
&\quad\quad\times \sup_{x\in\left(x_0, \infty\right)}\left\|\langle x\rangle
\mathcal{P}^+\left( z^{-1}  e^{-2iz x}\right)\right\|_{L_{z}^2(\mathbb{R})}\leq  c\left\|r_1\right\|_{H^1(\mathbb{R})}^2. \label{ugxj4}
\end{align}

By combining the estimates (\ref{ugxj3}) and (\ref{ugxj4})  with the triangle inequality,  we obtain
\begin{equation}
\|u\|_{L^{2,1}\left(\mathbb{R}^{+}\right)} \leq c \left(\left\|z^{-1}r_1(z) \right\|_{H^1(\mathbb{R})}+\left\|r_1(z)\right\|_{H^{1}(\mathbb{R})}^2\right).\label{u21}
\end{equation}

By using the reconstruction formula (\ref{uxx1}),  we have
\begin{align}
e^{ic_+(x)} \partial_{x}\left(\bar{u}_x(x) e^{ic_+(x)}\right)&=-\frac{1}{\pi}  \int_{\mathbb{R}} r_2(z) e^{2isx}\mathrm{d} z
-\frac{1}{\pi} \int_{\mathbb{R}} r_2(z) e^{2 iz x}\left[M_{+,22} -1\right] dz\nonumber\\
&:=I_3(x)+I_4(x).\label{uxjf}
\end{align}
Using similar procedures as above, we also derive
\begin{equation}
\left\|\partial_{x}\left(\bar{u}_x e^{ic_+(x)}\right)\right\|_{L^{2,1}(\mathbb{R}^+)} \leq c\left(\left\|r_2\right\|_{H^{1}(\mathbb{R})}+\left\|r_2\right\|_{H^{1}(\mathbb{R})}^2\right).\label{ux21}
\end{equation}

Next  we derivative  the equation  (\ref{uxjf})  and obtain
\begin{equation}
\begin{aligned}
\partial_x \left( e^{2i(c_-(x)+c)} u(x) \right)&=I_1'(x)+I_2'(x).\label{uf33}
\end{aligned}
\end{equation}
Direct calculation gives
$$I_1'(x)= -\frac{2}{\pi} \int_{\mathbb{R}}  \bar{r}_1(z)  e^{-2izx} dz = \widehat{  \bar{r}_1(z)} (2x), $$
which implies that
 \begin{equation}
  \|\langle x\rangle I_1^{\prime}(x)\|_{L^2(\mathbb{R}^+)}= 2\pi^{-1}  \|\langle x\rangle \widehat{  \bar{r}_1(z)} (2x)\|_{L^2(\mathbb{R})}\leq c   \|  r_1(z) \|_{H^1(\mathbb{R})},\label{u4h21}
\end{equation}

Differentiating  $I_2 (x)$ and  by  using  (\ref{mup1}), we  obtain
\begin{equation}
\begin{aligned}
I_2'(x)=&-2 i \int_{-\infty}^{\infty}  \bar{r}_1(s) e^{-2 iz x}\left(M_{-,11}(x ; z)-1\right) \mathrm{d}s+\int_{-\infty}^{\infty}
z^{-1}\bar{r}_1(z)e^{-2 i z x} \partial_{x} M_{-,11}(x ; z) \mathrm{d} z \\
=& 2 i \int_{-\infty}^{\infty} r_2(z) M_{+,12}(x; z) e^{2izx} \mathcal{P}^{+}\left( \bar{r}_1(z) e^{-2 i z x}\right)(z)
\mathrm{d} z \\
-&2 i \int_{-\infty}^{\infty}  zr_2(z) M_{+,12}(x ;z) e^{2 iz x} \mathcal{P}^+\left(z^{-1}\bar{r}_1(s) e^{-2 i z
x}\right)(z) \mathrm{d} z \\
-&\int_{-\infty}^{\infty} r_2(z) \partial_{x} M_{+,12}(x ; z) e^{2 i z x} \mathcal{P}^{+}\left(z^{-1}\bar{r}_1(z) e^{-2 i
z x}\right)(z) \mathrm{d} z,
\end{aligned} \nonumber
\end{equation}
Further with the estimation (\ref{pxmu}) and (\ref{pxeta}), it follows that
\begin{equation}
\begin{aligned}
&\sup _{x\in\left(x_0, \infty\right)}\left|\langle x\rangle I_2^{\prime}(x)\right|
\\ & \leq 2\left\|r_2\right\|_{L^{\infty}} \sup _{x \in\left(x_0,
\infty\right)}\left\| M_{+,12}(x ; z)\right\|_{L_{z}^2} \sup_{x\in\left(x_0, \infty\right)}\left\|\langle x\rangle
\mathcal{P}^+\left(z^{-1}\bar{r}_1(z) e^{-2iz x}\right)\right\|_{L_{z}^2} \\
&+2\left\|r_2\right\|_{L^{2,1}}\sup _{x \in\left(x_0, \infty\right)}\left\| M_{+,12}(x ; z)\right\|_{L_{z}^2}\sup_{x\in\left(x_0,
\infty\right)}\left\|\langle x\rangle \mathcal{P}^+\left(z^{-1}\bar{r}_1(z) e^{-2iz x}\right)\right\|_{L_{z}^{\infty}}\\
&+\left\|r_2\right\|_{L^{\infty}} \sup _{x \in\left(x_0, \infty\right)}\left\|\langle x\rangle \partial_x M_{+,12}(x ;
z)\right\|_{L_{z}^2} \sup_{x\in\left(x_0, \infty\right)}\left\| \mathcal{P}^+\left(z^{-1}\bar{r}_1(z) e^{-2iz
x}\right)\right\|_{L_{z}^2}\\
& \leq c\left\|r_2\right\|_{H^{1} \cap L^{2,1}}\left\| r_1\right\|_{H^{1} \cap L^{2,1}}\left(\left\|z^{-1}r_1\right\|_{H^{1} \cap
L^{2,1}}+\left\|r_2\right\|_{H^{1} \cap L^{2,1}}\right). \label{ewfk}
\end{aligned}
\end{equation}
Substituting  (\ref{u21}) and (\ref{ewfk}) into (\ref{uf33}),   we then obtain
\begin{equation}
\|u_x\|_{L^{2,1}(\mathbb{R}^{+})} \leq c \left(\left\|z^{-1}r_1\right\|_{H^{1} \cap
L^{2,1}}+\left\|r_1\right\|_{H^{1} \cap L^{2,1}}+\left\|r_2\right\|_{H^{1} \cap L^{2,1}}\right).\label{epe}
\end{equation}
Finally,  combining  (\ref{u21}),  (\ref{ux21}) and (\ref{epe}) yields
\begin{equation}
\|u\|_{H^{2,1}\left(\mathbb{R}^{+}\right)} \leq c\left(\left\|z^{-1}r_1\right\|_{H^1}+\left\|r_1\right\|_{H^1}+\left\|r_2\right\|_{H^1}\right),\label{uh21}
\end{equation}
where $c$ is a positive constant that depends on $\left\|r_{1,2}\right\|_{H^1\cap L^{2,1}}$.

In order to show $  u \in H^3\left(\mathbb{R}^{+}\right) $, we derivative  the equation  (\ref{uxjf})  and obtain
\begin{equation}
\begin{aligned}
\partial_x^2\left(  e^{ic_+(x)}u_x(x) \right)   & =I_3'(x)+I_4'(x),
\end{aligned}\label{uxjf3}
\end{equation}
in which
\begin{equation}
\begin{aligned}
I_3'(x) =4\pi^{-1} \int_{\mathbb{R}} z r_2(z) e^{2izx}d z = 4\pi^{-2}  \widehat{z r_2(z)}(-2x),
\end{aligned}
\end{equation}
which leads to
\begin{align}
\|I_3'(x)\|_{L^2_x}  = 4 \pi^{-2} \|  \widehat{z r_2(z)}(-2x)\|_{L^2_x}\leq  c  \left\|r_2\right\|_{H^{1} \cap L^{2,1}}. \label{uu1}
\end{align}

We take the derivative of $I_4$
\begin{equation}
\begin{aligned}
I_4^{\prime}(x)&=-4\pi^{-1} \int_{-\infty}^{\infty}r_2(z) e^{-2 i z x}\left(M_{+,22}(x;z)-1\right) dz\\
&-2i\pi^{-1}\int_{-\infty}^{\infty} s^{-1} r_2(z) e^{-2 i z x} \partial_{x} M_{+,22}(x ; z) dz. \nonumber
\end{aligned}
\end{equation}
The estimates for $I_4^{\prime}(x)$ can also be obtained accordingly
\begin{equation}
\begin{aligned}
\sup _{x\in\left(x_0, \infty\right)}\left|\langle x\rangle I_4^{\prime}(x)\right|\leq
c\left\|r_2\right\|_{H^{1} \cap L^{2,1}}\left\|r_1\right\|_{H^{1} \cap L^{2,1}}\left(\left\|z^{-1}r_1\right\|_{H^{1} \cap
L^{2,1}}+\left\|r_2\right\|_{H^{1} \cap L^{2,1}}\right),\nonumber
\end{aligned}
\end{equation}
by which we can further show that
\begin{align}
\|I_4'(x)\|_{L^2_x}   \leq  c \left(\left\|z^{-1}r_1\right\|_{H^{1} \cap
L^{2,1}}+\left\|r_2\right\|_{H^{1} \cap L^{2,1}}\right).\label{uu2}
\end{align}

From (\ref{uu1}) and (\ref{uu2}), we find that
\begin{align}
\|u_{xxx}(x)\|_{L^2_x}   \leq  c \left(\left\|z^{-1}r_1\right\|_{H^{1} \cap
L^{2,1}}+\left\|r_2\right\|_{H^{1} \cap L^{2,1}}\right),\label{uu3}
\end{align}
which together with (\ref{uh21}) yields  the  estimate
\begin{equation}
\|u\|_{H^3 \left(\mathbb{R}^{+}\right) \cap H^{2,1} \left(\mathbb{R}^{+}\right) } \leq c\left(\left\| r_1\right\|_{\mathcal{W}(\mathbb{R})}  +\left\|r_2\right\|_{\mathcal{W}(\mathbb{R})}\right).\nonumber
\end{equation}
\end{proof}

By Proposition \ref{l8}, we obtain the following Proposition:
\begin{proposition} Suppose $r_{1,2}\in \mathcal{W}(\mathbb{R})$ such that the inequality (\ref{1_rk}) is satisfied,
then  mapping
\begin{equation}
\mathcal{W}(\mathbb{R}) \ni\left(r_1, r_2\right) \mapsto u \in H^3\left(\mathbb{R}^{+}\right) \cap
H^{2,1}\left(\mathbb{R}^{+}\right),\label{rru1}
\end{equation}
is Lipschitz continuous.
\end{proposition}
\begin{proof}
Suppose $r_{1,2}, \tilde{r}_{1,2} \in H^1(\mathbb{R}) \cap L^{2,1}(\mathbb{R})$ satisfying $\left\|r_{1,2}\right\|_{H^{1} \cap
L^{2,1}},\left\|\tilde{r}_{1,2}\right\|_{H^{1} \cap L^{2,1}} \leq \rho$ for some $\rho>0$. Denote the corresponding potentials by $u$ and $\tilde{u}$
respectively. Then, there is a positive   constant $c$ such that
\begin{equation}
\begin{aligned}
&\|u-\tilde{u}\|_{H^3\left(\mathbb{R}^{+}\right) \cap H^{2,1}\left(\mathbb{R}^+\right)} \\
\leq &c \left( \left\|z^{-1}( r_1-\tilde{r}_1)\right\|_{H^1 \cap
L^{2,1}}+  \left\|r_1-\tilde{r}_1\right\|_{H^1 \cap
L^{2,1}}+\left\|r_2-\tilde{r}_2\right\|_{H^1 \cap L^{2,1}}\right).\label{rru2}
\end{aligned}
\end{equation}
The Lipschitz continuity here follows from the reconstruction formula (\ref{ugxj}) after repeating almost the same estimates as in Proposition \ref{l8}.
Using similar representation for $u$ and $\tilde{u}$, the Lipschitz continuity of (\ref{rru1}) is ensured with the bound (\ref{rru2}).
\end{proof}

\subsection{Estimates on the negative half-line}
\hspace*{\parindent}
 In order to obtain the estimate of potential $u(x)$ on negative real half-line, we
  need to rewrite the RH problem \ref{rhp:31} in an equivalent form. For this purpose,
we introduce a scalar RH problem
\begin{equation}
\begin{cases}
\delta_{+}(z) = (1+ \bar{r}_1(z) r_2(z))  \delta_{-}(z), \quad z \in \mathbb{R} \\
\delta_{\pm}(z) \rightarrow 1 \text { as } \quad|z| \rightarrow \infty.
\end{cases}\label{delta}
\end{equation}
Recall from (\ref{r11}) and (\ref{r22}) that
\begin{equation}
\begin{cases}1+\bar{r}_1(z) r_2(z)=1+|r(k)|^{2} \geq 1, & k \in \mathbb{R}^{+} \\ 1+\bar{r}_1(z)
r_2(z)=1-|r(k)|^{2} \geq c_{0}^{2}>0, & k \in \mathbb{R}^{-}\end{cases}
\end{equation}
where the latter inequality is due to (\ref{1_rk}).
\begin{proposition}
Suppose $r_{1,2} \in H^{1}(\mathbb{R}) \cap L^{2,1}(\mathbb{R})$ such that the inequality (\ref{1_rk}) is satisfied. The RH problem (\ref{delta})  exists unique
solution  $\delta_{\pm}(z)$    of the form
\begin{equation}
\delta(z)=  \exp [ {\mathcal{C}   \log \left(1+\bar{r}_1 r_2\right)}] , \quad z \in \mathbb{C}^{\pm},\label{dc}
\end{equation}
which   has  the limits
\begin{equation}
\delta_{\pm}(z)= \exp [ \mathcal{P}^{\pm} \log \left(1+\bar{r}_1 r_2\right)] , \quad z \in \mathbb{R}.\label{dp}
\end{equation}

\end{proposition}
\begin{proof}
As $r_{1,2} \in L_{z}^{2,1}(\mathbb{R}) \cap
L^{\infty}(\mathbb{R})$, we obtain $\bar{r}_1 r_2 \in L^{1}(\mathbb{R})$. It follows from the representation (\ref{rxdy}) as well as
from Propositions \ref{pp1} and \ref{pp2} that
\begin{equation}
\langle z\rangle|r(k)| \leq|r(k)|+\frac{1}{2}|k|\left|r_2(z)\right| \leq c, \quad z \in \mathbb{R},
\end{equation}
where $c$ is a positive constant. Therefore,
\begin{equation}
\log \left(1+|r(z)|^2\right) \leq \log \left(1+c^2\langle z\rangle^{-2}\right), \quad  z \in \mathbb{R}^+, \quad z \in \mathbb{R},
\end{equation}
so that $\log \left(1+\bar{r}_1 r_2\right) \in L^1\left(\mathbb{R}^+\right)$. In addition, with the inequality (\ref{1_rk}), we immediately obtain
\begin{equation}
\left|\log \left(1-|r(k)|^{2}\right)\right| \leq-\log \left(1-c^{2}\langle z\rangle^{-2}\right), \quad z \in \mathbb{R}^-, \quad z \in
\mathbb{R}.\label{log}
\end{equation}
It follows that $\log \left(1+\bar{r}_1 r_2\right) \in L^{1}\left(\mathbb{R}^{-}\right)$. Consequently, we have $\log \left(1+\bar{r}_1 r_2\right) \in L^{1}(\mathbb{R})$. Also from (\ref{log}), another conclusion is $\log \left(1+\bar{r}_1
r_2\right) \in L^{\infty}(\mathbb{R})$.

With the help of the H\"{o}lder inequality, we derive $\log \left(1+\bar{r}_1 r_2\right) \in L^{2}(\mathbb{R})$. By
Proposition \ref{p3} with $p=2$, the expression (\ref{dc}) defines unique analytic functions in $\mathbb{C}^{\pm}$, which recover the limits (\ref{dp})
and the limits when $|z| \rightarrow \infty$: $\lim _{|z| \rightarrow \infty} \delta_{\pm}(z)=1$. In the end, with $\mathcal{P}^{+}-\mathcal{P}^{-}=I$, we know
\begin{equation}
\delta_{+}(z) \delta_{-}^{-1}(z)=\mathrm{e}^{\log \left(1+\bar{r}_1(z) r_2(z)\right)}=1+\bar{r}_1(z) r_2(z),
\quad z \in \mathbb{R}.
\end{equation}
In a word, $\delta_{\pm}$ given by (\ref{dc}) satisfy the scalar RH problem (\ref{delta}).
\end{proof}
\begin{proposition}
\label{p9}
Suppose $r_{1,2} \in \mathcal{W}(\mathbb{R})$ such that the inequality (\ref{1_rk}) is satisfied, then  $\delta_+ \delta_- r_{1,2}
\in \mathcal{W}(\mathbb{R})$.
\end{proposition}
\begin{proof}
By Sokhotski-Plemelj theorem, we have the relations
\begin{equation}
\mathcal{P}^{\pm}(h)(z)=\pm \frac{1}{2} h(z)-\frac{i}{2} \mathcal{H}(h)(z) \quad z \in \mathbb{R}.
\end{equation}
We note $\mathcal{P}^++\mathcal{P}^-=-i \mathcal{H}$ , where $\mathcal{H}$ is the Hilbert transform and have
\begin{equation}
\delta_{+} \delta_{-}=\mathrm{e}^{-\mathrm{i} \mathcal{H} \log \left(1+\bar{r}_1 r_2\right)} .
\end{equation}
As $\log \left(1+\bar{r}_1 r_2\right) \in L^{2}(\mathbb{R})$, we obtain $\mathcal{H} \log \left(1+\bar{r}_1 r_2\right) \in L^2(\mathbb{R})$ being a
real-valued function. Thus, as $\left|\delta_{+}(z) \delta_{-}(z)\right|=1$ for almost every $z \in \mathbb{R}$, $\delta_+
\delta_{-} r_{1,2} \in L^{2,1}(\mathbb{R})$ follows from $r_{1,2} \in L^{2,1}(\mathbb{R})$.

A remaining task is to show $\partial_{z} \delta_+ \delta_- r_{1,2} \in L^2(\mathbb{R})$, i.e.,
$\partial_{z} \mathcal{H} \log (1+ \bar{r}_1 r_2) \in L^{2}(\mathbb{R})$. Thanks for the Parseval's identity and the proved fact $\|\mathcal{H} f\|_{L^{2}}=\|f\|_{L^{2}}$ for every $f \in L^{2}(\mathbb{R})$, it is obvious that
\begin{equation}
\left\|\partial_{z} \mathcal{H} \log \left(1+\bar{r}_1 r_2\right)\right\|_{L^2}=\left\|\partial_{z} \log \left(1+\bar{r}_1
r_2\right)\right\|_{L^2} .
\end{equation}
The right-hand side is bounded as $\partial_z \log \left(1+\bar{r}_1 r_2\right)=\frac{\partial_z\left(\bar{r}_1+r_2\right)}{1+\bar{r}_1
r_2} \in L^2(\mathbb{R})$ under the premise of this proposition. $\partial_z \delta_+ \delta_- r_{1,2} \in L^2(\mathbb{R})$ is proved to be sufficient.
\end{proof}

In the second step, we decompose the jump matrix $R(x;z)$ in an equivalent form:
\begin{equation}
\begin{aligned}
&\begin{pmatrix}
\delta_-(z) & 0 \\
0 & \delta_{-}^{-1}(z)
\end{pmatrix}(I+R(x;z))\begin{pmatrix}
\delta_+^{-1}(z) & 0 \\
0 & \delta_+(z)
\end{pmatrix}\\
&=\begin{pmatrix}
1 & \delta_-(z) \delta_+(z) \bar{r}_1(z) e^{-2 i z x} \\
\bar{\delta}_+(z) \bar{\delta}_-(z) r_2(z) e^{2 i z x} & 1+\bar{r}_1(z) r_2(z)
\end{pmatrix},
\end{aligned}
\end{equation}
where $\delta_-^{-1} \delta_+^{-1}=\overline{\delta_- \delta_+}$ is used.

 We
  define a new matrix
\begin{equation}
\hat{R}_{\delta}(x;z):=\begin{pmatrix}
0 & \bar{r}_{\delta,1}(z) e^{-2ixz} \\
r_{\delta,2}(z) e^{2ixz} & \bar{r}_{\delta,1}(z) r_{\delta,2}(z)
\end{pmatrix},
\end{equation}
where
\begin{equation}
r_{\delta,j}(z):=\bar{\delta}_+(z) \bar{\delta}_-(z) r_{j}(z), \ \ j=1,2.
\end{equation}
By Proposition \ref{p9}, we have $r_{\delta,j} \in H^1(\mathbb{R}) \cap L^{2,1}(\mathbb{R})$ similarly to the scattering data $r_{1,2}$.

With the functions $M_{\pm}(x;z)$ and $\delta_{\pm}(z)$, we define the functions
\begin{equation}
M_{\delta,\pm }(x;z):=M_{\pm}(x; z) \delta_{\pm}(z)^{-\delta_3},
\end{equation}
which satisfies the RH problem
\begin{rhp}
\label{rhp:5_1}
Find a  matrix  function $ M_{\delta}(x;z)$  with the following properties:
\begin{itemize}

\item[$\blacktriangleright$]  \emph{Analyticity}: $M_{ \delta,\pm }(x;z)$ are analytic functions in $\mathbb{C}^{\pm}$ ;
\item[$\blacktriangleright$]\emph{Asymptotic conditions}:
\begin{equation}
M_{\delta,\pm}(x;z) \rightarrow I, \quad \text {as} \quad |z| \rightarrow 0;
\end{equation}
\item[$\blacktriangleright$]\emph{Jump condition}: $M_{\delta,\pm }(x;z)$ satisfies the jump condition
\begin{equation}
M_{\delta,+}(x ;z)=M_{\delta,-}(x ;z)(I+ \hat{R}_{\delta}(x;z)), \quad z \in \mathbb{R}.\label{mdrhp}
\end{equation}
\end{itemize}

\end{rhp}
The above RH problem is transformed from the previous RH problem \ref{rhp:31}.  As the analysis of   Proposition \ref{lm3} and \ref{ll9},  the RH
problem \ref{rhp:5_1} admits a unique solution  given by
\begin{equation}
M_\delta(x ;z)=I+\mathcal{C}\left(M_{\delta,-} (x ; \cdot) \hat{R}_\delta(x;\cdot)\right)(z), \quad z \in
\mathbb{C}^{\pm}.\label{mdc}
\end{equation}
We denote the column vectors of $M_{\delta,\pm}$ by $M_{\delta,\pm}=\left[M_{\delta,\pm, 1}, M_{\delta,\pm, 2}\right]$.
Since $r_{\delta, j} \in H^1(\mathbb{R}) \cap L^{2,1}(\mathbb{R})$, we have $\hat{R}_{\delta}(x ; \cdot) \in L^1(\mathbb{R}) \cap L^2(\mathbb{R})$
for every $x \in \mathbb{R}$.
In consequence, the reconstruction formulas (\ref{uxx}) and (\ref{uxx3}) become
\begin{equation}
\begin{aligned}
u(x)e^{-i(2c_-(x)+c)}&=\frac{1}{\pi} \int_{\mathbb{R}} z^{-1}\bar{r}_{\delta,1}(z) e^{-2izx} M_{\delta,+,11}(x;z) dz,\label{ur5}
\end{aligned}
\end{equation}
and
\begin{equation}
e^{ic_+(x)} \partial_{x}\left(\bar{u}_x(x) e^{-ic_-(x)}\right)=-\frac{1}{\pi} \int_{\mathbb{R}} r_{\delta,2}(z) \mathrm{e}^{2 i z x} M_{\delta,-,22}(x ;z) dz.\label{ur6}
\end{equation}

From (\ref{mdc}), we can obtain the system of integral equations
for vectors $M_{\delta,+,1}$ and $M_{\delta,-,2}$ by projecting   representation
\begin{equation}
\begin{aligned}
&M_{\delta,+,1}(x ; z)=e_1+\mathcal{P}^+\left(r_{\delta,-,2} e^{2iz x} M_{\delta,-,2}(x; \cdot)\right) , \\
&M_{\delta,-,2}(x ; z)=e_2+\mathcal{P}^{-}\left(\bar{r}_{\delta,-,1} e^{-2iz x} M_{\delta,+,1}(x; \cdot)\right) .
\end{aligned}
\end{equation}
We formulate the above integral equation in a unified form as
\begin{equation}
T_{\delta}-\mathcal{P}^-\left(T_{\delta} R_{\delta}\right)=F_{\delta},
\end{equation}
where
\begin{equation}
T_{\delta}(x;z):=\left(M_{\delta,+,1}(x ;z)-e_1, M_{\delta,-,2}(x; z)-e_2\right)\begin{pmatrix}
1 & 0 \\
-r_{\delta,2}(z) e^{2izx} & 1
\end{pmatrix}
\end{equation}
and
\begin{equation}
F_{\delta}(x;z):=\left(\mathcal{P}^+\left(r_{\delta,2}(z) e^{2izx}\right)e_2,  \mathcal{P}^-\left(\bar{r}_{\delta,1}(z) e^{-2izx}\right)e_1\right).
\end{equation}

In order to the estimate on the  potential $u(x)$ on negative real axis   by using  reconstruction formulas (\ref{ur5}) and (\ref{ur6}),
  we  give similar results to Proposition  \ref{prop7} and  Proposition \ref{p6}.
\begin{proposition}
\label{pro3p23}
 If $f(z)\in L^2(\mathbb{R})$,  then we have
\begin{align}
&\mathcal{P}^+\left(f(z) e^{2iz x}\right) =-\int_{-\infty}^{2x}\widehat{f}(\xi)e^{-iz(\xi-2x ) }\mathrm{d}\xi,\\
&\mathcal{P}^-\left(f(z) e^{-2iz x}\right) =\int_{-\infty}^{ 2x} \widehat{f}(\xi)e^{ iz(\xi-2x ) }\mathrm{d}\xi.
\end{align}
\end{proposition}
\begin{proof}  The proof  is similar with  that of  Proposition \ref {prop7}.
\end{proof}

By using above  Proposition \ref{pro3p23}, we can further  show  the following estimates on negative half-line.
\begin{proposition}
\label{pro3p24}
For every $x_0 \in \mathbb{R}^-$ and every $r_{1,2}(z) \in H^1(\mathbb{R})$, we have
\begin{align}
&\sup _{x \in\left( -\infty,  x_{0} \right)}\left\| \langle x\rangle \mathcal{P}^{+}\left(z^{-i}\bar{r}_1(z) e^{ 2i z x}\right)\right\|_{L_z^2}
\leq\left\|z^{-i}\bar r_1(z)\right\|_{H^1},\quad i=0,1,\label{5xp}\\
&\sup _{x \in \left( -\infty,  x_{0} \right)}\left\|\langle x\rangle \mathcal{P}^-\left(r_2(z) e^{-2iz x}\right)\right\|_{L_{z}^2}
\leq\left\|r_2(z)\right\|_{H^1}.\label{5xpfr}
\end{align}
In addition, we have
\begin{align}
&\sup _{x \in \mathbb{R}}\left\|\mathcal{P}^+\left(z^{-i}\bar{r}_1(z)e^{ 2iz x}\right)\right\|_{L_z^\infty} \leq
\frac{1}{\sqrt{2}}\left\|z^{-i}r_1(z)\right\|_{H^1},\quad i=0,1,\label{5pzr}\\
&\sup_{x \in \mathbb{R}}\left\|\mathcal{P}^-\left(r_2(z) \mathrm{e}^{-2iz x}\right)\right\|_{L_{z}^{\infty}} \leq
\frac{1}{\sqrt{2}}\left\|r_2(z)\right\|_{H^{1}}.\label{5pfr}
\end{align}
Furthermore, if $r_{1,2} \in L^{2,1}(\mathbb{R})$, then
\begin{align}
&\sup _{x \in \mathbb{R}}\left\|\mathcal{P}^+\left(z \bar{r}_1(z) e^{ 2 i z x}\right)\right\|_{L_{z}^{2}} \leq\left\|
r_1(z)\right\|_{L_{z}^{2,1}}, \label{5lq1}\\
&\sup _{x \in \mathbb{R}}\left\|\mathcal{P}^{-}\left(z r_2(z) e^{-2 i z x}\right)\right\|_{L_{z}^{2}} \leq\left\|
r_2(z)\right\|_{L_{z}^{2,2}}.\label{5lq2}
\end{align}
\end{proposition}
\begin{proof}  The proof  is similar with  that of  Proposition \ref{p6}.
\end{proof}
By using  Proposition \ref{pro3p23} and Proposition \ref{pro3p24},  in same way to the Section \ref{sec:section4},
we  obtain estimates of  the potential $u(x)$ on  the negative half-line and Lipschitz continuity.

\begin{proposition}
Let $r_{1,2} \in \mathcal{W}(\mathbb{R})$ such that the inequality (\ref{1_rk}) is satisfied. Then, $u(x) \in
H^{3}\left(\mathbb{R}^{-}\right) \cap H^{2,1}\left(\mathbb{R}^{-}\right)$ satisfies the bound
\begin{equation}
\|u\|_{H^3\left(\mathbb{R}^-\right) \cap H^{2,1}\left(\mathbb{R}^-\right)} \leq c\left(\left\|z^{-1} r_1\right\|_{H^1 \cap
L^{2,1}}+\left\|r_1\right\|_{H^1 \cap
L^{2,1}}+\left\|r_2\right\|_{H^1 \cap L^{2,1}}\right),
\end{equation}
where $c$ is a positive constant that depends on $\left\|r_{1,2}\right\|_{H^1 \cap L^{2,1}}$.
Moreover   the mapping
\begin{equation}
\mathcal{W}(\mathbb{R}) \ni\left(r_1, r_2\right) \mapsto u \in H^{3}\left(\mathbb{R}^{-}\right) \cap
H^{2,1}\left(\mathbb{R}^{-}\right)
\end{equation}
is Lipschitz continuous.
\end{proposition}
\begin{proof}
Let $r_{1,2}, \tilde{r}_{1,2} \in H^1(\mathbb{R}) \cap L^{2,1}(\mathbb{R})$ satisfy $\left\|r_{1,2}\right\|_{H^{1} \cap
L^{2,1}},\left\|\tilde{r}_{1,2}\right\|_{H^{1} \cap L^{2,1}} \leq \rho$ for some $\rho>0$. Denote the corresponding potentials by $u$ and $\tilde{u}$
respectively. Then, there is a positive $\rho$-dependent constant $c$ such that
\begin{equation}
\|u-\tilde{u}\|_{H^3\left(\mathbb{R}^-\right) \cap H^{2,1}\left(\mathbb{R}^-\right)} \leq c\left(\left\|  r_1-\tilde{r}_1 \right\|_{\mathcal{W}(\mathbb{R})}+
 \left\|r_2-\tilde{r}_2\right\|_{\mathcal{W}(\mathbb{R})}\right).\nonumber
\end{equation}
\end{proof}

\section{Existence of global solutions to the FL equation}
\label{sec:section6}
\hspace*{\parindent}
By transforming the space part of the original Lax pair, we establish  a series of  RH problems  and   the
reconstruction formulas for  the potential $u(x)$  where time $t$ is just  considered as a parameter.   We  need to incorporate the time evolution
to get  the solution $u(x,t)$ to the FL equation.

\subsection{ Time evolution from reflection coefficients to RH problem }
\hspace*{\parindent}
For  the  partial spectral problem (\ref{cslp}) and the time spectral  problem (\ref{cslp2}),  we define the fundamental solutions
\begin{align}
&\psi_1(x,t;k)=\psi_1(x;k)  e^{-i \eta^2 t},\\
&\psi_2(x,t;k)=\psi_2(x;k)  e^{i \eta^2 t},
\end{align}
where $\eta=\sqrt{\alpha}(k-\frac{\beta}{2k})$.
From  the Voterra integral equation  (\ref{fi5}),
 the bounded Jost functions $\psi_1(x,t;k)$ and $\psi_2(x,t;k)$ have
the same analytical  property  in the $k$ plane and satisfy the same boundary conditions
\begin{align}
\psi_1(x,t;k) \rightarrow e_1,\quad \text { as } \quad x \rightarrow \pm \infty, \\
\psi_2(x,t;k) \rightarrow e_2,\quad \text { as } \quad x \rightarrow \pm \infty,
\end{align}
for every $t\in[0,T]$. From the linear independence of two solutions to  Lax pair  (\ref{cslp}), the columns of the  Jost functions $\psi^\pm(x,t;k)$ satisfy the
scattering relation
\begin{equation}
\psi_1^-(x,t;k)=a(k)\psi_1^+(x,t;k)+b(k) e^{2ik^2x+2i \eta^2t} \psi_2^+(x,t;k), \quad k \in \mathbb{R} \cup i
\mathbb{R},\label{p1p2}
\end{equation}
where the scattering coefficients $a(t;k)$ and $b(t;k)$ are  dependent of $ t $.  By using the time spectral  problem (\ref{cslp2}),
we can  find   the time evolution relations  of $a(t;k)$ and $b(t;k)$
$$ a(t;k)=a(k), \ \ b(t;k)= b(k) e^{-2 i \eta^2 t}, $$
which together with the definition  (\ref{rxdy}), we obtain  time evolution on the scattering data $r_{1,2}(z)$ as follows
\begin{align}
&r_1(t;z)=-\frac{b(k)}{2ik a(k)}e^{-2i\eta^2 t} = r_{1 }( z)e^{-2i\eta^2 t},\label{ryh1} \\
& r_2(t;z)=\frac{2ikb(k)}{a(k)}e^{-2i\eta^2 t}= r_{2 }( z)e^{-2i\eta^2 t}. \label{ryh2}
\end{align}

\begin{proposition}
If $r_{1,2}(z ) \in \mathcal{W}(\mathbb{R})$, then for a arbitrary  fixed  $T>0$  and  every    $t \in[0, T]$,
\begin{align}
&r_{1,2}(t ; z) \in  \mathcal{W}(\mathbb{R}).\label{wege}
\end{align}

\end{proposition}
\begin{proof}
By using (\ref{ryh1})-(\ref{ryh2}), direction calculation shows that
\begin{align}
&\left\|r_{1,2}(t;z)\right\|_{L^{2,1}}=\left\|r_{1,2}(z)\right\|_{L^{2,1}}, \label{wege0}\\
&\left\|z^{-2}r_{1,2}(t;z)\right\|_{L^{2,1}}=\left\|z^{-2}r_{1,2}(z)\right\|_{L^{2,1}}.\label{wege01}
\end{align}
Further by differentiation to  (\ref{ryh1})-(\ref{ryh2}) and taking   $L^2$-norm,  for $t \in[0, T]$, then  by Proposition \ref{l13}, we have
\begin{align}
& \|\partial_{z} r_{1,2}(t;
z) \|_{L^2} =   \big\| \partial_{z} r_{1,2}( z)e^{-2i\eta^2 t} - 2i\alpha t (1 -\frac{\beta^2}{4}z^{-2}) r_{1,2}(t;z) e^{-2i\eta^2 t}\big\|_{L^2} \nonumber\\
&\leq   2\alpha T \left\|  r_{1,2}( z)\right\|_{L^2}+\left\|\partial_{z} r_{1,2}( z)\right\|_{L^2}  +  \frac{1}{2} \alpha\beta^2 T  \| z^{-2} r_{1,2}( z) \|_{L^2}.\label{wege1}
\end{align}
In a similar way, by using Proposition \ref{l12},  we can show that
\begin{align}
&   \| z^{-2} r_{2 }( z) \|_{L^2(\mathbb{R})} \leq
 2 \|r_{ 2}(  z) \|_{L^{\infty}}  \|   z^{-2}  \|_{L^2(\delta,+\infty)} +   a_0^{-1}    \| z^{-1} k^{-1} b(k)   \|_{L^2( ( 0,\delta))}.\label{wege3}
\end{align}
Substituting  (\ref{wege2}) and (\ref{wege3})  into \ref{wege2})  yields
\begin{align}
& \|\partial_{z} r_{1,2}(t;
z) \|_{L^2}
 \leq  ( 2 \alpha T +1 ) \left\|  r_{1,2}( z)\right\|_{H^1}   +   2 \alpha\beta^2 T  \|r_{ 2}(  z) \|_{L^{\infty}}  \|   z^{-2}  \|_{L^2(\delta,+\infty)} \nonumber\\
 & + \frac{1}{2} T \alpha\beta^2  a_0^{-1} ( | z^{-2} k^{-1} b(k)   \|_{L^2( ( 0,\delta))}+ \| z^{-1} k^{-1} b(k)   \|_{L^2( ( 0,\delta))}).   \label{wege4}
\end{align}
Finally  (\ref{wege0}), (\ref{wege01}) and (\ref{wege4}) leads to the result   (\ref{wege}).
\end{proof}

\subsection{ The  local solution and  global  solution    }
\hspace*{\parindent}
For  $r_{1,2}(t ; z) \in \mathcal{W}(\mathbb{R})$, $t\in[0, T]$,    the constraint
(\ref{1_rk}) and the relation (\ref{rfrz}) is still true for every $t \in[0, T]$.

\begin{theorem}
For every $u_{0} (x) \in H^3(\mathbb{R}) \cap H^{2,1}(\mathbb{R})$   such that the linear equation (\ref{cslp}) admits no eigenvalues or resonances,
 there exists a unique  local   solution  to the Cauchy problem
 $$u(x,t) \in C( [0,T],  H^3(\mathbb{R}) \cap H^{2,1}(\mathbb{R})), \ t \in[0, T].$$
  Furthermore, the map
\begin{equation}
H^{3}(\mathbb{R}) \cap H^{2,1}(\mathbb{R}) \ni u_{0}(x)  \mapsto u (t,x) \in C\left([0, T], H^{3}(\mathbb{R}) \cap H^{2,1}(\mathbb{R})\right)\label{jbj}
\end{equation}
is Lipschitz continuous.
\label{thm:1}
\end{theorem}
\begin{proof}
The potential $u(t,x)$ is recovered from the scattering data $r_{1,2}(t ; z)$  with the inverse scattering transform shown in Section \ref{sec:section3}. Combining
the whole analysis in the previous section, for every $t \in[0,T)$ we finally prove that
\begin{equation}
\begin{aligned}
\|u(x,t)\|_{H^{3}(\mathbb{R}) \cap H^{2,1}(\mathbb{R})} & \leq c  \left(\left\|r_1(t ; z)\right\|_{H^{1}(\mathbb{R}) \cap L^{2,1}(\mathbb{R})}+\left\|r_2(t;z)\right\|_{H^{1}(\mathbb{R}) \cap
L^{2,1}(\mathbb{R})}\right) \\
& \leq c \left(\left\|r_1(z)\right\|_{H^1(\mathbb{R}) \cap L^{2,1}(\mathbb{R})}+\left\|r_2(z)\right\|_{H^{1}(\mathbb{R}) \cap L^{2,1}(\mathbb{R})}\right)\\
& \leq c (T)\left\|u_0(x)\right\|_{H^3(\mathbb{R}) \cap H^{2,1}(\mathbb{R})}.\label{ubkz}
\end{aligned}
\end{equation}
It is a strong indication that there exists a unique local solution
$$u(x,t) \in C( [0,T],  H^3(\mathbb{R}) \cap H^{2,1}(\mathbb{R}))\ t \in[0, T]$$
to  the Cauchy problem (\ref{cs})-(\ref{cs1}).
Combining with the Lipschitz continuity from $u_0(x)$ to $r_{1,2}(t;z)$ in Proposition \ref{l6} and
\begin{equation}
\begin{aligned}\nonumber
&\|u(x,t_2 )-u(x,t_1  )\|_{H^3\cap H^{2,1}} \leq c\|r_{1,2}( z)(e^{2i\eta^2 t_2}-e^{2i\eta^2  t_1})\|_{H^1\cap H^{2,1}}\\
&\leq c|t_2-t_1| \|r_{1,2}(z)\|_{H^1\cap L^{2,1}},
\end{aligned}
\end{equation}
we prove the  mapping
\begin{equation}
H^{3}(\mathbb{R}) \cap H^{2,1}(\mathbb{R}) \ni u_{0} \mapsto u \in C\left([0, T], H^{3}(\mathbb{R}) \cap H^{2,1}(\mathbb{R})\right) \nonumber
\end{equation}
is Lipschitz continuous.
\end{proof}

The following theorem show that there is a global solution in $H^3(\mathbb{R}) \cap H^{2,1}(\mathbb{R})$:
\begin{theorem}
For every $u_{0}(x) \in H^3(\mathbb{R}) \cap H^{2,1}(\mathbb{R})$   such that the linear equation
 (\ref{cslp}) admits no eigenvalues or resonances, there exists a unique global solution to  the Cauchy problem (\ref{cs})-(\ref{cs1})
 $$u(x,t) \in C\left([0, \infty); H^3(\mathbb{R}) \cap H^{2,1}(\mathbb{R})\right). $$
 Furthermore, the map
\begin{equation}
H^3(\mathbb{R}) \cap H^{2,1}(\mathbb{R}) \ni u_{0} \mapsto u \in C\left([0, \infty); H^3(\mathbb{R}) \cap H^{2,1}(\mathbb{R})\right)
\end{equation}
is Lipschitz continuous.
\label{thm:2}
\end{theorem}
\begin{proof}
Based on Theorem \ref{thm:1},
as $c(T)$ depends on $T$ and $\|r_{1,2}(0,z)\|_{H^1\cap{L^{2,1}}}$ grows at most in a polynomial order with respect to $\|u(x,0)\|_{H^3\cap H^{2,1}}$, we hence reach the conclusion that the local solution exists in   $C\left([0, T], H^{3}(\mathbb{R}) \cap H^{2,1}(\mathbb{R})\right)$  for an arbitrary fixed $T$.   Then   the global existence of the solution can be  asserted if
 $T=\infty$.  Suppose the maximal time in which  the local solution exists  is $T_\text{max}$.

 If $T_\text{max}=\infty$, then the local solution is global.
  Otherwise, the local solution exists  in the  finite   interval  $[0, T_\text{max}]$ or  $[0, T_\text{max})$.

If  the local solution exists in the closed interval $[0, T_\text{max}]$,   we can  use $u(x,T_\text{max}) \in H^3(\mathbb{R})\cap H^{2,1}(\mathbb{R})$ as the new initial data,
 according to  the  results derived in the previous sections,  there exists $T_1 > 0$ such that the following local solution
        $$
        u_1(x,t) \in C([T_\text{max}, T_\text{max}+T_1], H^3(\mathbb{R})\cap H^{2,1}(\mathbb{R}))
        $$
        exists in the time interval $[T_\text{max}, T_\text{max}+T_1]$.
        Therefore if  we   reconstruct the   function
        \begin{equation}
        \tilde{u}(x,t) = \begin{cases}
        u(x,t), &t\in[0,T_\text{max}] \\
        u_1(x,t), & t \in [T_\text{max},T_\text{max}+T_1],
        \end{cases}
        \end{equation}
        which  then  is the solution to the Cauchy problem (\ref{cs})-(\ref{cs1}) in the time interval $[0, T_\text{max}+T_1]$.
         This contradicts with the definition of $T_\text{max}$.

        If  the local solution exists in the open interval $[0,T_\text{max})$, then by priori  estimation in (\ref{ubkz}),     we have the following estimation \begin{equation}
            \|u(x,t)\|_{H^3(\mathbb{R})\cap H^{2,1}(\mathbb{R})} \le c(T_\text{max})\|u(x,0)\|_{H^3(\mathbb{R})\cap H^{2,1}(\mathbb{R})}, \quad{t\in [0,T_\text{max})},
            \label{eq:1325}
            \end{equation}
           Due to the continuity of $u(x,t)$ with respect to the time $t$, the limit of $u(x,t)$ when $t$ converges to $T_\text{max}$ exists.
            Taking the limit by $t\to{T}_\text{max}$ in (\ref{eq:1325}), we have
            \begin{equation}
            \|u_{T_\text{max}}\|_{H^2(\mathbb{R} )\cap H^{1,1}(\mathbb{R} )} \le c(T_\text{max})\|u(x,0)\|_{H^2(\mathbb{R} )\cap H^{1,1}(\mathbb{R} )}, \quad{t\in [0,T_\text{max})},
            \end{equation}
            where $u_{T_\text{max}} = \lim_{t\to{T}_\text{max}} u(x,t)$. By defining the new function
        \begin{equation}
        \tilde{u}(x,t) = \begin{cases}
        u(x,t), \quad t\in[0,T_\text{max}), \\
        u_{T_\text{max}}, \quad t = T_\text{max},
        \end{cases}
        \end{equation}
        which  also  is   the solution to the Cauchy problem in (\ref{cs})-(\ref{cs1}) in the time interval $[0,T_\text{max}]$.
        This  contradicts with the premise that $[0, T_\text{max})$ is the maximal open interval.

 Based on the discussion above,  we conclude that  the global   solution exists  in   $u(x,t) \in C\left([0, \infty), H^3(\mathbb{R}) \cap H^{2,1}(\mathbb{R})\right)$, which finally yields the proof of Theorem \ref{thm:2}.
\end{proof}

\begin{remark}
In our previuos work \cite{cqy}, we ever obtained   the long-time asymptotic behavior  of  the  FL
 equation with generic initial data in a Sobolev space $ H^{3,3}(\mathbb{R})$. In present work, we obtain the existence  of global solutions to the Cauchy problem  (\ref{cs})-(\ref{cs1}) of the   FL  equation on the line for the  initial data    $u_0(x)\in H^3(\mathbb{R})\cap H^{2,1}(\mathbb{R})$. Due to  $H^{3,3}(\mathbb{R})  \hookrightarrow H^3 (\mathbb{R}) \cap H^{2,1}(\mathbb{R}) $,   our present work    ensures   the strictness of our previuos work on  long-time asymptotic behavior  of  the  FL equation \cite{cqy}.  Based on the results obtained this paper together with Backlund transformation,  we  will  futher  prove the global existence for the FL  equation in the case when the initial datum includes a finite number of solitons in our future work.
\end{remark}

\noindent\textbf{Acknowledgements}

This work is supported by  the National Natural Science
Foundation of China (Grant No. 11671095, 51879045).\vspace{2mm}

\noindent\textbf{Data Availability Statements}

The data that supports the findings of this study are available within the article.\vspace{2mm}

\noindent{\bf Conflict of Interest}

The authors have no conflicts to disclose.

\end{document}